\documentclass[10pt,reqno]{amsart}
\usepackage{hyperref}
\usepackage{latexsym}
\usepackage{amsmath}
\usepackage{enumerate}
\usepackage{amsfonts}
\usepackage{amssymb}
\usepackage{latexsym}
\usepackage{fixmath}
\usepackage[usenames,dvipsnames]{color}
\usepackage{mathdots}
\usepackage{multicol}
\usepackage{graphicx}
\usepackage{xcolor}
\usepackage{subcaption} 
\usepackage[boxed, linesnumbered, noend, noline]{algorithm2e}
\usepackage{comment}
\usepackage{float}
\usepackage{mathtools}
\usepackage{amsmath}
\usepackage{dsfont}

\usepackage{todonotes}

\usepackage[T1]{fontenc}
\usepackage{fourier}
\usepackage{bbm}
\usepackage{color}
\usepackage{fullpage}

\usepackage{cleveref}

\numberwithin{equation}{section}

\newcommand\T{\vec T}

\newcommand\Var{\mathrm{Var}}

\newcommand{\Erdos}{Erd\"os}
\newcommand{\Renyi}{R\'enyi}

\renewcommand\Pr{\pr}

\newtheorem{definition}{Definition}[section]
\newtheorem{claim}[definition]{Claim}

\newtheorem{remark}[definition]{Remark}
\newtheorem{theorem}[definition]{Theorem}
\newtheorem{lemma}[definition]{Lemma}
\newtheorem{proposition}[definition]{Proposition}
\newtheorem{problem}[definition]{Problem}
\newtheorem{corollary}[definition]{Corollary}

\newtheorem{observation}[definition]{Observation}

\newcommand{\supp}{{\textsc{Supp}}}

\def\I{{\mathcal I}}

\def\L{{\mathcal L}}

\def\T{{\mathcal T}}

\newcommand{\trp}[2]{{#1}^{\!\top}{\!#2}}
\newcommand{\CF}[1]{\varphi_{#1}(\mathbf{t})}

\newcommand{\CFground}[1]{\varphi_{0}(\mathbf{t})}

\renewcommand{\d}{\mathrm{d}}

\newcommand{\gauss}[1]{\mathcal{N}(0, #1)}

\newcommand{\dtv}[2]{\mathrm{d}_{\text{TV}}\left(#1, #2\right)}
\newcommand{\dtvnop}[2]{\mathrm{d}_{\text{TV}}(#1, #2)}
\newcommand{\dchisq}[2]{\chi^2(#1, #2)}
\renewcommand{\Pr}[1]{\mathbb{P}\left( #1 \right)}
\newcommand{\Prnop}[1]{\mathbb{P}( #1 )}

\newcommand{\Expected}[1]{\mathbb{E}\left[ #1 \right]}
\newcommand{\Expectednop}[1]{\mathbb{E}[ #1 ]}
\newcommand{\Expectedsub}[2]{\mathbb{E}_{#1}\left[ #2 \right]}
\newcommand{\Expectedsubnop}[2]{\mathbb{E}_{#1}[ #2 ]}
\newcommand{\Expectedsubb}[2]{\underset{#1}{\mathbb{E}}\left[ #2 \right]}

\newcommand{\erdos}{Erd\H{o}s}
\newcommand{\renyi}{R\'enyi}
\newcommand{\ER}{\erdos--\renyi}

\newcommand{\SW}[1]{ \textsc{Sw}(#1) }

\newcommand{\He}{\mathsf{He}}

\newcommand{\Bern}[1]{\mathsf{Bern}(#1)}

\newcommand{\wishart}{\mathbb{W}(n, m, p, d)}
\newcommand{\wishartmasked}{\mathbb{W}_{\mathbf{M}}(n, m, p, d)}

\newcommand{\matnoiseedge}{\mathbb{W}(n, m, q, p, d)}
\newcommand{\matnoiseknown}{ \mathbb{W}_{\textbf{M}}(n, m, q, p, d)}

\newcommand{\gaussmat}{\mathbb{M}(n, m, p)}
\newcommand{\gaussmatarg}[1]{\mathbb{M}(#1)}

\newcommand{\bound}{nmq^4\log(n)}
\newcommand{\boundwedges}{m\sqrt{n}q^2 \log(n)}

\newcommand{\boundknown}{nmq^2\log(n)}
\newcommand{\boundwedgesknown}{m\sqrt{n}q \log(n)}

\newcommand{\boundlower}{nmq^4}
\newcommand{\boundwedgeslower}{m\sqrt{n}q^2  }

\newcommand{\boundknownlower}{nmq^2}
\newcommand{\boundwedgesknownlower}{m\sqrt{n}q }

\newcommand{\sighat}{\widehat{\sigma}}

\newcommand{\xr}{ \mathbf{X}_{\text{R}} }
\newcommand{\xl}{ \mathbf{X}_{\text{L}} }
\newcommand{\masknum}[1]{ \mathbf{M}^{(#1)} }
\newcommand{\xrnum}[1]{ \mathbf{X}_{\text{R}}^{(#1)} }

\newcommand{\leading}{ \Lambda_\alpha(\xr) }
\newcommand{\rest}{ r_\alpha(\xr) }
\newcommand{\leadingnum}[1]{ \Lambda_\alpha(\xrnum{#1}) }
\newcommand{\leadingnumalpha}[2]{ \Lambda_{#2}(\xrnum{#1}) }
\newcommand{\restnum}[1]{ r_\alpha(\xrnum{#1}) }

\newcommand{\latentstuff}{\xrnum{1}, \xrnum{2} \sim \Seventshort}
\newcommand{\latentstuffind}{\xrnum{1}, \xrnum{2} }
\newcommand{\SWnum}[2]{\textsc{Sw}^{(#2)}(#1)}

\newcommand{\onebutnotactuallyone}{\mathbf{1}_{h}(\mathbf{t}(e))}

\newcommand{\phiterm}[1]{\Phi_{#1}^{(r_1, \ldots, r_k)}(\mathbf{t})}
\newcommand{\phitermell}[1]{\Phi_{#1}^{(r_1, \ldots, r_\ell)}(\mathbf{t})}

\newcommand{\Sevent}{S_\rho}
\newcommand{\Seventshort}{S_\rho}
\newcommand{\Seventdef}{\left\{ \ \mathbf{X} \in \mathbb{R}^{n \times d} \text{ with columns } (\mathbf{x}_u)_{u \in [n]} \text{ such that } \bigg| \tfrac{1}{d} \langle \mathbf{x}_u, \mathbf{x}_v \rangle - \mathbf{I}_{u,v} \bigg| \le \frac{\rho}{\sqrt{d}}  \ \text{ for all } u, v \in [n] \ \right\}}

\newcommand{\leadinghat}{\widehat{\Lambda}_\alpha( \xr, \mathbf{t})}
\newcommand{\resthat}{\widehat{R}_\alpha(\xr, \mathbf{t})}

\newcommand{\Wshort}{\mathbb{W}}
\newcommand{\WshortM}{\mathbb{W}_{\textbf{M}}}
\newcommand{\Mshort}{\mathbb{M}}

\usepackage{thmtools} 
\usepackage{thm-restate}
\usepackage{graphicx}

\makeatletter
\newcommand{\definetitlefootnote}[1]{%
	\newcommand\addtitlefootnote{%
		\makebox[0pt][l]{$^{*}$}%
		\footnote{\protect\@titlefootnotetext}
	}
    \newcommand\@titlefootnotetext{\spaceskip=\z@skip $^{*}$#1}%
}
\makeatother

\begin{document}
\bibliographystyle{plainurl}
\definetitlefootnote{blah}
\title{Information-Theoretic Thresholds for Bipartite Latent-Space Graphs Under Noisy Observations}
\author{Andreas G\"obel, Marcus Pappik, Leon Schiller}

\address{Andreas G\"obel, {\tt andreas.goebel@hpi.de}, Hasso Plattner Insitute, University of Potsdam, Prof.-Dr.-Helmert-Str. 2-3, 14482 Potsdam, Germany.}
\address{Marcus Pappik, {\tt marcus.pappik@hpi.de}, Hasso Plattner Insitute, University of Potsdam, Prof.-Dr.-Helmert-Str. 2-3, 14482 Potsdam, Germany.}
\address{Leon Schiller, {\tt leon.schiller@hpi.de}, Hasso Plattner Insitute, University of Potsdam, Prof.-Dr.-Helmert-Str. 2-3, 14482 Potsdam, Germany.}
\maketitle
\begin{abstract}
    We study information-theoretic phase transitions for the detectability of latent geometry in bipartite random geometric graphs (RGGs) with Gaussian, $d$-dimensional latent vectors, while only a subset of edges carries latent information, determined by a random mask with i.i.d. $\Bern{q}$ entries. For any fixed edge density $p \in (0,1)$, we determine essentially tight thresholds for this problem as a function of $d$ and $q$. Our results show that the detection problem is substantially easier if the mask is known up-front, compared to the case where the mask is hidden.

    Our analysis is built upon a novel Fourier-analytic framework for bounding signed subgraph counts in Gaussian random geometric graphs that exploits cancellations which arise after approximating characteristic functions by an appropriate power series. The resulting bounds are applicable to much larger sub-graphs than considered in previous work, which enables tight information-theoretic bounds, while the bounds considered in previous works only lead to lower bounds from the lens of low-degree polynomials. As a consequence, we identify the optimal information-theoretic thresholds and rule out computational–statistical gaps. Our bounds further improve upon the bounds on Fourier coefficients of random geometric graphs recently given by Bangachev and Bresler [STOC'24] in the dense bipartite case. The techniques extend to sparser and non-bipartite settings as well, at least if the considered sub-graphs are sufficiently small. We further believe that they might help resolve open questions for related detection problems.
\end{abstract}
\section{Introduction }\label{sec:intro}

Latent geometric structure is a common feature in large-scale datasets appearing in various domains including data science~\cite{erba2020random}, statistical physics \cite{gorban2018blessing} or biological sciences~\cite{10002010map}. Random geometric graphs (RGGs) where vertices are represented by points randomly distributed in some latent geometric space and connected as a function of their distance, provide a natural way of modeling such high-dimensional datasets. When the latent geometric space incorporates a natural measure of dimensionality and suitable symmetries, RGGs have been shown to converge (in total variation distance) to \ER{} graphs\footnote{In the \ER{} model edges are drawn independently with probability $p\in(0,1)$.} as the dimension $d$ tends to infinity~\cite{Devorye2011HD}. This prompts the fundamental question of finding the precise regimes of parameters $d, n,$ and edge-density $p$ in which RGGs are distinguishable from \ER{} random graphs of the same edge-density, or as Duchemin and De Castro \cite{duchemin2023random} state it: ``determine the point where the geometry is lost in dimension''.

In this direction, one of the most canonical and widely-studied models is that of \emph{Gaussian RGGs}, which is obtained by associating a latent vector $\mathbf{x}_v$ to every vertex $v$ which is independently drawn from the standard $d$-dimensionals Gaussian distribution $\mathcal{N}(0, \mathbf{I}_d)$. Edges are inserted whenever the inner product $\langle \mathbf{x}_u, \mathbf{x}_v \rangle$ exceeds (or undershoots) a certain threshold $\tau$. 
The resulting adjacency matrix can also be seen as a discretized Wishart matrix, and the model further shares strong similarities with spherical RGGs where vertices are represented by points on $\mathbb{S}^{d-1}$, i.e. the surface of the $d$-dimensional sphere, drawn from the Haar measure. 

In the dense regime, where the edge-density $p\in(0,1)$ is fixed, it is known that Gaussian RGGs become information-theoretically indistinguishable from \ER{} once $d\gg n^3$, while $d\ll n^3$ ensures existence of a simple, efficient test, formed by \emph{signed triangles} ~\cite{bubeck2016testing,jiang2015approximation,racz2019smooth}.
However, understanding variations of this problem that involve some kind of \emph{sparsity} or additional \emph{noise} have turned out to be much more challenging. In this direction, a series of recent works was concerned with the sparse case (where $p = o(1)$), and showed that the thresholds for detectability shift as a function of $p$ \cite{liu2022testing, brennan2020phase, bangachev2024fourier}. Nonetheless, obtaining tight information-theoretic thresholds for all $p$ is still a striking open problem. In a different direction, the works \cite{liu2023noisy, liuracz, brennan2021finetti} have considered other notions of sparsity where $p$ is a fixed constant, while \emph{only a fraction of edges} carries latent information and the other entries are i.i.d. $\Bern{p}$ random variables\footnote{This exact setting is studied in \cite{liu2023noisy}, while the work \cite{brennan2021finetti} studies a related question for wishart matrices, and \cite{liuracz} considers different notions of noise, parameterized by the ``smoothness'' of the connection function}. However, like in the case of $p=o(1)$, finding the precise information-theoretic thresholds for all levels of sparsity/noise is still an open problem, as the bounds given in \cite{liu2023noisy} leave some gaps open \footnote{The bounds given in \cite{brennan2021finetti} are tight, but the model is continuous and assumes that the ``masked'' edges are known up-front, which is in contrast to \cite{liu2023noisy}.}.

In this work, we make progress along the latter line of research, by giving \emph{tight information theoretic thresholds} in an important special case of the models considered in previous work, formed by \emph{bipartite Gaussian RGGs}. In this setting, we close the gaps left open in \cite{liu2023noisy}, and our bounds further reveal striking differences between the setting where edges that carry latent information are revealed up-front (as considered in a similar form in \cite{brennan2021finetti}, also for the bipartite case), and the case where this information is concealed, i.e., not accessible for the testing procedure (like in \cite{liu2023noisy}). Our results further allow us to rule out the existence of computational statistical gaps for all ranges of parameters in our model. Moreover, we hope that the techniques introduced here extend beyond the scope of this work, and believe that they might also yield a better understanding of other variations of our detection problem, such as the case of $p = o(1)$.


\subsection{The model and associated testing problems}
We proceed by introducing the concrete testing problems we will be working with. 
To this end, denote by $\wishart$ the following distribution over $n \times m$ matrices with entries in $\{0,1\}$, where we refer to the set of rows as $R$ and to the set of columns as $L$, with $|R| = n$ and $|L| = m$. Without loss of generality, we further assume throughout that $m \ge n$. 
\begin{definition}[The distributions $\wishart$ and $\gaussmat$]
    We let $\wishart$ be the distribution over $n \times m$ matrices with entries in $\{0,1\}$ obtained as follows. Sample independent random vectors $\xr = (\mathbf{x}_{u})_{u \in R}$ and $\xl = (\mathbf{x}_{u})_{u \in L}$ from the standard $d$-dimensional Gaussian distribution $\mathcal{N}(0, \mathbf{I}_d)$.
    Define further the threshold $\tau = \tau(d, p)$ such that for any $u \in L, v \in R$, we have $
        \Prnop{d^{-1/2}\langle \mathbf{x}_u, \mathbf{x}_v \rangle \le \tau} = p.
    $
    Then, after drawing $(\mathbf{x}_{u})_{u \in R}$ and $(\mathbf{x}_{u})_{u \in L}$,  determine each entry of $W\sim \wishart$ by thresholding\footnote{Instead of thresholding using the indicator $d^{-1/2}\langle \mathbf{x}_u, \mathbf{x}_v \rangle \le \tau$, we could also use $d^{-1/2}\langle \mathbf{x}_u, \mathbf{x}_v \rangle \ge \tau$ instead and our results would remain valid, up to some flipped signs. } $$
        W_{u, v} \coloneqq \mathds{1}( d^{-1/2} \langle \mathbf{x}_u, \mathbf{x}_v \rangle \le \tau).
    $$
    Moreover, we define $\ \gaussmat$ as as the distribution over $n \times m$ matrices with entries in $\{0,1\}$ where each entry is an i.i.d. sample from the $\Bern{p}$ distribution.
\end{definition} 
Note that in $W \sim \wishart$, each entry is marginally a $\Bern{p}$ random variable just like in $\gaussmat$, however the difference is that there are intricate dependencies within $W \sim \wishart$ as all the entries depend on the same underlying latent information, whose strength crucially depends on the dimension $d$. A matrix $W \sim \wishart$ can also be seen as the adjacency matrix of a \emph{bipartite Gaussian RGGs} with edge-density $p$, while $M \sim \gaussmat$ is distributed like the adjacency matrix of a bipartite \ER{} random graph.

\begin{figure}
    \includegraphics[width=.95\linewidth]{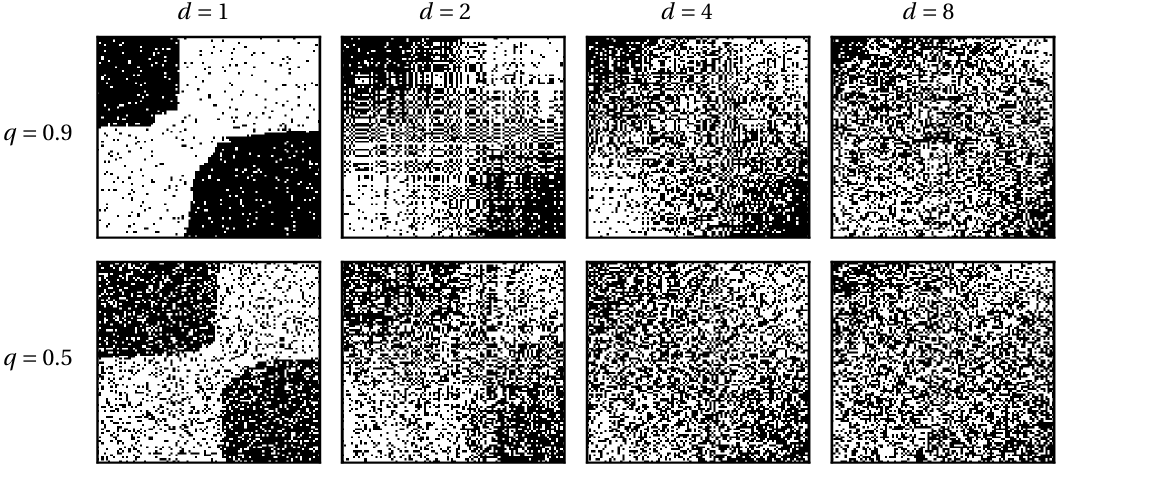}
    \caption{Illustration of the matrices sampled from $\matnoiseedge$ for different $d$ and $q$. The rows and columns are ordered by the first coordinate of the latent vectors. Note that in our problem, an algorithm would not have access to this information, instead the rows and columns would be given in a random permutation of the matrices shown above. We sorted the rows and columns only for the sake of visualization.}
    \label{fig:matrix}
\end{figure}

As briefly mentioned earlier, previous works (see for example \cite{bubeck2016testing, brennan2021finetti}) have been concerned with understanding for which values of $d$ (as a function of $n$), a matrix $W$ sampled from $\wishart$ and closely related distributions can be distinguished from a typical sample $M \sim \gaussmat$. In this regard, it follows from the work of Brennan, Bresler and Huang \cite[Corollary 4.3]{brennan2021finetti} that the total variation distance $\dtv{\wishart}{\gaussmat} \rightarrow 0$ if $d \gg nm$ while $\dtv{\wishart}{\gaussmat} \rightarrow 1$ if $d \ll nm$ and $p \in (0,1)$ is any fixed constant. 

\textbf{Introducing sparsity via random masks of density $q$.}
While this result fully characterizes the problem in terms of its feasibility as a function of $d$, it is much more challenging to find the right thresholds if the considered matrix additionally exhibits some form of \emph{sparsity} in the sense that only a small fraction of entries actually carries information about the underlying latent vectors.
One of the most natural ways of introducing sparsity to the above problem is the so called \emph{masked model}, as also introduced in \cite{brennan2021finetti}. 
Here, we we are concerned with detecting dependence within a matrix $M$ while being given as input a \emph{mask} $\ \mathbf{M} \sim \gaussmatarg{n, m, q}$ and the matrix $\mathbf{M} \odot M$ (where $\odot$ denotes the Hadamard product). The goal is to distinguish the following two hypotheses.
$$
    H_0: M  \sim \gaussmat \hspace{.5cm} \text{ and } \hspace{.5cm} H_1: M \sim \wishart.
$$
Accordingly, any testing procedure can only observe the subset of entries in $M$ that is not ``hidden'' by the mask $\mathbf{M}$ (i.e. those entries that correspond to $\mathbf{M}_{u,v} = 1$). Therefore, the masked model introduces a new parameter $q$ that controls the number of observable entries (or the sparsity of the resulting masked matrix $\mathbf{M} \odot M$) and heavily influences the resulting thresholds for distinguishability. Tight information-theoretic thresholds for testing in this masked setting as a function of $d$ and $q$ were given in \cite{brennan2021finetti} for a \emph{continuous} version of the above problem, where a Wishart matrix is to be distinguished from a matrix with i.i.d. Gaussian entries.   

\textbf{Testing with known and unknown masks.}
The masked setting considered so far assumes that the mask $\bf{M}$ is given explicitly, which means that a testing procedure ``knows'' which entries in the input $M$ carry latent information and which do not. It is very natural to ask what happens if this is not the case, i.e., if the mask $\mathbf{M}$ is not given explicitly such that one cannot clearly separate the masked and non-masked edges a-priori. To introduce this new setting formally and compare it to the original masked model, we define the following two distributions. 
\begin{definition}[The distributions $\matnoiseedge$ and $\wishartmasked$]
    Given a matrix $\mathbf{M} \in \{0,1\}^{ n \times m }$,
    define the distribution $\wishartmasked$ as the distribution over matrices $M\in \{0,1\}^{n \times m}$ obtained by sampling $W \sim \wishart$, $B \sim \gaussmat$ and then setting $$
        M \coloneqq W \odot \mathbf{M} + B \odot(1- \mathbf{M}),
    $$
    where $\odot$ denotes the Hadamard product.
    Moreover, define $\matnoiseedge$ as the distribution $\wishartmasked$ resulting after drawing $\mathbf{M}$ from $\gaussmatarg{n,m,q}$, i.e., $$
        \matnoiseedge = \Expectedsub{\mathbf{M} \sim \gaussmatarg{n,m,q}}{\wishartmasked}.
    $$
\end{definition}
Accordingly, a sample $M \sim \matnoiseedge$ is obtained by choosing a matrix $W \sim \wishart$ and a mask $\mathbf{M} \sim \gaussmatarg{n, m, q}$, and subsequently re-randomizing each entry $M_{u,v}$ corresponding to $\mathbf{M}_{u,v} = 0$ by replacing it with an independent sample from $\Bern{p}$. This re-randomization step serves the purpose of ``hiding'' the mask $\mathbf{M}$ by ensuring that every entry in $M$ has the same marginal distribution. This can equivalently be seen as a form of introducing noise, and it reproduces the setting considered in \cite{liu2023noisy} for (non-bipartite) RGGs on the sphere.

With this, our testing problems for known and unknown masks can be described as follows.
\begin{problem}[Distinguishing $\wishart$ and $\gaussmat$ for unknown masks]\label{prob:testingunknown}
    Given a matrix $M$, distinguish the following two hypotheses. 
    \begin{align*}
        H_0: M \sim \gaussmat \hspace{.3cm} \text{ and } \hspace{.3cm} H_1: M \sim \matnoiseedge.
    \end{align*}
\end{problem}
\begin{problem}[Distinguishing $\wishart$ and $\gaussmat$ for known masks]\label{prob:testingknown}
    Given a matrix $M$ and a mask $\mathbf{M} \sim \gaussmat$, distinguish the following two hypotheses. 
    \begin{align*}
        H_0: M \sim \gaussmat \hspace{.3cm} \text{ and } \hspace{.3cm} H_1: M \sim \wishartmasked.
    \end{align*}
\end{problem}
Note that the above testing problem for known masks is introduced slightly differently than in the previous paragraph in the sense that the entries ``hidden'' by the mask are independent $\Bern{p}$ random variables in both hypotheses, while previously they were set to zero under both $H_0$ and $H_1$. However, it is not hard to see that from the lens of distinguishability, the two settings are equivalent, assuming that the mask is part of the input. Our way of stating \Cref{prob:testingknown} and \Cref{prob:testingunknown} has the advantage that it allows us to to compare the the two settings more easily.
To get an intuition for how the dependencies within a sample from $\matnoiseedge$ behave and in particular become weaker as $d$ grows and $q$ tends to zero, we refer to \Cref{fig:matrix}. 

\subsection{Results}

\begin{figure}
\centering
\begin{subfigure}[c]{0.35\textwidth}
\centering
\includegraphics[width=1\textwidth]{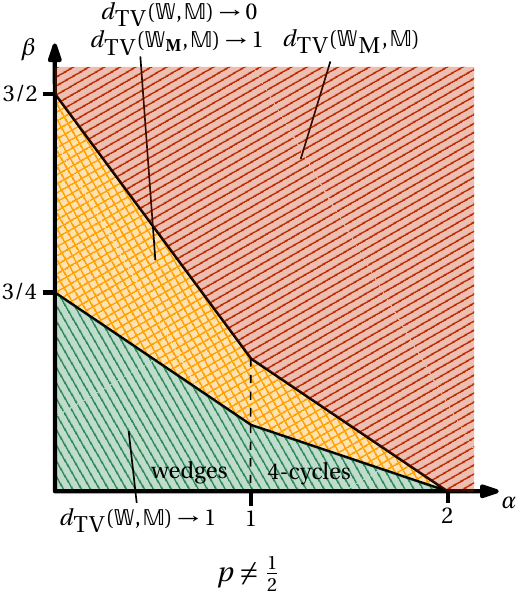}
\end{subfigure}
\begin{subfigure}[c]{0.05\textwidth}
\centering
\hspace{1.3cm}
\end{subfigure}
\begin{subfigure}[c]{0.35\textwidth}
\centering
\includegraphics[width=1\textwidth]{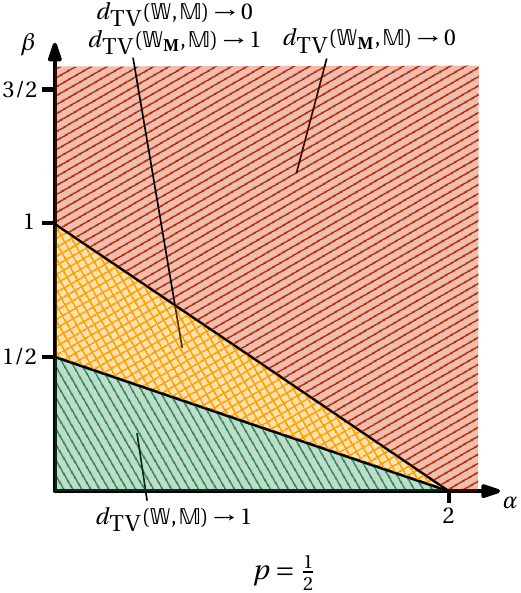}
\end{subfigure}
\captionsetup{singlelinecheck=off}
\caption[foo bar]{Phase diagrams for $q = n^{-\beta}$, $d = n^{\alpha}$, and $n = m$. 
The colors represent the following regimes:
\begin{itemize}
    \item[{\includegraphics[width=0.9em]{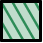}}] In the green regime (falling pattern), $\matnoiseedge$ and $\gaussmat$ are distinguishable, even if the mask is not given. Efficient tests are given by counting signed wedges and signed 4-cycles.
    \item[{\includegraphics[width=0.9em]{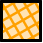}}] In the yellow regime (rising and falling pattern), the two models are indistinguishable if the mask is unknown but distinguishable if it is known. Efficient tests are again given by counting signed wedges and signed 4-cycles, but restricted to the mask.
    \item[{\includegraphics[width=0.9em]{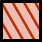}}] In the red regime (rising pattern), the models are indistinguishable even if the mask is given.
\end{itemize} 
For $p \neq \frac{1}{2}$, counting signed wedges supersedes counting signed $4$-cycles as an optimal test statistic for $\alpha \le 1$ (vertical dashed line).
In contrast, for $p=\frac{1}{2}$, counting signed wedges has no statistical power, and the testable regime is purely determined by the statistical power of signed $4$-cycles.
} \label{fig:diagram}
\end{figure}

The main result of this work is the establishment of precise, information-theoretic thresholds for the feasibility of \Cref{prob:testingknown} and \Cref{prob:testingunknown} that are tight up to logarithmic factors.
Before stating them formally, we introduce some further notation. To this end, we remark that we write $f \gg g$ for functions $f(n), g(n)$ whenever $\lim_{n \to \infty} f(n)/g(n) = \infty$. Furthermore, we use standard Landau-notation to characterize the asymptotic behavior of functions, and we denote graphs by lower-case Greek letters. A graph $\alpha$ is represented as a set of edges such that $|\alpha|$ denotes the number of edges, while $V(\alpha)$ is the corresponding set of vertices and $|V(\alpha)|$ is its cardinality.

With this, our first result concerns the feasibility of \Cref{prob:testingunknown} and reads as follows. Note that we assume $d \gg \log(n)^3$ throughout, simply because it simplifies some of our arguments, in particular related to the algorithmic upper bounds. However, we strongly believe that this assumption can be removed at the expense of a higher technical complexity.
\begin{theorem}[Information-theoretic thresholds for unknown masks]\label{thm:resultunknownmask}
    Consider any fixed $p \in (0,1)$. Then, the following holds whenever $d \gg \log(n)^3$.
    
    \vspace{0.5em}
    If $p \neq \frac{1}{2}$ then:\hspace{1em} $\begin{aligned}[t]
            \dtv{\matnoiseedge}{ \gaussmat} &= \begin{cases} 1 - o(1) & \text{if } d \ll \boundlower  \hspace{1.3cm} \text{ or } \hspace{.1cm} d \ll \boundwedgeslower \\
            o(1) & \text{if } d \gg \bound \hspace{.1cm} \text{ and } \hspace{.1cm} d \gg \boundwedges
            \end{cases}
        \end{aligned}$
        
    \vspace{0.5em}
    If $p = \frac{1}{2}$ then:\hspace{1em} $\begin{aligned}[t]
            \dtv{\matnoiseedge}{ \gaussmat} &= \begin{cases} 1 - o(1) & \text{if } d \ll \boundlower \\
            o(1) & \text{if } d \gg \bound
            \end{cases}
        \end{aligned}$

    \vspace{0.5em}
    \noindent Here, $o(1)$ denotes a function that tends to zero as $n \rightarrow \infty$. Recall further that $m \ge n$.
\end{theorem}

\noindent We highlight some of the implications of the above theorem in the following. See also \Cref{fig:diagram} for a visualization of the thresholds.
 
\textbf{Efficient tests and absence of computational-statistical gaps.} The above upper and lower bounds on the total variation distance between the two matrix ensembles are almost tight and differ only by a factor of $\log(n)$. In the regime where $\lim \dtv{\mathbb{W}}{\mathbb{M}} = 1$, there are two different regimes characterized by distinct, optimal tests. If $d \ll \boundlower$, then counting \emph{signed 4-cycles} distinguishes $W$ and $M$, and if $d \ll \boundwedgeslower$, then this holds for counting \emph{signed wedges}. Formally, for a matrix $M \in \{0, 1\}^{n \times m}$ these are given by 
\begin{align*}
    C_4(M) &\coloneqq \sum_{ \substack{i < j \\ i,j \in[n]}}\sum_{ \substack{k < \ell \\ k,\ell \in[m]}} (M_{i,k} - p)(M_{i,\ell} - p)(M_{j,k} - p)(M_{j,\ell} - p) \text{ and } 
    P_2(M) \coloneq \sum_{ \substack{i \in[n]}}\sum_{ \substack{k < \ell \\ k,\ell \in[m]}} (M_{i,k} - p)(M_{i,\ell} - p),
\end{align*}
respectively. \Cref{thm:resultunknownmask} shows that once we are in the regime where none of these statistics succeeds\footnote{or strictly speaking, if we exceed this regime by a factor of $\log(n)$}, then \emph{no algorithm} can distinguish the two distributions. Moreover, since $C_4(M)$ and $P_2(M)$ are clearly efficiently computable, this implies in particular that there are \emph{no computational statistical gaps}.

\textbf{The case $p = \frac{1}{2}$ vs. $p \neq \frac{1}{2}$.}
Moreover, \Cref{thm:resultunknownmask} highlights a somewhat surprising effect arising from the choice of $p$: for $p = \frac{1}{2}$, the regime in which $\lim \dtv{\mathbb{W}}{\mathbb{M}} = 1$ is significantly smaller than for the case where $p$ is strictly larger or smaller than $\frac{1}{2}$. This is a consequence of the symmetry of the Gaussian distribution. In particular, if $p = \frac{1}{2}$, then $\tau = 0$ which implies that counting wedges has no statistical power since the two entries $M_{i,k}$ and $M_{i,\ell}$ will be i.i.d. $\Bern{\frac{1}{2}}$ random variables, even conditional on a concrete latent vector $\mathbf{x}_i$. \Cref{thm:resultunknownmask} shows that this effect is not specific to signed wedges, but that it is inherent in the sense that no algorithm can distinguish $W$ and $M$ in the regime where signed four-cycles fail, even if signed wedges would succeed if $p \neq \frac{1}{2}$. Hence, choosing $p = \frac{1}{2}$ makes \Cref{prob:testingunknown} strictly harder compared to any other fixed choice of $p$.

\textbf{Known vs. unknown masks.}
In the setting of \Cref{thm:resultunknownmask}, we assume that the random mask determining which entries are to be re-randomized is not part of the input. The following theorem shows that the testing thresholds established in \Cref{thm:resultunknownmask} shift significantly if this is not the case, i.e., if we instead consider \Cref{prob:testingknown}.
\begin{theorem}[Information-theoretic hardness for known masks]\label{thm:resultknownmask}
    Consider any fixed $p \in (0,1)$. 
    Then, the following holds with probability $1 - o(n)$ over $\ \mathbf{M} \sim \gaussmatarg{n, m, q}$ whenever $d \gg \log(n)^3$.
    
    \vspace{0.5em}
    If $p \neq \frac{1}{2}$ then:\hspace{1em} $\begin{aligned}[t]
            \hspace{0.8cm}\dtv{\wishartmasked}{ \gaussmat} &= \begin{cases} 1 - o(1) & \text{if } d \ll \boundknownlower  \hspace{1.3cm} \text{ or } \hspace{.1cm} d \ll \boundwedgesknownlower \\
            o(1) & \text{if } d \gg \boundknown \hspace{.1cm} \text{ and } \hspace{.1cm} d \gg \boundwedgesknown
            \end{cases}
        \end{aligned}$
        
    \vspace{0.5em}
    If $p = \frac{1}{2}$ then:\hspace{1em} $\begin{aligned}[t]
        \hspace{0.8cm}\dtv{\wishartmasked}{ \gaussmat} &= \begin{cases} 1 - o(1) & \text{if } d \ll \boundknownlower \\
            o(1) & \text{if } d \gg \boundknown
            \end{cases}
        \end{aligned}$

    \vspace{0.5em}
    \noindent Here, $o(1)$ denotes a function that tends to zero as $n \rightarrow \infty$. Recall further that $m \ge n$.

    
\end{theorem}
\noindent Accordingly, comparing the thresholds given in \Cref{thm:hardnessknown} and \Cref{thm:hardnessunknown}, we can see that switching from known to unknown masks amounts to replacing $q$ by $q^2$. 

\textbf{Discrete vs. continuous models.} Note that while the setting of known masks was already considered in \cite{brennan2021finetti} for the continuous analogue of our model, the above theorem shows that the discrete version of the problem considered here exhibits \emph{different} thresholds once $q$ is sufficiently small, even if $p \neq \frac{1}{2}$. This is because in the continuous model, there are subtle differences in the marginal distribution of every entry $M_{u,v}$ when comparing $H_0$ and $H_1$. Testing for these differences enables an efficient test whenever $d \ll \sqrt{nmq}$. In our discrete version, the marginals match exactly and convergence in total variation therefore occurs earlier. More details on the differences of the discrete and the continuous setting can be found in \Cref{sec:challengespreviouswork}.

\textbf{Convergence in total variation, but  only for averaged masks.}
Abbreviating for now $\Wshort = \matnoiseedge$, $\WshortM = \wishartmasked$, and $\Mshort = \gaussmat$, then the bounds from \Cref{thm:resultunknownmask} and \Cref{thm:resultknownmask} lead to the remarkable phenomenon that for all fixed $p$, there is a regime of the parameters $d, q$ where $\dtv{\Wshort}{\Mshort} = o(1)$, while for the same $d$ and $q$, we have $\dtvnop{\WshortM}{\Mshort} = 1 - o(1)$, with high probability over the choice of $\mathbf{M}$. This means that for most fixed masks $\mathbf{M}$ of density $q$ (i.e. $\mathbf{M} \sim \gaussmatarg{n, m, q}$), the resulting conditional distribution $\WshortM$ is efficiently distinguishable from $\Mshort$, while in expectation over $\mathbf{M}$, convergence in total variation occurs. In particular, with high probability over $\mathbf{M}$, we have  
\begin{align*}
    \dtv{\WshortM}{\Mshort} = 1  - o(1) \text{ while } \dtv{\Expectedsub{\mathbf{M}}{\WshortM} }{\Mshort} = o(1).
\end{align*}
In other words, the distribution $\WshortM$ is very far from $M$ for most masks $\mathbf{M}$, while it is asymptotically indistinguishable from $M$ when considering the averaged distribution $\Wshort = \Expectedsubnop{\mathbf{M}}{\WshortM}$. Another point of view on this phenomenon is that every algorithm that distinguishes $\Wshort$ from $M$ in the regime where $\dtv{\Wshort}{\Mshort} = o(1)$ must depend on $\mathbf{M}$.

\section{Technical Contributions}

\subsection{Second moment method}
Like many related results from the literature, our information-theoretic lower bounds are essentially based on applying a version of the \emph{second moment method}. Concretely, this means that given two probability distributions $\mu, \nu$, we wish to relate total variation distance $\dtv{\mu}{\nu}$ to the so-called $\chi^2$-divergence $\dchisq{\mu}{\nu}$. These divergences are given by
$$
    \dtv{\mu}{\nu} \coloneqq \frac{1}{2}\Expectedsubb{\xi \sim \nu}{ \left| \frac{\d \mu}{\d \nu}(\xi)-1\right| } \text{ and } \dchisq{\mu}{\nu} \coloneqq \Expectedsubb{\xi \sim \nu}{ \left( \frac{\d \mu}{\d \nu}(\xi)\right)^2 } - 1,
$$
where $\frac{\d \mu}{\d \nu}(\xi)$ denotes the probability of sampling the matrix $\xi$ under distribution $\mu$ divided by the probability of sampling $\xi$ under $\nu$.
Now, we bound $\dtv{\mu}{\nu}$ by means of Pinsker's inequality and by passing from relative entropy to $\chi^2$-divergence as
$
    2\dtv{\mu}{\nu}^2 \le \dchisq{\mu}{\nu}.  
$
The distribution $\mu$ can be expressed as a mixture $\mu = \Expectedsub{\phi}{\mu_\phi}$ over some latent randomness $\phi$ where $\mu_\phi$ denotes the distribution $\mu$ conditional on $\phi$. 
We then expand the squared expectation into two expectations over independent copies $\phi^{(1)}, \phi^{(2)}$ of the latent randomness and exchange the inner and outer expectation to get 
$$
    \dchisq{\mu}{\nu} = \Expectedsubb{\xi \sim \nu}{ \left( \frac{\d \mu}{\d \nu}(\xi)\right)^2 } - 1 = \Expectedsubb{\phi^{(1)}, \phi^{(2)} \sim \varrho \otimes \varrho}{ \Expectedsubb{\xi \sim \nu}{ \frac{\mathrm{d} \mu_{\phi^{(1)}}}{\mathrm{d} \nu}(\xi) \frac{\mathrm{d} \mu_{\phi^{(2)}}}{\mathrm{d} \nu}(\xi) }  } - 1,
$$
where $\varrho$ is the law of the latent randomness.
The advantage of this re-formulation is that the inner expectation can now be computed exactly since $\nu$ is a simple product distribution.

\subsection{Challenges arising from previous work}\label{sec:challengespreviouswork}
The starting point for this work is the work of Brennan, Bresler, and Huang \cite{brennan2021finetti} who studied a similar problem for fixed and known masks $\mathbf{M}$ of density $q$. The main difference to our setting is that they study a continuous version of our problem where the goal is to distinguish a (masked) Wishart matrix $W = \frac{1}{\sqrt{d}}(\mathbf{X}^\top\mathbf{X} - d\mathbf{I})$ from a matrix with i.i.d. standard Gaussian entries, and that they only consider the setting of known masks.

\textbf{Different thresholds in the continuous vs. discrete model.}
While information-theoretic lower bounds for the continuous model considered in \cite{brennan2021finetti} directly imply the same lower bounds for the discrete model studied here (for the case of known masks), it turns out that in the presence of noise, these lower bounds are not strong enough to tightly characterize testing in the discrete setting, even if $p$ is a constant (this is in contrast to the noiseless case). This can already be seen by comparing our bounds from \Cref{thm:resultknownmask} with the bounds in \cite[Theorem 4.2]{brennan2021finetti}. Concretely, in the case of relatively small $d$, the continuous model is much \emph{less} noise sensitive than its discrete counterpart as it remains efficiently distinguishable from the i.i.d. Gaussian ensemble as long as $d \ll \sqrt{nmq}$, while our model becomes \emph{indistinguishable} from $M$ if $d \gg \boundwedgesknown$, i.e., at much lower levels of noise, provided that $d$ is sufficiently small. Consider for example the case where $n = m$ and $d$ scales as some power of $\log(n)$, then the discrete model converges in total variation to $M$ if $q \ll n^{-3/2}\log(n)$, while for the continuous case, convergence to the i.i.d. Gaussian ensemble only occurs for $q \ll n^{-2}$. If $p = \frac{1}{2}$, the difference is even bigger, here convergence occurs if $q \ll n^{-1}\log(n)$ in the discrete setting, while the threshold stays at $q \ll n^{-2}$ for the continuous model.

\textbf{Handling unknown masks.}
In addition to the fact that the two models provably behave differently in the presence of noise, the techniques given in \cite{brennan2021finetti} do not fully extend to the technically much more challenging case of unknown masks. This is because not knowing the mask results in convergence in total variation at much lower levels of noise $q$, which means that for a given $q$, the proofs need to yield convergence for asymptotically much smaller values of $d$, which requires much more careful arguments. In particular, for any fixed $q$, every column of $W \sim \matnoiseedge$ will contain roughly $nq$ entries that have not been re-randomized. Conditional on the random vectors $\xr$ in $R$, the distribution of these $k \approx nq$ entries is jointly Gaussian with a covariance matrix $\boldsymbol{\Sigma}(\xr) \in \mathbb{R}^{k\times k}$ dictated by the inner products of the vectors in $R$. A crucial fact exploited in the proofs in \cite[Section 9]{brennan2021finetti} is that $\boldsymbol{\Sigma}(\xr)$ is positive definite, which ensures that the resulting distribution has a well-behaved probability density such that we can bound total variation by means of the Radon-Nikodym derivative of (two copies of) $\mathcal{N}(0, \boldsymbol{\Sigma}(\xr))$ and the standard Gaussian $\mathcal{N}(0, \mathbf{I})$, which arises from $\dchisq{\mu}{\nu}$ when applying the second moment method.

However, to ensure a that $\mathcal{N}(0, \boldsymbol{\Sigma}(\xr))$ has a well-behaved density, we have to require $d \gg nq$ as for $d \ll nq$, it is not hard to see that $\boldsymbol{\Sigma}(\xr)$ is a matrix of rank much less than $k \approx nq$, so it has a non-trivial kernel and is \emph{not} positive definite. This makes the resulting Radon-Nikodym derivative \emph{infinite} and, thus, results in failure of the second moment method. For our setting, however, we cannot assume that $d \gg nq$ as we wish to obtain convergence already when $d \gg \bound + \boundwedges$.

\subsection{Bounding total variation in terms of signed subgraph counts}
While the proofs in \cite{brennan2021finetti} bound a continuous version of $\dchisq{\mu}{\nu}$, we proceed with discrete distributions $\mu, \nu$ and use a calculation similar to the one in \cite{liuracz} to get a sum over the \emph{expected signed weight} $\Expected{\SW{\alpha}}$ of all sub-graphs $\alpha \subseteq K_{n,m}$, defined as follows.
\begin{definition}[Signed weights]
Let $\sigma(\langle \mathbf{x}_u, \mathbf{x}_v \rangle)$ be the indicator function of the event $\frac{1}{\sqrt{d}} \langle \mathbf{x}_u, \mathbf{x}_v \rangle \le \tau$. For any $\alpha \subseteq K_{n, m}$, and latent vectors $(\mathbf{x})_{u \in L \cup R}$ given by the columns of $\ \xl, \xr$, we define the \emph{signed weight} of $\alpha$ as
    $$
        \SW{\alpha} \coloneqq {\prod}_{\{u,v\} \in \alpha} ( \sigma(\langle \mathbf{x}_u, \mathbf{x}_v \rangle) - p  ).
    $$
\end{definition}
\noindent In the setting of an unknown mask, our expression for $\dchisq{\mu}{\nu}$ admits an explicit representation in terms of expected signed weights:
\begin{align}\label{eq:sumsignedweights}
    1 + \dchisq{\mu}{\nu} = \Expectedsubb{\phi^{(1)}, \phi^{(2)} \sim \rho \otimes \rho}{ \Expectedsubb{\xi \sim \nu}{ \frac{\mathrm{d} \mu_{\phi^{(1)}}}{\mathrm{d} \nu}(\xi) \frac{\mathrm{d} \mu_{\phi^{(2)}}}{\mathrm{d} \nu}(\xi) }  }= 1 + \sum_{\emptyset \neq \alpha \subseteq K_{n, m}} \frac{q^{2|\alpha|}}{(p(1-p))^{|\alpha|}}\ \Big( \  \underset{\xr, \xl}{\mathbb{E}}{\big[\SW{\alpha}\big] } \ \Big)^2
\end{align}
and our problem reduces to showing that the above sum is $o(1)$ once $d \gg \bound + \boundwedges$. For the details in deriving the above expression, we refer to \Cref{sec:info}


\subsection{Bounding expected signed weights}
The main technical challenge is now to find good bounds on the expected signed weights appearing in the above sum. We outline the main challenges arising from this approach.

\textbf{Bounds given in previous work.}
While signed subgraph counts for random geometric graphs have been studied in several recent works, see for example \cite{bangachev2024fourier, bangachev2024detection, baguley2025testing}, meaningful bounds have so far only been obtained for very small patterns $\alpha$ with at most $\text{polylog}(n)$ edges. For this work, however, we require bounds that stay meaningful for \emph{all} $\alpha \subseteq K_{n, m}$, with a number of edges of order up to $\sim nm$.

While an extension of the bounds given in the recent work of Bangachev and Bresler \cite{bangachev2024fourier} to large $\alpha$ would be no easy task on its own, it turns out that even if the bounds given in \cite{bangachev2024fourier} would extend to larger $\alpha$ as they are, they would nonetheless be too weak for our purposes, even in the noiseless case. The reason for this is that (as described in \cite[Remark 1]{bangachev2024fourier}) the bounds given in \cite{bangachev2024fourier} are of order at least $(1/\sqrt{d})^{|V(\alpha)|}$, while we would need a decay of order $(1/\sqrt{d})^{|\alpha|}$ (i.e. a bound that decays exponentially in the number of \emph{edges} instead of vertices) to make the sum in \eqref{eq:sumsignedweights} small enough. This is because a decay of order $(1/\sqrt{d})^{|V(\alpha)|}$ fails to counterbalance the number of possible $\alpha \subseteq K_{n , m}$ already if $\alpha$ has only $n^\varepsilon$  vertices. 

One possible approach for dealing with this issue (as also suggested in \cite[Remark 1]{bangachev2024fourier}) is to condition on a subset of the latent information given by one side of the bipartition. While this is indeed an important aspect of our approach, the arguments given in \cite{bangachev2024fourier} would still not yield an exponent that is strong enough for our purposes. Moreover, as we will see later, we do not only need a bound on $\Expected{\SW{\alpha}}$, but we also need to explicitly control how exactly this quantity \emph{depends} on the latent vectors we condition on, which seems to be out of reach of the techniques considered so far.

\textbf{Our approach.}
Our main technical contribution is therefore a novel, Fourier-theoretic way of bounding expected signed weights for bipartite Gaussian RGGs that seemlessly extends to large $\alpha$ while also improving on the known bounds for small patterns. It makes a statement on $\Expected{\SW{\alpha}}$ in the conditional probability space obtained by conditioning on the high-probability event that the inner products of all right-sided random vectors are not much larger than we would expect. Concretely, we define $\Sevent \in \sigma(\xr)$ as $$
    \Sevent \coloneqq \Seventdef.
$$
Our main technical theorem now expresses the expected signed weight of an arbitrary star $K_{1, \alpha}$ with center in $L$ and leaves given by $\alpha \subseteq R$ as a polynomial in the latent vectors in $\alpha$ plus a lower-order correction term. 

\begin{restatable}[Bound on conditional signed weights]{prp}{fourierstuff}
\label{prop:fourierexpressionthing}
    Define $\sighat$ such that $\Phi_{\sighat}(\tau) = p$, where $\Phi_{\sighat}$ is the CDF of the distribution $\mathcal{N}(0, \sighat^2)$. For any given set $\alpha \subseteq R$ define the random variable
    $
        \SW{K_{1,\alpha}} \coloneqq \prod_{u \in \alpha} (\sigma(\langle \mathbf{x}_{u}, \mathbf{x} \rangle) - p)
    $
    where $(\mathbf{x}_u)_{u \in \text{R}}$ are the columns of $\ \xr$.
    Then, there are functions $\leading, \rest$ such that for all $\alpha \subseteq R$, we have
    $
        \Expectedsub{\mathbf{x}}{ \SW{K_{1,\alpha}} \mid \xr } = \leading + \rest
    $
    whenever $\xr \in \Seventshort$.
    Defining $\ell \coloneqq \lceil |\alpha| / 2 \rceil$, we further have 
    \begin{align*}
        \leading \coloneqq \frac{ (-1)^\ell}{2^\ell \ell! } \sum_{ \substack{ r_1, \ldots r_{\ell} \in \alpha\times \alpha \\ \supp(r_1, \ldots, r_\ell) = \alpha }} \ \Bigg( \prod_{j = 1}^k \Big( \tfrac{1}{d} \big\langle \mathbf{x}_{r_j(1)} ,\mathbf{x}_{r_j(2)} \big\rangle - \sighat^2 \mathbf{I}_{r_j(1),r_j(2)}\Big) \Bigg) \ \Bigg( \prod_{e \in \alpha}\phi^{(s_e -1)}(\tau) \Bigg)
    \end{align*} where {\normalfont  $\supp(r_1, \ldots, r_k) \coloneqq \{ e \in \alpha \mid r_j(1) = e \text{ or } r_j(2) = e \text{ for some }j \}$} is the set of edges covered by the 2-tuples $r_1, \ldots, r_k$, where $s_e$ is the number of tuples $r_j$ that cover edge $e$, i.e., 
    $
        s_e \coloneqq | \{ j \in [k] \mid r_j(1) = e \} | + | \{ j \in [k] \mid r_j(2) = e \} |,
    $
    and where $\phi^{(s)}$ is the $s$-th derivative of the standard Gaussian density. Moreover, there is an absolute constant $C > 0$ such that 
    $$
        |\leading| \le \Bigg( \frac{C\rho  |\alpha| }{\sqrt{d}} \Bigg)^{\ell} \text{ and } |\rest| \le \Bigg( \frac{C \rho |\alpha| }{\sqrt{d}} \Bigg)^{\ell + 1}.
    $$
\end{restatable}
 Note that while we only make a statement about sub-graphs of the form $K_{1,\alpha}$, the above immediately yields an analogous bound for arbitrary $\alpha \subseteq K_{n, m}$ by exploiting independence over the latent vectors in $L$. Looking at the given explicit bounds on $\leading$ and $\rest$, it is easy to see that the conditional signed weights decay exponentially in $|\alpha|$, i.e., in the number of \emph{edges}, instead of the number of vertices. For small $|\alpha|$, it is easy to further remove the conditioning on $\Seventshort$ and get an explicit bound on the low-degree Fourier coefficients that, in particular, improve upon the bounds given in~\cite{bangachev2024fourier} if $\alpha$ is bipartite.

\begin{corollary}[Bound on unconditional signed weights]\label{cor:sw}
    For any $\alpha \subseteq K_{n, m}$, we have 
    \begin{align*}
        \underset{\xl, \xr}{\mathbb{E}}{\big[\SW{\alpha}\big] } \le \left( \frac{C|\alpha|^{3/2} \log(d)^{1/2}}{\sqrt{d}} \right)^{|\alpha|/2}.
    \end{align*}
\end{corollary}
\begin{proof}
    We set $\rho = C\sqrt{|\alpha|\log(d)}$ for a sufficiently large constant $C > 0$. Then, using that each inner product is the sum of $d$ i.i.d. random variables, we get from a Chernoff bound and a union bound over all of at most $|V(\alpha)|^2$ pairs of latent vectors that $
        \Prnop{ \overline{\Seventshort} } \le 2 |V(\alpha)|^2 \exp(-c\rho^2 ) \le (C/\sqrt{d})^{|\alpha|/2}
    $ for some constant $c > 0$ and our choice of $\rho$, provided that the constant $C$ is chosen large enough. Then, applying the bounds from \Cref{prop:fourierexpressionthing} and the law of total expectation yields the corollary.
\end{proof}

\subsection{Information-theoretic bounds via a conditional second moment method}

With \Cref{prop:fourierexpressionthing} at hand, we wish to derive an upper bound on the sum in \eqref{eq:sumsignedweights}. However, as already mentioned, simply summing over unconditional squared expected signed weights using \Cref{cor:sw} does not yield the desired bounds. Instead, we condition the right-sided latent vectors to be in line with $\Sevent$ while setting $\rho \approx \sqrt{\log(n)}$. Then, we apply the second moment method to $\dtvnop{\L(\mu\mid \Seventshort)}{\nu}$ where $\L(\mu\mid \Seventshort)$ is the distribution $\mu$ conditional on $\Seventshort$. 

\textbf{The case of $p=\frac{1}{2}$.}
After conditioning, the stronger bounds from \Cref{prop:fourierexpressionthing} are applicable. If $p=\frac{1}{2}$, this is essentially all we need. 
However, naively plugging the bounds from \Cref{prop:fourierexpressionthing} into \eqref{eq:sumsignedweights} would still not give the right thresholds.
Nonetheless, it turns out that the only obstacle is given by sub-graphs $\alpha$ that have leaves, i.e., vertices of degree one. If we omit the conditioning, it is easy to see that by the symmetries of our model arising for $p = \frac{1}{2}$, the expected signed weight of such $\alpha$ is $0$, since the single edge incident to a leaf is always independent of all the rest.
While conditioning on $\Seventshort$ generally breaks independence between latent vectors, it is not too hard to argue that certain symmetries are preserved even after conditioning. Specifically, since $\Seventshort$ only makes a statement about the absolute value of inner products, we can still independently re-randomize the sign in front of every latent vector. Applying this ``noise operator'' preserves all of the inner products in absolute value, but still zeroes out signed weights of sub-graphs with leaves. After accounting for this fact, we can show that the remaining sum is small enough once we carefully account for all $\alpha \subseteq K_{n, m}$. 

\textbf{The case of arbitrary $p$.} For the more general case $p \neq \frac{1}{2}$, 
the picture is quite a lot more complex. The most important difference to before is that sub-graphs with leaves actually have \emph{non-zero} expected signed weight, even unconditionally. This should be no surprise, as there is a regime where counting \emph{signed wedges} is actually a test that outperforms signed four-cycles. As a consequence, we can work under the stronger (for our purposes less restrictive) assumption that $d \gg \bound + \boundwedges$ and we might hope that under this new assumption, the sum converges, even if including sub-graphs with leaves.

Unfortunately, this is not the case when given only the bounds from \Cref{prop:fourierexpressionthing}, because the leading terms $\leading$ are too large for this purpose. One approach to deal with this situation would be to use similar symmetries as used for $p = \frac{1}{2}$ to eliminate at least some of the leading terms we obtain from \Cref{prop:fourierexpressionthing} (instead of the entire expectation) whenever we encounter a graph with leaves. However, there are two important problems with this idea. First, summing explicitly over all squared expected signed weights would be significantly more involved as before because every individual term in the sum fractures into a sum of multiple individual terms that depend on the exact degree sequence of $\alpha$ and out of which some terms ``zero-out'' while others do not. Secondly (and more importantly), even if we could blindly omit all the leading terms from \Cref{prop:fourierexpressionthing}, then the resulting sum would still not be small enough, simply due to the presence of \emph{single edges} which actually carry a non-negligible weight, even after ignoring the associated leading terms $\leading$. While this problem arises purely from conditioning, it can not simply be solved by arguing about preserved symmetries like we did for leaved graphs in case of $p=\frac{1}{2}$. This is because this time, the problem arises from fluctuations in the \emph{length} of the latent vectors, and not their \emph{orientation}.

However, there is a quite different way of bounding our sum that avoids both problems. To this end, we abandon the last equality in \Cref{eq:sumsignedweights} and instead derive an upper bound on $\dchisq{\mu}{\nu}$ in terms of an \emph{exponential sum}. Concretely, we obtain
\begin{align*} 
    1 + \dchisq{\L{(\mu\mid \Seventshort)}}{\nu} &\le \Expectedsubb{\latentstuff}{ \exp\left( m \sum_{\emptyset \neq \alpha \subseteq R} q^{2|\alpha| } \frac{\Expected{ \SWnum{K_{1,\alpha}}{1} } \Expected{ \SWnum{K_{1,\alpha}}{2} }}{(p(1-p))^{|\alpha|}} \right) }
\end{align*}
where for $k \in \{0, 1\}$ we define $
    \Expectednop{\SWnum{K_{1,\alpha}}{k}} \coloneqq \Expectedsubnop{\mathbf{x}}{ \prod_{j \in \alpha} (\sigma(\langle \mathbf{x}_{j}^{(k)}, \mathbf{x} \rangle) - p) \mid \xrnum{k}}
$. In other words, we express $\dchisq{\mu}{\nu}$ as an exponential sum over \emph{conditional expected signed weights of stars}. Plugging in the bounds from \Cref{prop:fourierexpressionthing}, some fortunate simplifications occur already in the exponent.

\textbf{Omitting large stars and higher order terms.}
First of all, it turns out the sum in the exponent is dominated only by relatively few of its leading terms. Concretely our bounds on the \emph{conditional expectation} $\Expectednop{\SWnum{K_{1,\alpha}}{k}}$ already allow us to omit all terms that correspond to $|\alpha| \ge \log(m)$. Moreover, after splitting each term $\Expectednop{\SWnum{K_{1,\alpha}}{k}}$ into leading and remainder terms as in \Cref{prop:fourierexpressionthing}, we obtain the expression $\Expectednop{ \SWnum{K_{1,\alpha}}{1} } \Expectednop{ \SWnum{K_{1,\alpha}}{2} } = \leadingnum{1}\leadingnum{2} + \leadingnum{1}\restnum{2} + \leadingnum{2}\restnum{1} + \restnum{1}\restnum{2}$, and we can observe that the sum over all remainder terms of the form $\restnum{1}\restnum{2}$ can be omitted as well. 

\textbf{Bounding the remaining exponential terms via hypercontractivity} After these simplifications, each term in the exponent can explicitly be expressed as a polynomial in the right-sided latent vectors $\xrnum{1}, \xrnum{2}$ we conditioned on. 
After we split the exponential sum into a product via Cauchy-Schwarz, the exponent appearing in each term $\leadingnum{1}$ or $\leadingnum{2}$ from \Cref{prop:fourierexpressionthing} for some small set $\alpha \subseteq R$. Using the explicit representation of $\leading$ from \Cref{prop:fourierexpressionthing}, we can explicitly control the variance of the polynomials in the exponents over the randomness of the latent vectors $\xr \sim \mathcal{N}(0, \mathbf{I})$. 
Since the integrand (hidden in the expectation over $\xrnum{1}, \xrnum{2}$) is non-negative due to the exponential, we can essentially remove the conditioning on $\Seventshort$ while still taking advantage of the fact that requiring $\xr \in \Seventshort$ implies an explicit (deterministic) upper bound on the integrand. Since the $\xr$ are standard Gaussian, we can use tail bounds obtained from Gaussian hypercontractivity to integrate each term over the randomness of $\xrnum{1}, \xrnum{2}$ up to the explicit upper bound obtained from conditioning on $\Seventshort$. 
This argument shows that each exponential is essentially of order $1 + C\sigma$, where $\sigma$ denotes the variance of the polynomial in the exponent. 
We remark that this approach forms an extension of the ideas first considered in \cite[Section 9]{brennan2021finetti}. The difference is that we work with different, and at the same time higher-degree polynomials than considered before.

\textbf{Bounding the exponential sum over all edges}
The sum resulting from this is almost small enough. The only obstacle left is formed by the exponential sum over humble \emph{single edges} for which even the higher order terms $\restnum{1}, \restnum{2}$ carry non-negligible contributions. Fortunately, due to the non-negativity of the exponential, we can remove the conditioning on $\Seventshort$ and expand the result as a Taylor series. Once again, we now end up with an (infinite) sum over expected signed weights, similar to the one considered in the beginning. However, since we already eliminated all patterns except for single edges from the exponent, we can now ensure that all expected signed weights that still appear in the sum correspond only to \emph{disjoint unions of stars}. For such simple structures, the expected signed weight can be bounded explicitly, even unconditionally. 

What is important to note about this final step is that the purpose of switching from a sum to an exponential only to get back at a (even larger) sum is helpful because it allows us to remove the conditioning on $\Sevent$ which previously had caused the problems related to single edges. As explained before, this would not be as easy if we summed over all edges before switching to the exponential.

\textbf{Known vs unknown masks.}
So far we assumed that we compare $\matnoiseedge$ to $\gaussmat$, i.e., that we are in the case of a unknown mask. Elegantly, the exact same proof also applies to the setting of a known mask $\mathbf{M}$. To see how, we note that we have so far shown that 
$$
    \dtv{\L(\mu \mid \Seventshort) } {\nu} =  \dtv{\Expectedsub{\mathbf{M}}{\mathcal{L}(\mu \mid \mathbf{M}, \Seventshort)}}{\nu}  = o(1),
$$
where $\Expectedsub{\mathbf{M}}{\mathcal{L}(\mu \mid \mathbf{M}, \Seventshort)}$ is the distribution $\mu$ conditioned on a fixed mask $\mathbf{M}$ and on the event $\Seventshort$, so $\L(\mu \mid \Seventshort) = \Expectedsub{\mathbf{M}}{\mathcal{L}(\mu \mid \mathbf{M}, \Seventshort)}$. On the other hand, if we pulled the expectation over $\mathbf{M}$ out of the expression and could instead show
$$
    \Expectedsub{\mathbf{M}}{ \dtv{\mathcal{L}(\mu \mid \mathbf{M}, \Seventshort)}{\nu} } = o(1),
$$ then $\dtv{\matnoiseknown}{\gaussmat}= o(1)$ would follow with probability $1 - o(1)$ over the choice of $\mathbf{M}$, by Markov's inequality. Remarkably, if we apply the second moment method to $\dtv{\mathcal{L}(\mu \mid \mathbf{M}, \Seventshort)}{\nu}$ and then take the expectation over $\mathbf{M}$, all that changes is that we only apply \emph{one expectation} over $\mathbf{M}$ instead of the expectation over \emph{two independent replicas} of $\mathbf{M}$ as before. This reflects precisely in replacing the factor of $q^{2|\alpha|}$ in front of every term by a factor of $q^{|\alpha|}$ while leaving the remaining proof untouched. 

\subsection{Bounding signed weights after cancellations in Fourier space}

It remains to outline the essential core ideas underlying all of our arguments: the proof of \Cref{prop:fourierexpressionthing}. Inspired by the recent ideas presented in \cite{baguley2025testing} for bounding the expected signed weight of certain pattern graphs in terms of joint cumulants via modified Edgeworth expansions, we adopt a similar Fourier-theoretic viewpoint. 

\newcommand{\zbeta}[1]{\mathbf{z}_{#1}}
\newcommand{\covmatbeta}[1]{\boldsymbol{\Sigma}_{#1}}
To this end, we wish to bound the expected signed weight $\Expectedsub{\mathbf{x}}{ \SW{K_{1,\alpha}} \mid \xr } $ of a star $K_{1,\alpha}$ with leaves given by $\alpha \subseteq R$ conditional on some fixed, right-sided latent vectors given by $\xr$. For this, we essentially wish to understand how the joint distribution of the edges in $\alpha$ deviates from independence. To understand the dependencies that are present, we start from the observation that the inner products 
$
    \zbeta{\alpha} \coloneqq \tfrac{1}{\sqrt{d}}(
        \langle \mathbf{x}_1, \mathbf{x} \rangle,
        \langle \mathbf{x}_2, \mathbf{x} \rangle,
        \ldots,
        \langle \mathbf{x}_{\ell}, \mathbf{x} \rangle)^\top
$ for $\alpha = \{1, \ldots, \ell\}$ and fixed latent vectors $(\mathbf{x}_u)_{u \in \alpha}$ are jointly Gaussian with covariance matrix $\covmatbeta{\alpha}$ given as $$
    (\covmatbeta{\alpha})_{u,v} = \tfrac{1}{d} \langle\mathbf{x}_u, \mathbf{x}_v \rangle.
$$
To bound $\Expectedsub{\mathbf{x}}{ \SW{K_{1,\alpha}} \mid \xr } $ we would like to integrate the above distribution component-wise up to the connection threshold $\tau$ and then study its deviations from a random vector with i.i.d. $\Bern{p}$ entries by evaluating the alternating sum
\begin{align}\label{eq:swexpansion}
    \Expectedsub{\mathbf{x}}{ \SW{K_{1,\alpha}} \mid \xr } = \mathbb{E}_{\mathbf{x}} \bigg[ {\prod_{\ell \in \alpha} ( \sigma(\langle \mathbf{x}_\ell, \mathbf{x} \rangle) - p ) } \bigg]  = \sum_{\beta \subseteq \alpha} (-1)^{|\alpha \setminus \beta|} p^{|\alpha \setminus \beta|} \mathbb{E}_{\mathbf{x}} \bigg[  {\prod_{\ell \in \beta} \sigma(\langle \mathbf{x}_\ell, \mathbf{x} \rangle)  } \bigg].
\end{align}
To this end, we essentially wish to understand the probability $\mathbb{E}_{\mathbf{x}} \big[  {\prod_{\ell \in \beta} \sigma(\langle \mathbf{x}_\ell, \mathbf{x} \rangle)  } \big]$ for all edge sub-graphs $\beta \subseteq \alpha$. However, to get a good bound on the entire sum, it does not nearly suffice to compute each $\mathbb{E}_{\mathbf{x}} \big[  {\prod_{\ell \in \beta} \sigma(\langle \mathbf{x}_\ell, \mathbf{x} \rangle)  } \big]$ up to the asymptotically correct order. The reason for this is that \emph{cancellations} occur in the above expression that render the entire sum \emph{much smaller} than the absolute value of its largest term. Because of this, it is actually very important to precisely understand many of the \emph{higher-order terms} that make up $\mathbb{E}_{\mathbf{x}} \big[  {\prod_{\ell \in \beta} \sigma(\langle \mathbf{x}_\ell, \mathbf{x} \rangle)  } \big]$. This is even more true if we even wish to get an explicit representation of $\mathbb{E}_{\mathbf{x}} \big[  {\prod_{\ell \in \beta} \sigma(\langle \mathbf{x}_\ell, \mathbf{x} \rangle)  } \big]$ as a polynomial in the fixed latent vectors like in \Cref{prop:fourierexpressionthing}.

\textbf{Expressing perturbations from a ground state.}
To start understanding dependencies, we wish to compare $\zbeta{\alpha}$ to an independent ground state.
As the intuition would suggest, the appropriate choice for this ground state is essentially a standard Gaussian vector $\zbeta{\emptyset} \in \mathbb{R}^{|\alpha|}$. We say `essentially', because we actually need to slightly correct the variance along each dimension, simply to ensure that $\Pr{\zbeta{\emptyset}(j) \le \tau} = p$ for all $j$ such that we match the ground-truth density $p$ after integrating. Therefore, we set our ground state as $\zbeta{\emptyset} \sim \mathcal{N}(0, \sighat^2\mathbf{I})$ where $\sighat$ is the \emph{reference variance} chosen as a function of $p$.

\textbf{Intermediate states.}
To better represent every single term in \eqref{eq:swexpansion}, it is convenient to define some more random vectors with varying levels of dependence in between the very dependent setting given by $\zbeta{\alpha}$ and the independent ground state $\zbeta{\emptyset}$. Concretely, for every $\beta \subseteq \alpha$, we define an \emph{intermediate state} $\zbeta{\beta}$ as a Gaussian vector in $\mathbb{R}^{|\alpha|}$ with mean $0$ and covariance $\covmatbeta{\beta}$ given by 
$$
    (\covmatbeta{\beta})_{u, v} = \begin{cases}
        \frac{1}{d}\langle \mathbf{x}_{u}, \mathbf{x}_{u}\rangle & \text{if } u, v \in \beta \\
        \sighat^2 & \text{if } u = v \text{ and } u, v \notin \beta \\
        0 & \text{otherwise}.
    \end{cases}
$$ Accordingly, $\zbeta{\beta}$ behaves like $\zbeta{\alpha}$ when restricted to the the entries in $\beta$, but has independent components like $\zbeta{\emptyset}$ otherwise. 
The big advantage defining these intermediate states is that now, 
$
    p^{|\alpha \setminus \beta|} \ \Expectedsub{\mathbf{x}}{{\prod}_{\ell \in \beta} \sigma(\langle \mathbf{x}_\ell, \mathbf{x} \rangle) } = \Pr{ \mathlarger\cap_{e \in \alpha} \{ \mathbf{z}_\beta(e)  \le \tau \} }
$, so our expression for $\Expectedsub{\mathbf{x}}{ \SW{K_{1,\alpha}} \mid \xr }$ turns into the succinct expression 
\begin{align*}
    \Expectedsub{\mathbf{x}}{ \SW{K_{1,\alpha}} \mid \xr } = \sum_{\beta \subseteq \alpha} (-1)^{|\alpha \setminus \beta|} \ \Pr{ \mathlarger\cap_{e \in \alpha} \{ \mathbf{z}_\beta(e)  \le \tau \} }.
\end{align*}

\textbf{Switching to Fourier space.}
It turns out that it is much easier to understand the cancellations happening in the above sum after applying a Fourier transform all our states $\zbeta{\beta}$. The reason for this is that the Gaussian distribution has a succinct Fourier representation (a.k.a. characteristic function) that explicitly allows us to express the dependencies between the entries of $\zbeta{\alpha}$ and relate them to the ground state.

The first step towards this is to express $\Expectedsub{\mathbf{x}}{ \SW{K_{1,\alpha}} \mid \xr }$ in terms of the \emph{characteristic functions} $\CF{\beta}$ (i.e. the Fourier transforms) of each of the states $\zbeta{\beta}$. The inversion theorem allows us to express
\begin{align*}
    \Expectedsub{\mathbf{x}}{ \SW{K_{1,\alpha}} \mid \xr } &= \frac{1}{(2\pi)^{|\alpha|}} \int_{\T} e^{-\trp{i\mathbf{t}}{\mathbf{x}}_{\tau}} \left( \sum_{\beta \subseteq \alpha} (-1)^{|\alpha \setminus \beta|}  \hspace{.1cm}\CF{\beta} \right) \prod_{e \in \mathbf{t}} \frac{\onebutnotactuallyone}{i\mathbf{t}(e)} \d \mathbf{t},
\end{align*} where $\onebutnotactuallyone$ is a limiting object in $h \rightarrow \infty$ that mostly behaves like 1 for our purposes, where $\mathbf{x}_{\tau}$ is a vector that contains $\tau$ in every coordinate, and where $\T = [-\infty, \infty]^{|\alpha|}$. 

\textbf{Cancellations after expanding characteristic functions.}
Now, we switched from an alternating sum of probabilities to an alternating sum of characteristic functions. 
The advantage of this is that it is now much easier to compare every term to the ground state given by $\CF{\emptyset}$ by considering the ratio $\CF{\beta}/\CF{\emptyset}$ in the sense that
\begin{align*}
    \left( \sum_{\beta \subseteq \alpha} (-1)^{|\alpha \setminus \beta|}  \hspace{.1cm}\CF{\beta} \right) = \CF{\emptyset} \left( \sum_{\beta \subseteq \alpha} (-1)^{|\alpha \setminus \beta|}  \hspace{.1cm}\frac{\CF{\beta}}{\CF{\emptyset}} \right) = \CF{\emptyset} \left( \sum_{\beta \subseteq \alpha} (-1)^{|\alpha \setminus \beta|}  \hspace{.1cm}\exp\left( - \frac{1}{2}\mathbf{t}^\top \! (\boldsymbol{\Sigma}_\beta - \sighat^2\mathbf{I}) \ \mathbf{t} \right) \right).
\end{align*}
Then, applying a Taylor expansion to the exponential in the sum, we get a sum of the form 
\begin{align*}
    \sum_{\beta \subseteq \alpha} (-1)^{|\alpha \setminus \beta|}  \hspace{.1cm}\exp\left( - \frac{1}{2}\mathbf{t}^\top \! (\boldsymbol{\Sigma}_\beta - \sighat^2\mathbf{I}) \ \mathbf{t} \right) &= \sum_{k=0}^\infty  \frac{(-1)^k}{2^k k!} \sum_{r_1, \ldots r_k \in \alpha \times \alpha} \sum_{\beta \subseteq \alpha} (-1)^{|\alpha \setminus \beta|} \phiterm{\beta}\\
    \text{with } \phiterm{\beta} &= \prod_{j = 1}^k (\boldsymbol{\Sigma}_\beta - \sighat^2\mathbf{I})_{r_j(1), r_j(2)} \mathbf{t}(r_j(1))\mathbf{t}(r_j(2)).
\end{align*} Now, if we recall the definition of the support $\supp(r_1, \ldots, r_k)$ as the set of elements $\gamma \subseteq \alpha$ covered by at least one of the $2$-tuples $r_1, \ldots, r_k$, then it turns out that every term corresponding to $\supp(r_1, \ldots, r_k) \neq \alpha$ cancels out. This phenomenon arises due to the special structure of the $\phiterm{\beta}$ resulting from our definition of the intermediate states $\zbeta{\beta}$. 

We refer to \Cref{sec:proofofprop} for the details underlying the above phenomena. The upshot is that the cancellation phenomena add the so-called \emph{coverage constraint} $\supp(r_1, \ldots, r_k) = \alpha$ to the sum over $r_1, \ldots, r_k$. In particular, this has the consequence that \emph{all terms corresponding to $k < \lceil |\alpha|/2 \rceil$ vanish} because satisfying the coverage constraint $\supp(r_1, \ldots, r_k) = \alpha$ is impossible.

\textbf{Reversing the Fourier transform.}
After our cancellations, the alternating sum inside of the integral evaluates as 
\begin{align*}
    \left( \sum_{\beta \subseteq \alpha} (-1)^{|\alpha \setminus \beta|}  \hspace{.1cm}\CF{\beta} \right) = \CF{\emptyset} \left( \frac{(-1)^
    \ell}{2^\ell \ell!} \sum_{ \substack{ r_1, \ldots r_\ell \in \alpha \times \alpha \\ \supp(r_1, \ldots, r_k ) = \alpha}} (-1)^{|\alpha \setminus \beta|} \phitermell{\alpha} + \resthat \right)
\end{align*} for $\ell \coloneqq \lceil |\alpha|/2 \rceil$ and remainder terms grouped together in $\resthat$. Due to the simple expression for $\CF{\emptyset}$ and the explicit expression of the $\phiterm{\beta}$ the integral over the first term above evaluates explicitly as a polynomial involving derivatives of the PDF of our ground state $\zbeta{\emptyset}$, which is equal to $\leading$ from \Cref{prop:fourierexpressionthing}. After carefully controlling the integral over the absolute value of all the remainder terms in $\resthat$, we further obtain an explicit bound on $|\rest|$ which is of lower order. For this step, we can avoid complicated sums over Hermite polynomials by exploiting some spectral properties of the matrix $\boldsymbol{\Sigma}_\beta - \sighat^2\mathbf{I}$ implied by conditioning on $\Seventshort$.

\section{Outlook}

Let us now give some ideas for possible extensions of this work. The most important thing to mention is that our main technical result \Cref{prop:fourierexpressionthing} can be extended to the non-bipartite case to give meaningful bounds at least for small $\alpha$ of size poly-logarithmic in $n$. On a technical level, this requires some further work in the sense that we cannot limit ourselves to work with characteristic functions of the Gaussian distribution anymore. However, it turns out that similar cancellations in Fourier space as described above occur if we split the characteristic function into the product of a ground state corresponding to independent edges and a remainder, and then apply a series expansions to this remainder but not to the ground state. The terms in these expansions can be expressed in terms of \emph{joint cumulants} whose properties enable similar cancellation phenomena as observed here, while the remainder terms and the resulting errors can be bounded explicitly at the expense of poly-logarithmic factors as long as $|\alpha| \le \text{polylog}(n)$. The resulting bounds on $\Expected{\SW{\alpha}}$ decay in powers of $1/\sqrt{d}$ while the exponent is proportional to $|\alpha|$ instead of $|V(\alpha)|$. This, in particular, improves the bounds given in \cite{bangachev2024fourier} and might enable new insights into statistical testing and estimation involving RGGs from the lens of \emph{low-degree polynomials}.
Moreover, the techniques underlying our \Cref{prop:fourierexpressionthing} are not limited to Gaussian latent vectors, but can be expected to continue to work on other product spaces like the torus or in case of an anisotropic Gaussian distribution of latent vectors.

Finally, it would be very interesting to see whether the results can be extended to the sparse case where $p=o(1)$. After all, finding the right information-theoretic thresholds for distinguishing RGGs from \Erdos-\Renyi~ graphs for all $p = o(1)$ is a prominent open problem that has attracted quite some attention over the last years \cite{bangachev2024fourier, liu2022testing}. The bipartite setting is thereby an important special case. Since our techniques yield tight testing thresholds for a different notion of sparsity ($q = o(1)$ rather than $p = o(1)$), it would be interesting to see whether they also help us understand the case of $p=o(1)$.


\section{Preliminaries}

\noindent We use the following auxiliary statements.

\begin{lemma}[Bound on Hermite polynomials, Inequality (1.2) in \cite{van1990new}]\label{lem:hermite}
    Let $\He_k(x) \coloneqq  (-1)^k e^{\frac{x^2}{2}} \frac{\d^k}{\d x^k} e^{-\frac{x^2}{2}}$. Then for every $k \in \mathbb{N}$ and $x \in \mathbb{C}$, 
    $$
        |\He_k(x)| \le 2^{k/2} \sqrt{k!} e^{\sqrt{2k}|x|}.
    $$
\end{lemma}

\begin{lemma}[Strirling's approximation for the Gamma function]\label{lem:stirling}
    Let $\Gamma(s) \coloneqq \int_{0}^{\infty} x^{s-1}e^{-x} \d x$. Then, there is a constant $C > 0$ such that for every $s \in \mathbb{R}, s \ge 1$, $
        \Gamma(s) \le (C s)^{s}.
    $
\end{lemma}

\begin{lemma}[Gaussian Hypercontractivity, \cite{o2014analysis}, Theorem 9.21]\label{lem:hypercont}
    Let $q \ge 2$ and let $f$ be an arbitrary polynomial of degree $k \in \mathbb{N}$ in i.i.d. standard Gaussian inputs $\mathbf{x}$. Then,
    $$
        \Expectedsub{\mathbf{x}}{|f(\mathbf{x})|^q} \le (q-1)^{qk/2} \Expectedsub{\mathbf{x}}{f(\mathbf{x})^2}^{q/2}.
    $$
\end{lemma}

\begin{theorem}{Berry-Esseen theorem, \cite{klenke2008probability}}\label{thm:berryesseen}
    Let $X_1, \ldots, X_d$ be i.i.d. random variables with $\Expected{X_1} = 0$, $\Expected{X_1^2} = \sigma^2$ and $\Expected{X_1^3} < \infty$. Let $X = \frac{1}{\sqrt{d\sigma^2}} \sum_{i = 1}^d X_i$ and let $\Phi$ denote the CDF of the standard Gaussian distribution. 
    Then, 
    \begin{align*}
        \sup_x | \Pr{X \le x} - \Phi(x) | \le \frac{0.8\Expected{X_1^3}}{\sigma^3 \sqrt{d}}.
    \end{align*}
\end{theorem}

\begin{theorem}[Fourier inversion, Theorem 3, Chapter XV.3 in \cite{feller1991introduction}]\label{thm:inversion}
    If $\varphi$ is the characteristic function of a random variable $X$ in $\mathbb{R}^k$ such that $|\varphi(\mathbf{t})|$ is integrable, then the density $f$ of $X$ is given by
    \begin{align*}
        f(\mathbf{x}) = \frac{1}{(2\pi)^k} \int_{[-\infty, \infty]^k} e^{i\trp{\mathbf{t}}{\mathbf{x}}} \varphi(\mathbf{t}) \d \mathbf{t}
    \end{align*}
\end{theorem}

We will also use the following lemma that expresses the difference of the CDF of a standard Gaussian with that of variance $\sigma$. 

\begin{lemma}\label{lem:divergenceofgaussianswithdifferentvariances}
    Denote by $\text{sgn}(x)$ the sign of a number $x \in \mathbb{R}, x\neq 0$.
    Given any $0 < \sigma < 1$, let further $X \sim \gauss{1}$ and $Z \sim \gauss{\sigma^2}$ with CDFs $\Phi$ and $\Phi_\sigma$, respectively. Then for every $x \in \mathbb{R}, x \neq 0$, and assuming that $|1 - \sigma^2|$ is small enough (as a function of $x$), there are constants $c, C > 0$ independent on $\sigma$ (but dependent on $x$) such that
    $$
        c \ \mathrm{sgn}(x) \cdot (1 - \sigma^2) \le \Phi_\sigma(x) -  \Phi(x) \le C \ \mathrm{sgn}(x) \cdot (1 - \sigma^2).
    $$
\end{lemma}
\begin{proof}
    Fourier inversion yields that 
    $$
        \Phi_\sigma(x) = \frac{1}{2\pi} \int_{-\infty}^x \int_{-\infty}^\infty e^{-i\! \mathrm{t} \! x} e^{-\frac{1}{2}\sigma \mathrm{t}^2} \d \mathrm{t}.
    $$
    From a Taylor expansion, we get 
    $$
        e^{-\frac{1}{2}\sigma \mathrm{t}^2} = e^{-\frac{1}{2}\mathrm{t}^2} e^{\frac{1}{2}\mathrm{t}^2 - \frac{1}{2}\sigma \mathrm{t}^2} = e^{-\frac{1}{2}\mathrm{t}^2} \left( 1 + \sum_{k=1}^\infty \frac{1}{2^k k!} (1 - \sigma^2)^k\mathrm{t}^{2k} \right) = e^{-\frac{1}{2}t^2} \left( 1 + \sum_{k=1}^\infty \frac{(-1)^k}{2^k k!} (1 - \sigma^2)^k (i\mathrm{t})^{2k} \right).
    $$
    Thus,
    \begin{align*}
        \Phi_\sigma(x) &= \frac{1}{2\pi}\int_{-\infty}^x \int_{-\infty}^\infty e^{-i\! \mathrm{t} \! x} e^{-\frac{1}{2} \mathrm{t}^2} \d \mathrm{t} + \frac{1}{2\pi} \sum_{k=1}^\infty \frac{(-1)^k}{2^k k!} \int_{-\infty}^x \int_{-\infty}^\infty e^{-i\! \mathrm{t} \! x} (1 - \sigma^2)^k (i\mathrm{t})^{2k} \d \mathrm{t} \\
        &= \Phi(x) + \sum_{k=1}^\infty \frac{(-1)^k}{2^kk!} ( 1 - \sigma^2)^k \phi^{(2k-1)} (x).
    \end{align*} where $\phi^{(k)}$ is the $k$-th derivative of the standard Gaussian density and where the exchange of integral and sum is justified by dominated convergence. We remark that exchanging the integrals and the infinite sum is justified here, as we make formally clear in \Cref{sec:proofofprop}, where similar arguments are presented for the more general, multivariate case. 
    
    Now, note that due to $x \neq 0$, the first term (corresponding to $k = 1$) is non-zero as $\phi^{(1)}(x) \neq 0$. Moreover, $\mathrm{sgn}(\phi^{(1)}(x)) = -\mathrm{sgn}(x)$, so the sign of the first term is $\mathrm{sgn}(x)$. Using further that $|\phi^{(2k-1)}(x)| = |\He_{2k-1}(x)|e^{-\frac{1}{2}x^2} \le (Ck)^k$ by \Cref{lem:hermite}, we get that the sum of all terms starting at $k = 2$ is at most $C' (1 - \sigma^2)^2$ in absolute value for some constant $C'$ once $|1 - \sigma|$ is small enough. For sufficiently small $|1 - \sigma|$, the first term then dominates and the lemma follows. \qedhere 

\end{proof}

Finally we will make use of the following lemma, stating that the reference variance $\sighat$ is close to one.
\begin{lemma}\label{lem:sigbound}
    Let any fixed $p \in (0,1)$ be given. 
    Define $\sighat =\sighat(p, d)$\footnote{note that $\sighat$ depends on $d$ since $\tau$ used in the following definition depends on $d$} such that for a random variable $Z \sim \mathcal{N}(0, \sighat^2)$,
    $
        \Pr{Z \le \tau} = p.
    $ Then, there is a constant $C$ independent of $d$ such that for all sufficiently large $d$,
    $
        |\sighat^2 - 1| \le \frac{C}{\sqrt{d}}.
    $
\end{lemma}
\begin{proof}
    We wish to show that 
    $
        \Prnop{ d^{-1/2}\langle \mathbf{x}, \mathbf{y} \rangle \le \tau} = \Phi_{\sighat}(\tau)
    $ for $\mathbf{x}, \mathbf{y} \sim \mathcal{N}(0, \mathbf{I}_d)$ and some $\sighat$ with $|\sighat^2 - 1| \le C/\sqrt{d}$. To this end, we use \Cref{thm:berryesseen} to get that 
    $
        | \Prnop{ d^{-1/2} \langle \mathbf{x}, \mathbf{y} \rangle \le \tau} - \Phi(\tau)  | \le C/\sqrt{d}.
    $ By \Cref{lem:divergenceofgaussianswithdifferentvariances} and for $d$ large enough, we can choose a $\sighat^2 $ with $|\sighat^2 - 1| \le C'/\sqrt{d}$ to counterbalance the above differences whenever $\tau$ is bounded away from $0$ for all sufficiently large $d$. Note that this is given whenever $p \neq \frac{1}{2}$. For $p = \frac{1}{2}$, the statement of the lemma is trivial since $\tau = 0$ for all $d$, so $\sighat = 1$.
\end{proof}

\section{Information-Theoretic Hardness}\label{sec:info}

We can use the methods described previously not only for results on low-degree hardness, but also for deriving information-theoretic lower bounds. Concretely, 
this section is devoted to proving the following
\begin{theorem}[Information-theoretic hardness for unknown masks]\label{thm:hardnessknown}
    Consider any fixed $p \in (0,1)$. Then 
    $$
        \dtv{\matnoiseedge}{ \ \gaussmat} = o(1)
    $$ under one of the following conditions. 
    \begin{enumerate}
        \item[(i)] $d \gg \bound$ and $d \gg \boundwedges$.
        \item[(ii)] $d \gg \bound$ and $p = 1/2$.
    \end{enumerate}
\end{theorem}

\begin{theorem}[Information-theoretic hardness for known masks]\label{thm:hardnessunknown}
    Consider any fixed $p \in (0,1)$. Then with probability $1- o(1)$ over the choice of $\ \mathbf{M} \sim \gaussmatarg{n,m,q}$
    $$
        \dtv{ \matnoiseknown}{ \ \gaussmat} = o(1)
    $$ under one of the following conditions 
    \begin{enumerate}
        \item[(i)] $d \gg \boundknown$ and $d \gg \boundwedgesknown$.
        \item[(ii)] $d \gg \boundknown$ and $p = 1/2$.
    \end{enumerate}
\end{theorem}

The proof of both theorems uses a conditional second moment method. Let us for now focus on outlining the techniques used for $\matnoiseedge$, we will later see how to adapt them to $\matnoiseknown$. To this end, we denote by $\mu$ the distribution of $\matnoiseedge$, and we denote by $\nu$ the distribution of $\gaussmat$. To bound the total variation distance between the two models, we use the triangle inequality to bound
$$
    \dtv{\mu}{\nu} \le \Prnop{\overline{\Seventshort}} + \dtv{\mathcal{L}(\mu \mid \Seventshort)}{\nu},
$$
where $\mathcal{L}(\mu \mid \Seventshort)$ denotes the distribution $\mu$ conditional on the latent information being in some 'good' event $\Seventshort \in \sigma( \mathbf{X}_{\text{R}})$ to be specified later.
Since $\Seventshort$ will be chosen such that $\Prnop{\overline{\Seventshort}} = o(1)$, it suffices to consider the second term. To this end, we pass from total variation to $\chi^2$-divergence, i.e., 
$$
    2\dtvnop{\mathcal{L}(\mu \mid \Seventshort)}{\nu}^2 \le \dchisq{\mathcal{L}(\mu \mid \Seventshort)}{\nu} = \Expectedsub{\xi \sim \nu}{ \left( \frac{\mathrm{d} \mathcal{L}(\mu \mid \Seventshort)}{\mathrm{d} \nu}(\xi) \right)^2 } - 1.
$$ Now, using that $\mathcal{L}(\mu \mid \Seventshort)$ is a mixture distribution in terms of some latent information $\phi$, we can express $\mathcal{L}(\mu \mid \Seventshort) = \Expectedsub{\phi}{ \mu_\phi}$ for some distribution $\mu_\phi$ and then replace the squared expectation by an expectation over two independent copies of the latent randomness $\phi \coloneqq (\xr, \mathbf{M})$ where $\xr \in \mathbb{R}^{n \times d}$ is the matrix of latent vectors in $R$, and $\mathbf{M} \in \{0, 1\}^{n \times m}$ is the mask used to re-randomize the entries in $W$. Concretely, we get $$
    \dchisq{\mathcal{L}(\mu \mid S)}{\nu} = \Expectedsubb{\xi \sim \nu}{ \left( \Expectedsubb{\phi \sim \varrho}{\frac{\mathrm{d} \mu_{\phi}}{\mathrm{d} \nu}(\xi)} \right)^2 } - 1 =  \Expectedsubb{\phi^{(1)}, \phi^{(2)} \sim \varrho \otimes \varrho}{ \Expectedsubb{\xi \sim \nu}{ \frac{\mathrm{d} \mu_{\phi^{(1)}}}{\mathrm{d} \nu}(\xi) \frac{\mathrm{d} \mu_{\phi^{(2)}}}{\mathrm{d} \nu}(\xi) }  } - 1
$$
where $\varrho$ denotes the law of $\phi$ conditional on $\Seventshort$. 
Since $\nu$ is a sufficiently simple distribution, the inner expectation evaluates explicitly and allows us to simplify the above expression. Concretely, denoting the columns of $\xr, \xl$ by $(\mathbf{x}_u)_{u \in R}$ and  $(\mathbf{x}_u)_{u \in L}$, respectively, we use a similar calculation as in \cite{liuracz} and express the probability of drawing a concrete matrix $\xi \in \mathbb{R}^{n \times m}$ under the conditional distribution $\mu_{\phi}$ as 
\newcommand{\Expectedsubbsmall}[2]{\underset{#1}{\mathbb{E}} \Bigg[ #2 \Bigg]}
\newcommand{\sigmanum}[1]{\sigma^{(#1)}_{u,v}}
\newcommand{\sigmanumuv}[3]{\sigma^{(#1)}_{#2,#3}}
\begin{align*}
    \frac{\d \mu_{\phi}}{\d \nu}(\xi) &= \underset{(\mathbf{x}_u)_{u \in L}}{\mathbb{E}} \left[ \ { \prod_{ \substack{u \in L}} \prod_{ \substack{v \in R}} \frac{\sigma_{\phi}(\langle \mathbf{x}_u, \mathbf{x}_v \rangle)^{\xi_{u,v}} \left(1 -\sigma_{\phi}(\langle \mathbf{x}_u, \mathbf{x}_v \rangle)\right)^{1 - \xi_{u,v}} }{p^{\xi_{u,v}}(1-p)^{1-\xi_{u,v}}} }  \right]
\text{ with }
    \sigma_{\phi}(\langle \mathbf{x}_u, \mathbf{x}_v \rangle) \coloneqq \begin{cases}
        p & \text{if } \mathbf{M}_{u,v} = 0\\
        \sigma(\langle \mathbf{x}_u, \mathbf{x}_v \rangle) & \text{otherwise}.
    \end{cases} 
\end{align*} 
To plug this back into our expression for $\dchisq{\mathcal{L}(\mu \mid \Seventshort)}{\nu}$, we express our latent randomness as $\phi^{(1)} \coloneqq (\xrnum{1}, \masknum{1})$ and $\phi^{(2)} \coloneqq (\xrnum{2}, \masknum{2})$, and---in order to simplify notation---we abbreviate $\sigmanum{1} = \sigma_{\phi^{(1)}}(\langle \mathbf{x}_u^{(1)}, \mathbf{x}_v^{(1)} \rangle)$ and $\sigmanum{2} = \\ \sigma_{\phi^{(2)}}(\langle \mathbf{x}_u^{(2)}, \mathbf{x}_v^{(2)} \rangle)$. Then, the inner expectation evaluates as \begin{align*} 
    &\Expectedsubb{\xi \sim \nu}{ \frac{\mathrm{d} \mu_{\phi^{(1)}}}{\mathrm{d} \nu}(\xi) \frac{\mathrm{d} \mu_{\phi^{(2)}}}{\mathrm{d} \nu}(\xi) } = \Expectedsubb{(\mathbf{x}_{u}^{(1)}, \mathbf{x}_{u}^{(2)} )_{u \in L}}{ \prod_{u \in L}\prod_{ v \in R} \ \Expectedsubb{ \xi_{u,v} \sim \Bern{p}}{ \left( \frac{1}{p^2} \sigmanum{1} \sigmanum{2} \right)^{\xi_{u,v}} \left(\frac{1}{(1-p)^2}  (1 - \sigmanum{1})(1 - \sigmanum{2} ) \right)^{1 - \xi_{u,v}} } }\\
    & \hspace{1.5cm} = \Expectedsubb{\mathbf{x}_{1}^{(1)}, \mathbf{x}_{1}^{(2)}}{ \prod_{v \in R}  \frac{1}{p} \sigmanumuv{1}{1}{v} \sigmanumuv{2}{1}{v} +  \frac{1}{1-p}  (1 - \sigmanumuv{1}{1}{v})(1 - \sigmanumuv{2}{1}{v} )  }^{m}\\
    & \hspace{1.5cm} = \Expectedsubb{\mathbf{x}_{1}^{(1)}, \mathbf{x}_{1}^{(2)} }{ \prod_{v \in R} \left( 1 + \frac{1}{p(1-p)} (\sigmanumuv{1}{1}{v} - p) (\sigmanumuv{2}{1}{v} - p) \right) }^{m}.
\end{align*}
where we used that all the $(\mathbf{x}_{u}^{(1)}, \mathbf{x}_{u}^{(2)} )_{u \in L}$ are i.i.d. such that the product over $u \in L$ can be pulled out of the expectation and results in raising the entire expression to the power of $m = |L|$.
Expanding the above product, we get 
\begin{align*}
    \Expectedsubb{\xi \sim \nu}{ \frac{\mathrm{d} \mu_{\phi^{(1)}}}{\mathrm{d} \nu}(\xi) \frac{\mathrm{d} \mu_{\phi^{(2)}}}{\mathrm{d} \nu}(\xi) }=  \left( 1 + \sum_{\emptyset \neq \alpha \subseteq R} \frac{1}{(p(1-p))^{|\alpha|}} \Expectedsub{\mathbf{x}_1^{(1)}}{ \prod_{v \in \alpha} (\sigmanumuv{1}{1}{v} - p)}\Expectedsub{\mathbf{x}_1^{(2)}}{ \prod_{v \in \alpha} (\sigmanumuv{2}{1}{v} - p)} \right)^{m}
\end{align*}
Now, we can take the expectation over $\phi^{(1)}, \phi^{(2)}$ and due to the independence of the entries in $\masknum{1}, \masknum{2}$, we can take the expectation over $\masknum{1}, \masknum{2}$ into the products above in order to obtain
\begin{align}\label{eq:signedweighthardness}
    &\Expectedsubb{\phi^{(1)}, \phi^{(2)} \sim \varrho \otimes \varrho}{ \Expectedsubb{\xi \sim \nu}{ \frac{\mathrm{d} \mu_{\phi^{(1)}}}{\mathrm{d} \nu}(\xi) \frac{\mathrm{d} \mu_{\phi^{(2)}}}{\mathrm{d} \nu}(\xi) }  } \nonumber \\ 
    &\hspace{2cm}= \Expectedsubb{\latentstuff}{ \left( 1 + \sum_{\emptyset \neq \alpha \subseteq R} \frac{1}{(p(1-p))^{|\alpha|}} \Expectedsub{\mathbf{x}_1^{(1)}}{ \prod_{v \in \alpha} \Expectedsubnop{\masknum{1}_{1,v}}{(\sigmanumuv{1}{1}{v} - p)} } \Expectedsub{\mathbf{x}_1^{(2)}}{ \prod_{v \in \alpha} \Expectedsubnop{ \masknum{2}_{1,v}}{(\sigmanumuv{2}{1}{v} - p)}  } \right)^{m} } \nonumber \\
    &\hspace{2cm}=\Expectedsubb{\latentstuff}{ \left( 1 + \sum_{\emptyset \neq \alpha \subseteq R} q^{2|\alpha| } \ \frac{\Expected{ \SWnum{K_{1,\alpha}}{1} } \Expected{ \SWnum{K_{1,\alpha}}{2} }}{(p(1-p))^{|\alpha|}}\right)^{m}  }
\end{align}
where for $k \in \{1,2\}$ 
$$
    \SWnum{K_{1,\alpha}}{k} \coloneqq \prod_{j \in \alpha} (\sigma(\langle \mathbf{x}_{j}^{(k)}, \mathbf{x} \rangle) - p), \text{ and } \sigma(\langle \mathbf{x}, \mathbf{y} \rangle) \coloneqq \mathds{1}(\langle \mathbf{x}, \mathbf{y}\rangle \le \tau)
$$ and where we used $$\Expectedsubnop{\masknum{1}_{1,v}}{(\sigmanumuv{1}{1}{v} - p)} = q (\sigma(\langle \mathbf{x}_1^{(1)}, \mathbf{x}_v^{(1)} \rangle) - p), \text{ and } \Expectedsubnop{\masknum{2}_{1,v}}{(\sigmanumuv{2}{1}{v} - p)} = q ( \sigma(\langle \mathbf{x}_1^{(2)}, \mathbf{x}_v^{(2)} \rangle) - p). $$ Note that in the last line of \eqref{eq:signedweighthardness}, the inner expectations are over a single $\mathbf{x} \sim \gauss{\mathbf{I}_d}$ and to be understood as conditional on $\xrnum{1}, \xrnum{2}$\footnote{we omitted an explicit conditioning here and in the following in order to simplify notation.}. Accordingly, $\Expected{ \SWnum{K_{1,\alpha}}{1} }$ is the expected signed weight of a star with set of leaves $\alpha$, conditional on the latent vectors of the leaves, taken over the randomness of the latent vector corresponding to the center. 

\subsection{The case of arbitrary $p$}
To show the bound for arbitrary $p$, from \Cref{eq:signedweighthardness} we get the following.
\begin{align}\label{eq:expp} 
    &\Expectedsubb{\phi^{(1)}, \phi^{(2)} \sim \varrho \otimes \varrho}{ \Expectedsubb{\xi \sim \nu}{ \frac{\mathrm{d} \mu_{\phi^{(1)}}}{\mathrm{d} \nu}(\xi) \frac{\mathrm{d} \mu_{\phi^{(2)}}}{\mathrm{d} \nu}(\xi) }  } = \Expectedsubb{\latentstuff}{ \left( 1 + \sum_{\emptyset \neq \alpha \subseteq R} q^{2|\alpha| } \frac{\Expected{ \SWnum{K_{1,\alpha}}{1} } \Expected{ \SWnum{K_{1,\alpha}}{2} }}{(p(1-p))^{|\alpha|}}\right)^{m} } \nonumber \\
    &\hspace{4cm}\le \Expectedsubb{\latentstuff}{ \exp\left( m \sum_{\emptyset \neq \alpha \subseteq R} q^{2|\alpha| } \frac{\Expected{ \SWnum{K_{1,\alpha}}{1} } \Expected{ \SWnum{K_{1,\alpha}}{2} }}{(p(1-p))^{|\alpha|}} \right) }
\end{align} where the expectations are taken over $\mathbf{x} \sim \mathcal{N}(0, \mathbf{I}_d)$.

\subsubsection{Identifying the good event $S$ and deriving an expression for $\Expected{ \SW{K_{1,\alpha}} }$}

To proceed, we rely the Fourier-theoretic expression of $\Expected{ \SWnum{K_{1,\alpha}}{1} }, \Expected{ \SWnum{K_{1,\alpha}}{2} }$ from \Cref{prop:fourierexpressionthing}.
To this end, recall the definition of $\Sevent$,
$$
     \Sevent \coloneqq \Seventdef, 
$$ where $\mathbf{I}$ is the $n \times n$ identity matrix. We proceed by setting $\rho \coloneqq C\sqrt{\log(n)}$ and conclude that for sufficiently large $C > 0$, the event $\Seventshort$ occurs with high probability.
\begin{observation}\label{lem:sislikely}
    For sufficiently large $C > 0$ and $\rho \coloneqq C\sqrt{\log(n)}$, we have $\Prnop{\Seventshort} = 1-o(1)$.
\end{observation}
\noindent Under occurrence of $\Sevent$, we derive recall the following bound on the conditional expectations $\Expected{ \SWnum{K_{1,\alpha}}{1}}$ and $\Expected{ \SWnum{K_{1,\alpha}}{2} }$, i.e., \Cref{prop:fourierexpressionthing}, which we formally prove in \Cref{sec:proofofprop}. 

\fourierstuff*

\newcommand{\term}{\Lambda_\alpha(\mathbf{X}_R^{(1)}, \mathbf{X}_R^{(2)})}
\newcommand{\defterm}{\leadingnum{1} \leadingnum{2} + \leadingnum{1}\restnum{2} + \leadingnum{2}\restnum{1}}
\noindent For our purposes, it is further convenient to define another quantity $\term$ such that for any $\alpha \subseteq R$, we have
\begin{align*}
    \Expected{\SWnum{K_{1,\alpha}}{1} } \Expected{\SWnum{K_{1,\alpha}}{2}} &= \term + \restnum{1} \restnum{2}, \\
    \text{that is } \term &\coloneqq \defterm.
\end{align*}

\subsubsection{Simplifying the sum over $\alpha$}

\noindent Using the bound from \Cref{prop:fourierexpressionthing} combined with the assumptions $d \gg \bound$ and $d \gg \boundwedges$, we can already simplify the sum in \eqref{eq:expp} 
To this end, we first note that we can omit terms corresponding to $|\alpha| \ge \log(m)$.
\begin{lemma}[Ignoring terms with large $|\alpha|$]
    If $d \gg \bound$, then 
    $$
        \sum_{\substack{\emptyset \neq \alpha \subseteq R \\ |\alpha| \ge \log(m)}}  q^{2|\alpha|}\frac{ \Expected{ \SWnum{K_{1,\alpha}}{1} } \Expected{ \SWnum{K_{1,\alpha}}{2} }}{(p(1-p))^{|\alpha|}} = o\left( \frac{1}{m} \right).
    $$
\end{lemma}
\begin{proof}
    Using the bound from \Cref{prop:fourierexpressionthing}, we get 
    \begin{align*}
        \sum_{\substack{\emptyset \neq \alpha \subseteq R \\ |\alpha| \ge \log(m)}}  q^{2|\alpha|}\frac{ \Expected{ \SWnum{K_{1,\alpha}}{1} } \Expected{ \SWnum{K_{1,\alpha}}{2} }}{(p(1-p))^{|\alpha|}} \le \sum_{k \ge \log(m) } \binom{n}{k} \ \frac{q^{2k}}{(p(1-p))^k}  \left( \frac{C k \rho }{\sqrt{d}} \right)^{k} \le \sum_{k \ge \log(m) } \left( \frac{2Ce \rho }{p(1-p)} \frac{n q^2}{\sqrt{d}} \right)^{k}.
    \end{align*} Due to the assumption $d \gg \bound$ and $m \ge n$, we get that the base of the exponential in the first sum is $o(1)$. Since the exponent $k$ is at least $\log(m)$, we get that the total expression is $o(1/m)$, as desired. 
\end{proof}
\noindent Moreover, we can ignore all the higher-order terms obtained after applying \Cref{prop:fourierexpressionthing}, which we do in the following. 
\begin{lemma}[Ignoring higher order terms] If $d \gg \bound$ and $d \gg \boundwedges$ then
    $$
    \sum_{ \substack{\alpha \subseteq R \\ |\alpha| \ge 2}} q^{2|\alpha|} \frac{\Expected{\SWnum{K_{1,\alpha}}{1} } \Expected{\SWnum{K_{1,\alpha}}{2} }}{(p(1-p))^{|\alpha|}} \le  \sum_{ \substack{\alpha \subseteq R \\ |\alpha| \ge 2}} q^{2|\alpha|}  \frac{\term}{(p(1-p))^{|\alpha|}} + o\left( \frac{1}{m} \right)
    $$
\end{lemma}
\begin{proof}
    Using \Cref{prop:fourierexpressionthing} and splitting the sum based on whether $|\alpha|$ is odd and even, we get 
    \begin{align*}
        &\sum_{ \substack{\alpha \subseteq R \\ |\alpha| \ge 2}} q^{2|\alpha|} \frac{\Expected{\SWnum{K_{1,\alpha}}{1} } \Expected{\SWnum{K_{1,\alpha}}{2} }}{(p(1-p))^{|\alpha|}} - \sum_{ \substack{\alpha \subseteq R \\ |\alpha| \ge 2}} q^{2|\alpha|} \frac{\term}{(p(1-p))^{|\alpha|}}  \\
        &\hspace{2cm}\le \sum_{k=1}^{\lfloor\frac{n-1}{2}\rfloor} \binom{n}{2k+1} q^{4k+2} \left( \frac{C (2k+1) \rho }{\sqrt{d}} \right)^{2k + 4} + \sum_{k=1}^{\lfloor\frac{n}{2}\rfloor} \binom{n}{2k} q^{4k} \left( \frac{2C k \rho }{\sqrt{d}} \right)^{2k + 2}\\ 
        &\hspace{2cm}\le \sum_{k=1}^{\lfloor\frac{n-1}{2}\rfloor} \left( \frac{2e n}{2k+1}\right)^{2k+1} q^{4k + 2}\left( \frac{C (2k + 1) \rho }{\sqrt{d}} \right)^{2k + 4} + \sum_{k=1}^{\lfloor\frac{n}{2}\rfloor} \left( \frac{2e n}{2k} \right)^{2k} q^{4k} \left( \frac{2 Ck \rho }{\sqrt{d}} \right)^{2k + 2}\\
        &\hspace{2cm}\le \sum_{k=1}^{\lfloor\frac{n-1}{2}\rfloor} \left( \frac{(2k+1)\rho}{\sqrt{d}}\right)^{3} \left( \frac{2Ce n q^2 \rho }{\sqrt{d}} \right)^{2k + 1} + \sum_{k=1}^{\lfloor\frac{n}{2}\rfloor} \left( \frac{2k \rho }{\sqrt{d}} \right)^{2} \left( \frac{2Ce n q^2 \rho }{\sqrt{d}} \right)^{2k}\\
        &\hspace{7cm} = O\left( \frac{n^3 q^6 \log(n)^{3} }{d^3} \right) + O\left( \frac{n^2 q^4 \log(n)^2 }{d^2} \right) = o\left(\frac{1}{m}\right),
    \end{align*}
    where the penultimate step is due to $m \ge n$ and thus, by assumption, $d \gg \bound \ge n^2 q^4 \log(m)$.
    Using further the assumption $d \gg \boundwedges \ge n \sqrt{m}q^2 \log(m)$ yields the $o\left(\frac{1}{m}\right)$ term.
\end{proof}

\subsubsection{Bounding the exponential via Gaussian Hypercontractivity}
Now, we are left with the following expression.
\begin{align*}
    &1 + \dchisq{\mathcal{L}(\mu \mid \Seventshort)}{\nu} \le \\ & \hspace{.7cm} (1 + o(1)) \ \underset{\latentstuff}{\mathbb{E}} \left[ \exp\left(  mq^2 \sum_{u\in R} \frac{ \Expected{ \SWnum{K_{1,\{u\}}}{1} } \Expected{ \SWnum{K_{1,\{u\}}}{2} }}{p(1-p)} \right.\right. + m \sum_{\substack{\emptyset \neq \alpha \subseteq R\\ 2\le |\alpha| \le \log(m)}} \left.\left. q^{2|\alpha|}\frac{ \term  }{(p(1-p))^{|\alpha|}}  \right) \right].
\end{align*}
To bound the above sum, we handle each term separately, which is justified after a repeated application of Cauchy-Schwarz in the following lemma.

\begin{lemma}\label{lem:thisisactuallywhatwewannabound}
    \begin{align*}
        &1  +\dchisq{\mathcal{L}(\mu \mid \Seventshort)}{\nu} \le (1 + o(1)) \Expectedsubb{ \latentstuff }{ \exp\left( 2m q^{2} \sum_{u\in R} \frac{ \Expected{ \SWnum{K_{1,\{u\}}}{1} } \Expected{ \SWnum{K_{1,\{u\}}}{2} }}{p(1-p)} \right) }^{1/2} \\
        &\hspace{6cm} \times \prod_{\substack{k = 2}}^{\log(m)} \Expectedsubb{ \latentstuff }{ \exp\left( m \left( \frac{ 2q^2 }{(p(1-p))}\right)^{k} \sum_{\alpha \subseteq R, |\alpha| = k}  \term  \right) }^{1/2^k}.
    \end{align*}
\end{lemma}
\begin{proof}
    The lemma follows after repeatedly applying the Cauchy-Schwarz inequality, i.e. 
    \begin{align*}
        \Expected{\exp(f g)} \le \Expected{\exp(2f)}^{1/2}\Expected{\exp(2g)}^{1/2}. 
    \end{align*}\qedhere
\end{proof}
\noindent We show that each term in the above product is $1 + o(1)$ using the following tail bound on polynomials with i.i.d. Gaussian inputs that follows from hypercontractivity.
\begin{proposition}[Concentration from Hypercontractivity]\label{prop:hypercontractivity}
    There are constants $C > 0, 1 > c > 0$ such that for every polynomial $f$ of degree $k \ge 1$ in i.i.d. standard Gaussian inputs $\mathbf{x}$ with $\Expectedsub{\mathbf{x}}{f(\mathbf{x})} = 0$, we have 
    $$
        \Pr{ |f(\mathbf{x})| \ge t \sigma(f)  } \le C\exp( - c t^{2/k})
    $$ where $\sigma = \sqrt{ \Var{f(\mathbf{x})}}$.
\end{proposition}
\begin{proof}
    For any $q \ge 2$, due to $\sigma(f) = \|f\|_2$, and due to Markov's inequality, we have that 
    \begin{align*}
        \Pr{ |f(\mathbf{x})| \ge t \sigma(f) } = \Pr{ |f(\mathbf{x})|^q \ge t^q \|f\|_2^q } \le \frac{\Expectedsub{\mathbf{x}}{|f(\mathbf{x})|^q}}{t^q \|f\|_2^q}.
    \end{align*} Applying hypercontractivity (\Cref{lem:hypercont}), we get 
    $$
    \Pr{ |f(\mathbf{x})| \ge t \sigma(f) } \le \frac{(q-1)^{qk/2}}{t^q} = \exp\left( \frac{k}{2}q\log(q-1) - q\log(t) \right).
    $$ Setting $q = \frac{1}{2}t^{2/k}$ yields 
    \begin{align*}
        \Pr{ |f(\mathbf{x})| \ge t \sigma(f) } &\le \exp\left( \frac{k}{4}t^{2/k}\log(t^{2/k} / 2 ) - \frac{1}{2} t^{2/k}\log(t) \right)
        \\ &= \exp\left( \frac{1}{2} t^{2/k} \left( \frac{k}{2}\log(t^{2/k}) + \frac{k}{2}\log\left( 1/2 \right) - \log(t) \right) \right) = \exp\left( \frac{k}{4} \log(1/2) t^{2/k} \right),
    \end{align*}
    which yields our lemma for $c = \frac{k}{4} \log(2) \ge 0.17$ for all $k \ge 1$.
\end{proof}
We mainly use the above to derive the following statement asserting that $\Expected{\exp(f)}$ is typically of order $1 + \sigma(f)$.
\begin{lemma}[Exponential integrability by hypercontractivity]\label{lem:integrationbytails}
    Let $f$ be a polynomial of degree $\ell \ge 2$ in $i.i.d.$ standard Gaussian inputs $\mathbf{x}$ with $\Expectedsub{\mathbf{x}}{f(\mathbf{x})} = 0$ and variance $\sigma \coloneqq \sqrt{ \Var{f(\mathbf{x})}}$. Let further $c$ be the constant in the exponent of the tail bound from \cref{prop:hypercontractivity} and assume that $|f(\mathbf{x})| \le U$ for all $\mathbf{x} \in S$ where $S$ is some measurable set. Assume $U\sigma^{\frac{2}{\ell-2}} \le \frac{c^2}{4}$ if $\ell \ge 3$, and that $\sigma \le \frac{c}{2}$ if $\ell = 2$. Note that in case $\ell = 2$, we allow $U = \infty$. Then, there is a constant $C > 0$ independent of $f$ such that \begin{align*}
        \Expectedsub{\mathbf{x}}{ \mathds{1}( \mathbf{x} \in S) \exp(f(\mathbf{x})) } \le 1 + (C \ell)^{\frac{\ell + 2}{2}} \sigma.
    \end{align*}
\end{lemma}
\begin{proof}
    Assume for now that $\ell \ge 3$. Then,
    \begin{align*}
        \Expected{\exp(f) \mathds{1}(\mathbf{x} \in S) } &\le \Expected{\exp(|f|) \mathds{1}(|f| \le U) } \le \int_{0}^{e^U} \Pr{\exp(|f|) > t} \d t.
    \end{align*}
    Integration by substitution then yields
    \[
        \int_{0}^{e^U} \Pr{\exp(|f|) > t} \d t 
        = \int_{-\infty}^{U} \Pr{|f| > s} \exp(s) \d s
        \le 1 + \int_{0}^{U} \Pr{|f| > s} \exp(s) \d s 
        \le 1 + C' \int_{0}^{U} \exp\left(- c \left( \frac{s}{\sigma} \right)^{2/\ell} + s \right) \d s.
    \]
    Due to the assumption $U\sigma^{\frac{2}{\ell-2}} \le \frac{c^2}{4}$, we get 
    $
        s - c \left( \frac{s}{\sigma} \right)^{2/\ell} \le -\frac{c}{2}\left( \frac{s}{\sigma} \right)^{2/\ell},
    $ so 
    $$
        \Expected{\exp(f) \mathds{1}(\mathbf{x} \in S) } \le 1 + C' \int_{0}^{U} \exp\left(- \frac{c}{2} \left( \frac{s}{\sigma} \right)^{2/\ell}  \right) \d s.
    $$
    Substituting $t = \frac{c}{2} \left( \frac{s}{\sigma} \right)^{2/\ell}$ yields 
     \begin{align*}
         \Expected{\exp(f) \mathds{1}(\mathbf{x} \in S) } &\le 1 + \frac{1}{2}C'\sigma \ell \left(\frac{2}{c}\right)^{\ell /2} \int_{0}^{\infty} t^{\frac{\ell - 2}{2}} e^{-t} \d s \le 1 + \frac{1}{2}C'\sigma \ell \left(\frac{2}{c}\right)^{\ell /2} \Gamma\left(\frac{\ell}{2}\right) \le 1 + \sigma (C\ell)^{\frac{\ell + 2}{2}},
    \end{align*}
    as desired. 
    In case $\ell = 2$, we get from analogous calculations that 
    \begin{align*}
        \Expected{\exp(f) \mathds{1}(\mathbf{x} \in S) } &\le  1 +C' \int_{0}^{\infty} \exp\left(- c \left( \frac{s}{\sigma} \right)^{2/\ell} + s \right) \d s \le  1 + C'\int_{0}^{\infty} \exp\left( -\frac{cs}{2\sigma} \right) \d s \le 1 + 2C\sigma,
    \end{align*}
    as desired.
\end{proof}

\newcommand{\sigsq}[2]{\sigma^2_{#1}(#2)}
\begin{lemma}
    For each $2 \le k \le \log(m)$, $$
        \Phi_k \coloneqq \Expectedsubb{ \latentstuff }{ \exp\left( m \left( \frac{ 2q^2 }{(p(1-p))}\right)^{k} \sum_{\alpha \subseteq R, |\alpha| = k} \term  \right) } = 1 + o(1). \qedhere
    $$
\end{lemma}
\begin{proof}
    Recall from the definition of the term $\term$ that $$
        \term = \defterm
    $$ where $\leadingnum{1}, \leadingnum{2}$ are polynomials of degree $\lceil k/2\rceil$ in $\xrnum{1}, \xrnum{2}$, respectively. Therefore, let us define 
    \begin{align*}
        f_{\alpha}(\xrnum{1}, \xrnum{2}) &\coloneqq m \left( \frac{ 2 q^{2} }{(p(1-p))}\right)^{k} \sum_{\alpha \subseteq R, |\alpha| = k} \leadingnum{1}\leadingnum{2}  \\
        g_{\alpha}(\xrnum{1}, \xrnum{2}) &\coloneqq m \left( \frac{ 2 q^2 }{(p(1-p))}\right)^{k} \sum_{\alpha \subseteq R, |\alpha| = k} \leadingnum{1}\restnum{2} \\
        h_{\alpha}(\xrnum{1}, \xrnum{2}) &\coloneqq m \left( \frac{ 2 q^2 }{(p(1-p))}\right)^{k} \sum_{\alpha \subseteq R, |\alpha| = k} \restnum{1}\leadingnum{2}.
    \end{align*}
    Note that the three quantities $f_{\alpha}(\xrnum{1}, \xrnum{2}), g_{\alpha}(\xrnum{1}, \xrnum{2}), h_{\alpha}(\xrnum{1}, \xrnum{2})$ each a polynomial of degree $2\ell$ where $\ell \coloneqq \lceil k/2 \rceil$.
    To bound $\Phi_k$, we once again apply Cauchy-Schwarz and arrive at 
    $$
        \Phi_k \le \Expectedsubb{ \latentstuff }{ \exp\left( 2 f_{\alpha}(\xrnum{1}, \xrnum{2}) \right) }^{1/2} \Expectedsubb{ \latentstuff }{ \exp\left( 4 g_{\alpha}(\xrnum{1}, \xrnum{2}) \right) }^{1/4} \Expectedsubb{ \latentstuff }{ \exp\left( 4 h_{\alpha}(\xrnum{1},  \xrnum{2}) \right) }^{1/4}.
    $$ 
    We bound each of the three terms by removing the conditioning on one of $\xrnum{1}, \xrnum{2}$ and then applying Gaussian hypercontractivity. Specifically, for a fixed $\xrnum{2} \in S$, we use the non-negativity of the integrand to bound
    \begin{align*}
        \Expectedsubb{ \xrnum{1} \sim \Seventshort }{ \exp\left( 2 f_{\alpha}(\xrnum{1}, \xrnum{2} ) \right) } = \Pr{\Seventshort}^{-1} \Expectedsubb{ \xrnum{1} \sim \mathcal{N}(0, \mathbf{I}_d)^{\otimes n} }{ \mathds{1}(\xrnum{1} \in \Seventshort) \exp\left( 2 f_{\alpha}(\xrnum{1}, \xrnum{2}) \right) }.
    \end{align*}
    Now, we wish to apply \Cref{lem:integrationbytails}. To this end, we use the following estimates about the magnitude and variance of $f, g, h$. To this end, let us specify what exactly we mean by variance. We define 
    \begin{align*}
        \sigsq{f_\alpha}{\xrnum{2}} &\coloneqq \Expectedsubb{\xrnum{1} \sim \mathcal{N}(0, \mathbf{I}_d)^{\otimes n}}{ \left. f_\alpha(\xrnum{1}, \xrnum{2})^2 \ \right| \ \xrnum{2} } \\
        \sigsq{g_\alpha}{\xrnum{2}} &\coloneqq \Expectedsubb{ \xrnum{1} \sim \mathcal{N}(0, \mathbf{I}_d)^{\otimes n}}{ \left. g_\alpha(\xrnum{1}, \xrnum{2})^2  \ \right| \ \xrnum{2} } \\
        \sigsq{h_\alpha}{\xrnum{1}}
        &\coloneqq \Expectedsubb{ \xrnum{2} \sim \mathcal{N}(0, \mathbf{I}_d)^{\otimes n}}{ \left. h_\alpha(\xrnum{1}, \xrnum{2})^2  \ \right| \ \xrnum{1} }.
    \end{align*}
    \newcommand{\sigtil}{\widehat{\sigma}}
    \begin{claim}\label{clm:variance}
        There is a constant $C > 0$ independent of $k, \alpha$, such that for $\xrnum{1}, \xrnum{2} \in S$ and $\ell \coloneq \lceil k/2 \rceil$, we have 
        \begin{align*}
            |f_\alpha(\xrnum{1}, \xrnum{2})|, |g_\alpha(\xrnum{1}, \xrnum{2})|, |h_\alpha(\xrnum{1}, \xrnum{2})| &\le m (n q^{2})^{ k } \left( \frac{C k \log(n)^{\frac{1}{2}} }{\sqrt{d}} \right)^{2\ell} \eqqcolon U \\
            \text{and } \sigsq{f_\alpha}{\xrnum{2}}, \sigsq{g_\alpha}{\xrnum{2}}, \sigsq{h_\alpha}{\xrnum{1}} &\le m^2 n^{2\ell} q^{4k} k^3 \left( \frac{C k \log(n)^{\frac{1}{2}} }{\sqrt{d}} \right)^{4\ell} \eqqcolon \sigtil^2
        \end{align*}
    \end{claim}
    \noindent The proof is deferred to \Cref{sec:deferredproofs}. With the above bounds, note that for even $k \ge 4$, we have 
    \begin{align*}
        U\sigtil^{\frac{2}{2\ell-2}} \le \left( \frac{C k \log(n)^{\frac{1}{2}} nq^2 }{\sqrt{d}} \right)^{k} \left( \frac{C^2 k^2 \log(n) n q^4 }{d} \right)^{\frac{k}{k - 2}} m^{\frac{2}{k-2}} k^{\frac{3}{k-2}} = o(1)
    \end{align*} 
    since $d \gg \bound$. 
    For odd $k \ge 3$, we get $2\ell = k + 1$ and therefore 
    \begin{align*}
        \sigtil^{\frac{2}{2\ell-2}} = \sigtil^{\frac{2}{k-1}} &\le \left( m^2 n^{k+1}q^{4k}k^3 \left( \frac{C^2 k^2 \log(n) }{d} \right)^{k+1}  \right)^{\frac{1}{k-1}}  \le 
        \left( \frac{C^{4} \log(n)^{2} k^7 m^2 n^2 q^4}{d^2} \right)^{\frac{1}{k-1}} \frac{C^2 \log(n) k^2 nq^4 }{d}\\
        \text{and thus } U\sigtil^{\frac{2}{2\ell-2}} &\le \left( \frac{C^{4} \log(n)^{2} k^7 m^2 n q^4}{d^2} \right)^{\frac{1}{k-1}} \frac{C^2 \log(n) k^2 m n^{\frac{1}{k-1}} q^2 }{d} \left( \frac{C \log(n)^{\frac{1}{2}} k nq^2 }{\sqrt{d}} \right)^{k+1} = o(1)
    \end{align*} since $d \gg \boundwedges$ and $d \gg \bound$. This makes \Cref{lem:integrationbytails} applicable and yields 
    \begin{align*}
        \Expectedsubb{ \xrnum{1} \sim \Seventshort }{ \exp\left( 2 f_{\alpha}(\xrnum{1}, \xrnum{2}) \right) } \le \Pr{\Seventshort}^{-1} \left(1 +  (C \ell)^{\frac{2\ell+2}{2}} \sigtil \right) \le
            \Pr{\Seventshort}^{-1} \left( 1 + Ck^4 \left(\frac{k}{m} \right)^{\frac{k-2}{2}} \left( \frac{C^2 \log(n) k m n q^4 }{d} \right)^{\frac{k}{2}}\right)
    \end{align*} if $k$ is even (in this case $2\ell = k$), and 
    \begin{align*}
        \Expectedsubb{ \xrnum{1} \sim \Seventshort }{ \exp\left( 2 f_{\alpha}(\xrnum{1}, \xrnum{2}) \right) } &\le \Pr{\Seventshort}^{-1} \left( 1 + (C \ell)^{\frac{2\ell+2}{2}} \sigtil \right) \\
        &\le 
            \Pr{\Seventshort}^{-1} \left( 1 + C^2k^6 \left(\frac{C^4\log(n)^{2} m^2 n q^4 }{d^2}\right)^{\frac{1}{2}} \left( \frac{C^3 \log(n) k^3 n^{1 + \frac{1}{k-1}} q^4 }{d} \right)^{\frac{k-1}{2}}\right).
    \end{align*} if $k$ is odd. In both cases, this evaluates to $1 + o(1)$ since $d \gg \bound$ and $d \gg \boundwedges$. The same reasoning applies to $g_{\alpha}$ and $h_{\alpha}$.
\end{proof}

\subsubsection{Bounding the exponential sum over all edges}

The only thing that is left to be bounded is the first term in \Cref{lem:thisisactuallywhatwewannabound} which is an exponential sum over expected signed weights of single edges. The idea here is that we can again use non-negativity of the integrand to remove the conditioning and then expand the exponential by a Taylor series. The result then factorizes as a sum over the (unconditional) expected signed weight of stars, which can be bounded explicitly. We capture all this in the following lemma.
\begin{lemma}
$$
    \Expectedsubb{ \latentstuff }{ \exp\left( 2mq^2 \sum_{u\in R} \frac{ \Expected{ \SWnum{K_{1,\{u\}}}{1} } \Expected{ \SWnum{K_{1,\{u\}}}{2} }}{p(1-p)} \right) }^{1/2} = 1 + o(1).
$$
\end{lemma}
\begin{proof}
    We start by using non-negativity of the integrand to remove the conditioning on $S$ at the cost of a small factor.
    \begin{align*}
        &\Expectedsubb{ \latentstuff }{ \exp\left( 2mq^2 \sum_{u\in R} \frac{ \Expected{ \SWnum{K_{1,\{u\}}}{1} } \Expected{ \SWnum{K_{1,\{u\}}}{2} }}{p(1-p)} \right) } \\ 
        &\hspace{4cm} \le \Pr{S}^{-2} \Expectedsubb{ \latentstuffind }{ \exp\left( 2mq^2 \sum_{u\in R} \frac{ \Expected{ \SWnum{K_{1,\{u\}}}{1} } \Expected{ \SWnum{K_{1,\{u\}}}{2} }}{p(1-p)} \right) }.
    \end{align*} Now, by a Taylor expansion,
    \begin{align*}
        T &\coloneqq \Expectedsubb{ \latentstuffind }{ \exp\left( 2mq^2 \sum_{u\in R} \frac{ \Expected{ \SWnum{K_{1,\{u\}}}{1} } \Expected{ \SWnum{K_{1,\{u\}}}{2} }}{p(1-p)} \right) }\\
        &\hspace{4cm}\le 1 + \sum_{k = 1}^{\infty} \frac{1}{k!} \left( \frac{2mq^2}{p(1-p)} \right)^k \sum_{ \substack{ j_1, \ldots, j_k \in \mathbb{N}_{\ge 0} \\ \sum_\ell j_\ell = k }} n^{\supp(j_1, \ldots, j_k)} \prod_{\ell = 1}^{k} \binom{k}{j_\ell} \ \Expected{ \SW{K_{1, j_\ell}} }^2 \\
        &\hspace{4cm}\le 1 + \sum_{k = 1}^{\infty} \sum_{ \substack{ j_1, \ldots, j_k \in \mathbb{N}_{\ge 0} \\ \sum_\ell j_\ell = k }} n^{k} \frac{k^k}{k!} \prod_{\ell = 1}^{k} \frac{1}{j_\ell!} \left( \frac{2mq^2}{p(1-p)} \right)^{j_\ell} \Expected{ \SW{K_{1, j_\ell}} }^2 \\ &\hspace{4cm}\le  1 + \sum_{k=1}^{\infty} \left( n \sum_{ \substack{ \ell = 1}}^\infty \frac{1}{\ell!} \left(\frac{Cmq^2}{p(1-p)}\right)^\ell  \Expected{ \SW{K_{1, \ell}} }^2 \right)^k,
    \end{align*} where $\supp(j_1, \ldots, j_k)$ denotes the number of non-zero $j_\ell$, where we used $ \frac{k^k}{k!} \le (2e )^k$ and where $C > 0$ is a constant. Note that the terms $\Expected{ \SWnum{K_{1,\{u\}}}{1} }$  $\Expected{ \SWnum{K_{1,\{u\}}}{2} }$ in the previous expression are expectations \emph{conditional} on $\xrnum{1}, \xrnum{2}$, respectively, while $\Expected{ \SW{K_{1, j_k}}}$ above is an unconditional expectation (we omitted an explicit conditioning here for the sake of notational simplicity). Now, it suffices to show that the inner sum above is $o(1)$. To this end, we need the following bound on $\Expected{ \SW{K_{1, j_\ell}}}$.
    \begin{claim}\label{clm:stars}
        There is a constant $C> 0$ such that for any $\ell \ge 2$, 
        $
            \Expected{ \SW{K_{1, \ell}}} \le \left(\frac{C\ell}{d}\right)^{\ell/2 }.
        $
    \end{claim}
    \begin{proof}
        The proof is deferred to \Cref{sec:deferredproofs}.
    \end{proof}
    \noindent With this, we can bound
    \begin{align*}
        n \sum_{ \substack{ \ell = 1}}^\infty \frac{1}{\ell!} \left(\frac{Cmq^2}{p(1-p)}\right)^\ell  \Expected{ \SW{K_{1, \ell}} }^2 &\le n \sum_{ \substack{ \ell = 1}}^\infty \frac{1}{\ell!} \left(\frac{Cmq^2}{p(1-p)}\right)^\ell  \left(\frac{C\ell}{d}\right)^{\ell} \le n \sum_{ \substack{ \ell = 1}}^\infty \left(\frac{2C^2e}{p(1-p)} \frac{mq^2}{d}\right)^\ell = o(1),
    \end{align*} since $\ell! \ge \left( \frac{\ell}{2e} \right)^\ell$ and $\d \gg \boundwedges$. This implies 
    $
    T \le 1 + o(1),
    $ as desired.
\end{proof}

\subsection{The case of $p= 1/2$}\label{sec:ponehalfcase}


To prove the stronger bounds for $p = 1/2$, the methods presented so far would be sufficient, however, they would need some strengthening that would lead to quite some increase in technical complexity. However, it turns out that there is a much simpler and more direct proof, which crucially exploits symmetries arising only in the $p= 1/2$ case.

Specifically, going back to \Cref{eq:signedweighthardness}, instead of switching to an exponential over the sum of conditional signed weights of stars, we expand the expectation into a single sum over squared, conditional signed weights over all subgraphs of $K_{n,m}$.
\newcommand{\Expectedbig}[1]{\mathbb{E}\big[ #1 \big]}
\begin{align*}
    1 + \dchisq{\mathcal{L}(\mu \mid \Seventshort)}{\nu} &= \Expectedsubb{\latentstuff}{ \left( 1 + \sum_{\emptyset \neq \alpha \subseteq R} q^{2|\alpha|} \frac{\Expected{ \SWnum{K_{1,\alpha}}{1} } \Expected{ \SWnum{K_{1,\alpha}}{2} }}{(p(1-p))^{|\alpha|}}\right)^{m}  } \\
    &\hspace{0cm}= 1 + \sum_{\emptyset \neq \alpha \subseteq K_{n, m} } \left( \frac{q^2}{p(1-p)} \right)^{|\alpha|} \Big( \mathbb{E}_{\xl}\Expectedsub{\xr \sim S}{\SW{\alpha}}  \Big)^2.
\end{align*}
Above, the expression $\mathbb{E}_{\xl}\Expectedsub{\xr \sim S}{\SW{\alpha}}$ refers to the expected signed weight of $\alpha$ such that the left sided latent vectors are drawn unconditionally, while the right sided ones are drawn conditional on $S$. What is very important to take into account when summing over all $\alpha$ is that we need to argue that $\mathbb{E}_{\xl}\Expectedsub{\xr \sim S}{\SW{\alpha}} = 0$ if $\alpha$ has no leaves (i.e. vertices of degree one). This is not the case if $p \neq 1/2$, but due to the fact that $\tau = 0$ if $p = 1/2$ and the spherical symmetries of the Gaussian distribution, it is not hard to see that it is the case here. This property even holds when conditioning on $S$ even though this breaks the independence of the latent vectors in $R$. The reason is that $S$ is invariant under re-randomizing the latent vectors in a symmetry-preserving way, which enables us to define a simple noise operator that ``zeroes-out'' expected signed weights for graphs with leaves, even conditional on $S$. 
\begin{observation}[Noise operator in the conditional probability space]
    Define the noise operator $\mathbf{T}$ that acts on functions $f : \mathbb{R}^{n\times d} \rightarrow \mathbb{R}$ as $\mathbf{T}f(\mathbf{x}_1, \ldots, \mathbf{x}_n) = \Expectedsubb{\xi_1, \ldots, \xi_n \sim \{\pm 1\}}{ f(\xi_1\mathbf{x}_1, \ldots, \xi_n\mathbf{x}_n) }$ where the $\xi_i$ are i.i.d. Rademacher. Then, for any $f$, 
    $$
        \Expectedsubb{\xr \sim S}{ \mathbf{T} f(\xr) } = \Expectedsubb{\xr \sim S}{ f(\xr) },
    $$
    i.e. the law of $\ \xr$ conditional on $S$ is invariant under applying the noise operator $\mathbf{T}$. In particular, this implies that for $p = 1/2$ and any $\alpha \subseteq K_{n, m}$ such that $\alpha$ contains a vertex of degree one, we have $\Expectedsub{\xr \sim S}{\SW{\alpha}} = 0$.
\end{observation}

\noindent With the above and the bound from \Cref{prop:fourierexpressionthing}, we can now sum over all graphs with $a$ vertices on the left and $b$ vertices on the right that contain no leaves. To this end, we define $s = \max\{a,b\}$ and $r = \min\{a, b\}$, and we sum over all possible degrees $\delta_1, \ldots, \delta_r$ for the vertices on the larger side to obtain 
\begin{align*}
    &1 + \sum_{\emptyset \neq \alpha \subseteq K_{n, m} } \left( \frac{q^2}{p(1-p)} \right)^{|\alpha|} \Big( \mathbb{E}_{\xl}\Expectedsub{\xr \sim S}{\SW{\alpha}}  \Big)^2 \le 1 + \sum_{a=2}^n \sum_{b=2}^m \binom{n}{a} \binom{m}{b} \sum_{\substack{\delta_1, \ldots, \delta_{s}\\2 \le \delta_j \le r }} \prod_{j = 1}^s \binom{r}{\delta_j} \left( \frac{C \log(n)^{\frac{1}{2}} q^2 \delta_j }{\sqrt{d}} \right)^{ \delta_j}\\
    &\hspace{2cm}\le 1 + \sum_{a=2}^n \sum_{b=2}^m \binom{n}{a} \binom{m}{b} \sum_{\substack{\delta_1, \ldots, \delta_{s}\\2 \le \delta_j \le r }} \left( \frac{2Ce \log(n)^{\frac{1}{2}} q^2 r }{\sqrt{d}} \right)^{ \sum_j \delta_j} \le 1 + \sum_{a=2}^n \sum_{b=2}^m \binom{n}{a} \binom{m}{b} \left( \sum_{\substack{2 \le \delta \le r }} \left( \frac{2Ce \log(n)^{\frac{1}{2}} q^2 r }{\sqrt{d}} \right)^{ \delta} \right)^{s}.
\end{align*} 
Since $d \gg \bound$, the inner most sum is geometric and we obtain that there is a constant $C' \ge 0$ such that 
\begin{align*}
    &1 + \sum_{\emptyset \neq \alpha \subseteq K_{n, m} }  \left( \frac{q^2}{p(1-p)} \right)^{|\alpha|} \Big( \mathbb{E}_{\xl}\Expectedsub{\xr \sim S}{\SW{\alpha}}  \Big)^2 \le 1 + \sum_{a=2}^n \sum_{b=2}^m \left( \frac{2en}{a} \right)^a \left( \frac{2em}{b} \right)^b \left( \frac{C' \log(n)^{1/2} q^2 r }{\sqrt{d}} \right)^{ 2s} \\
    &\hspace{3cm}\le 1 + \sum_{a=2}^n \left( \frac{2en}{a} \right)^a \sum_{b=a}^m  \left( \frac{2C'^2e \log(n) q^4 m a }{d} \right)^{ b} + \sum_{b=2}^m \left( \frac{2em}{b} \right)^b \sum_{a=b}^n  \left( \frac{2C'^2e \log(n) q^4 n b }{d} \right)^{ a},
\end{align*} where the first sum covers the case of $b  \ge a$ and the second sum covers the case $a \ge b$ (note further how we bounded $r^2/b \le a$ and $r^2/a \le b$ ). Now, again because $d \gg \bound$, the inner sums are geometric, so there is a constant $C'' > 0$ such that 
\begin{align*}
    &1 + \sum_{\emptyset \neq \alpha \subseteq K_{n, m} }  \left( \frac{q^2}{p(1-p)} \right)^{|\alpha|} \Big( \mathbb{E}_{\xl}\Expectedsub{\xr \sim S}{\SW{\alpha}}  \Big)^2 \le  1 + \sum_{a=2}^n \left( \frac{C'' \log(n) q^4 m n }{d} \right)^{ a} + \sum_{b=2}^m  \left( \frac{C'' \log(n) q^4 m n }{d} \right)^{ b} = 1 + o(1),
\end{align*} as desired.

\subsection{Known vs. Unknown Masks: Proof of \Cref{thm:hardnessunknown}} 
At this point it becomes important to further specify the latent randomness $\phi$ appearing in the previous equation.
Our proof has the convenient property that it can handle both the case of a known mask $\mathbf{M}$, and the case where $\mathbf{M}$ is hidden, and only the matrix $W \sim \matnoiseedge$ is part of the input. In the former case, in order to show information-theoretic hardness, it suffices to show that $$
    \Expectedsub{\mathbf{M}}{ \dtv{\mathcal{L}(\mu \mid \mathbf{M}, \Seventshort)}{\nu} } = o(1),
$$ where $\mathcal{L}(\mu \mid \mathbf{M}, S)$ is the distribution of $\matnoiseedge$ conditional on a concrete mask $\mathbf{M}$ and the good event $S$. If the mask is unknown, we instead wish to show that
$$
    \dtv{\mu } {\nu} =  \dtv{\Expectedsub{\mathbf{M}}{\mathcal{L}(\mu \mid \mathbf{M}, \Seventshort)}}{\nu}  = o(1),
$$ i.e. that the overall distribution over matrices, averaged over all masks, converges to $\nu$ in total variation. In terms of applying the second moment method, this difference amounts to whether or not the random mask $M$ is part of the replicated latent information $\phi^{(1)}, \phi^{(2)}$ or not. This determines whether or not we average twice over our mask or only once and therefore essentially determines whether we multiply each term by $q$ or by $q^2$. Specifically, going back to \Cref{eq:signedweighthardness}, if the mask is known, we apply the expectation over $\mathbf{M}$ only once and ultimately obtain

\begin{align*}
    1 + \Expectedsub{\mathbf{M}}{ \dtv{\mathcal{L}(\mu \mid \mathbf{M}, S)}{\nu} } &\le \Expectedsubb{\phi^{(1)}, \phi^{(2)} \sim \varrho \otimes \varrho}{ \Expectedsubb{\xi \sim \nu}{ \frac{\mathrm{d} \mu_{\phi^{(1)}}}{\mathrm{d} \nu}(\xi) \frac{\mathrm{d} \mu_{\phi^{(2)}}}{\mathrm{d} \nu}(\xi) }  } \nonumber \\ 
    &= \Expectedsubb{\latentstuff}{ \left( 1 + \sum_{\emptyset \neq \alpha \subseteq R} \frac{1}{(p(1-p))^{|\alpha|}} \Expectedsub{\mathbf{x}_1^{(1)}, \mathbf{x}_1^{(2)}}{ \prod_{v \in \alpha} \Expectedsubnop{\mathbf{M}_{1,v}}{(\sigmanumuv{1}{1}{v} - p) (\sigmanumuv{2}{1}{v} - p)}  } \right)^{m} } \nonumber \\
    &=\Expectedsubb{\latentstuff}{ \left( 1 + \sum_{\emptyset \neq \alpha \subseteq R} q^{|\alpha| } \ \frac{\Expected{ \SWnum{K_{1,\alpha}}{1} } \Expected{ \SWnum{K_{1,\alpha}}{2} }}{(p(1-p))^{|\alpha|}}\right)^{m}  },
\end{align*} since 
$$
    \Expectedsubnop{\mathbf{M}_{1,v}}{(\sigmanumuv{1}{1}{v} - p) (\sigmanumuv{2}{1}{v} - p)} = q (\sigma(\langle \mathbf{x}_1^{(1)}, \mathbf{x}_v^{(1)} \rangle) - p) (\sigma(\langle \mathbf{x}_1^{(2)}, \mathbf{x}_v^{(2)} \rangle) - p).
$$
Given this, the entire rest of the proof proceeds analogously.

\section{Proof of \Cref{prop:fourierexpressionthing}}\label{sec:proofofprop}

\newcommand{\dassump}{\sqrt{d} \gg |\alpha| \log(n)}
\renewcommand{\dassump}{|\alpha|^2 \rho^2 \le \varepsilon d}
To derive \Cref{prop:fourierexpressionthing}, let us fix a set $\alpha \subseteq R$ which, for simplicity, we assume to consist of the natural numbers from $1$ to $\ell$, i.e., $\alpha = [\ell]$. To integrate $\SW{K_{1,\alpha}}$, we consider a random vector $\mathbf{z}_\beta \in \mathbb{R}^{|\alpha|}$ for every $\beta \subseteq \alpha$ with
$$
    \mathbf{z}_\beta \sim \mathcal{N}(0, \boldsymbol{\Sigma}_\beta) \text{ with }
    (\boldsymbol{\Sigma}_\beta)_{j_1, j_2} = \begin{cases}
        \frac{1}{d}\langle \mathbf{x}_{j_1}, \mathbf{x}_{j_2}\rangle & \text{if } j_1, j_2 \in \beta \\
        \sighat^2 & \text{if } j_1 = j_2 \text{ and } j_1, j_2 \notin \beta \\
        0 & \text{otherwise},
    \end{cases}
$$
Here, the quantity $\sighat$ is defined like in \Cref{lem:sigbound}, i.e., such that 
$$
    p = \Prnop{ \tfrac{1}{\sqrt{d}}\langle \mathbf{x}, \mathbf{y}\rangle \le \tau } = \Pr{Z \le \tau}
$$ for $\mathbf{x}, \mathbf{y}$ being i.i.d. random vectors sampled from $\mathcal{N}(0, \mathbf{I}_d)$ and $Z \sim \mathcal{N}(0, \sighat^2)$. 
Recall the definition of our good event $\Sevent$ given by 
$$
\Sevent \coloneqq \Seventdef.
$$

Throughout this section, we assume that $\dassump$ for some sufficiently small constant $\varepsilon > 0$. The reason we can do this is that for $\dassump$, the bound we wish to prove becomes larger than 1 (at least when setting the constant $C$ large enough), while expected signed weights can trivially be bounded by 1 in absolute value. Our assumption on $d$ together with the above definition of $S$ yield some desirable spectral properties of $\boldsymbol{\Sigma}$, that will become important later. We capture them in the following lemma.
\newcommand{\opnorm}[1]{\left\| #1 \right\|_{\mathrm{op}}}
\begin{lemma}\label{lem:spectralbound}
    There is a constant $C>0$ such that for $\dassump$ and latent vectors in $\Sevent$ it holds that for all $\beta \subseteq \alpha$ and sufficiently large $n$, 
    $$
        \opnorm{\boldsymbol{\Sigma}_\beta - \mathbf{I} } \le \delta(\varepsilon), \opnorm{\boldsymbol{\Sigma}_\beta - \sighat^2\mathbf{I} } \le \frac{C|\alpha|\rho}{\sqrt{d}}.
    $$ where $\delta(\varepsilon)$ is a function that tends to zero as $\varepsilon \rightarrow 0$. 
\end{lemma}
\begin{proof}
    Note that we can decompose $\boldsymbol{\Sigma}_\beta - \mathbf{I} = \widehat{\boldsymbol{\Sigma}_\beta} + (\sighat^2-1) \mathbf{I}$ with 
    $$
        (\widehat{\boldsymbol{\Sigma}_\beta})_{i,j} = \begin{cases}
        \frac{1}{d}\langle \mathbf{x}_{i}, \mathbf{x}_{j}\rangle - \sighat^2 \delta_{ij} & \text{if } i, j \in \beta \\
        0 & \text{otherwise},
    \end{cases}
    $$ where $\delta_{ij}$ is the Kronecker delta. Hence $ \opnorm{\boldsymbol{\Sigma}_\beta - \mathbf{I} } \le \opnorm{\widehat{\boldsymbol{\Sigma}_\beta}} + (\sighat^2-1)$. By \Cref{lem:sigbound}, we get that $\sighat^2-1 = o(1)$, so it only remains to bound $\opnorm{\widehat{\boldsymbol{\Sigma}_\beta}}$. To this end, we use the occurrence of $S$ to get 
    $$
        \opnorm{\widehat{\boldsymbol{\Sigma}_\beta}} \le |\beta| \|\widehat{\boldsymbol{\Sigma}_\beta}\|_\infty \le |\beta| \left( \frac{\rho}{\sqrt{d}}  + |\sighat^2 - 1| \right),
    $$ which tends to zero as $\varepsilon \rightarrow 0$ due to $|\beta| \le |\alpha|$ and our assumption $\dassump$. For the matrix $\boldsymbol{\Sigma}_\beta - \sighat^2\mathbf{I}$, we can directly use the definition of $\Seventshort$ and \Cref{lem:sigbound} to get that that there is a constant $C > 0$ such that $\|\boldsymbol{\Sigma}_\beta - \sighat^2\mathbf{I}\|_\infty \le C\rho/\sqrt{d}$, and the second part of the statement follows as well.
\end{proof}

\subsection{A Fourier-theoretic expression for expected signed weights}
Using this definition, we get the expression 
\begin{align*}
    &\mathbb{E}_{\mathbf{x}} \big[ \SW{K_{1,\alpha}} \mid (\mathbf{x}_u)_{u \in \alpha} \big]= \mathbb{E}_{\mathbf{x}} \Bigg[ {\prod_{\ell \in \alpha} ( \sigma(\langle \mathbf{x}_\ell, \mathbf{x} \rangle) - p ) } \Bigg]  = \sum_{\beta \subseteq \alpha} (-1)^{|\alpha \setminus \beta|} p^{|\alpha \setminus \beta|} \mathbb{E}_{\mathbf{x}} \Bigg[  {\prod_{\ell \in \beta} \sigma(\langle \mathbf{x}_\ell, \mathbf{x} \rangle)  } \Bigg].
\end{align*}
Now by definition of $\mathbf{z}_\beta$, we have $p^{|\alpha \setminus \beta|} \ \Expectedsub{\mathbf{x}}{{\prod}_{\ell \in \beta} \ \sigma(\langle \mathbf{x}_\ell, \mathbf{x} \rangle) } = \Pr{ \mathlarger\cap_{e \in \alpha} \{ \mathbf{z}_\beta(e)  \le \tau \} }$, so we get 
\begin{align*}
    \mathbb{E}_{\mathbf{x}} \big[ \SW{K_{1,\alpha}} \mid (\mathbf{x}_u)_{u \in \alpha} \big]= \sum_{\beta \subseteq \alpha} (-1)^{|\alpha \setminus \beta|} \ \Pr{ \mathlarger\cap_{e \in \alpha} \{ \mathbf{z}_\beta(e)  \le \tau \} }.
\end{align*}
To find a more convenient expression for the last term, we use Fourier inversion to get the following.
\begin{lemma}\label{lem:swfourier2}
    Whenever $\boldsymbol{\Sigma}_\beta$ is positive definite for all $\beta \subseteq \alpha$, we have
    \begin{align*}
        \Expectedsub{\mathbf{x}}{ \SW{K_{1,\alpha}} \mid (\mathbf{x}_u)_{u \in \alpha} } &= \lim_{h \rightarrow \infty} \frac{-1}{(2\pi)^{|\alpha|}} \int_{[-\infty, \infty]^{|\alpha|}} e^{-\trp{i\mathbf{t}}{\mathbf{x}}_{\tau}} \left( \sum_{\beta \subseteq \alpha} (-1)^{|\alpha \setminus \beta|}  \hspace{.1cm}\CF{\beta} \right) \prod_{e \in \mathbf{t}} \frac{1 - e^{ih\mathbf{t}(e)}}{i\mathbf{t}(e)} \d \mathbf{t},
    \end{align*}
    where $\CF{\beta} = \Expected{\exp(i\trp{\mathbf{t}}{\mathbf{z}_\beta})}$ is the characteristic function of $\mathbf{z}_\beta$, and $\mathbf{x}_{\tau}$ is the vector that contains $\tau $ everywhere. 
\end{lemma}

\begin{remark}
    Instead of writing the limit over $h$ explicitly, we will mostly write 
    \begin{align*}
        \Expectedsub{\mathbf{x}}{ \SW{K_{1,\alpha}} \mid (\mathbf{x}_u)_{u \in \alpha} } &= \frac{-1}{(2\pi)^{|\alpha|}} \int_{[-\infty, \infty]^{|\alpha|}} e^{-\trp{i\mathbf{t}}{\mathbf{x}}_{\tau}} \left( \sum_{\beta \subseteq \alpha} (-1)^{|\alpha \setminus \beta|}  \hspace{.1cm}\CF{\beta} \right) \prod_{e \in \mathbf{t}} \frac{\onebutnotactuallyone}{i\mathbf{t}(e)} \d \mathbf{t}
    \end{align*}
    where we think of $\onebutnotactuallyone$ as the limit of $1 - e^{ih\mathbf{t}(e)}$ as $h \rightarrow \infty$. The quantity $\onebutnotactuallyone$ usually behaves like one in most of our proofs, but it is important to keep in mind that it actually refers to a limiting object which in particular has the property that $| \prod_{e \in \mathbf{t}} \frac{1 - e^{ih\mathbf{t}(e)}}{i\mathbf{t}(e)} | \le 1$ for all $\mathbf{t}$ and all $h$, which ensures well-behaved integrals. 
\end{remark}

\begin{proof}[Proof of \Cref{lem:swfourier2}]
    We define  $\T \coloneqq [-\infty, \infty]^{|\alpha|}$ and 
    convolve the distribution of $\mathbf{z}_{\beta}$ with the uniform distribution $\mathsf{Unif}([0, h])^{\otimes |\alpha|}$ which has the CF
    \begin{align*}
        \varphi(\mathbf{t}) = h^{-|\alpha|} \prod_{e \in \alpha} \frac{e^{i\mathbf{t}(e)h} - 1}{i \mathbf{t}(e)}.
    \end{align*} Then, applying the inversion theorem \Cref{thm:inversion} yields that for any $\mathbf{x}$,
    \begin{align*}
        \frac{1}{h^{|\alpha|}} \Pr{ \mathlarger{\cap}_{e \in \alpha} \{ \mathbf{x}(e) - h \le \mathbf{z}_\beta(e) \le \mathbf{x}(e) \} } &= \frac{1}{h^{|\alpha|}} \frac{1}{(2\pi)^{|\alpha|}}  \int_{\mathcal{T}}  \CF{\beta} \ e^{-\trp{i\mathbf{t}}{\mathbf{x}}} \prod_{e \in \alpha} \frac{e^{i\mathbf{t}(e)h} - 1}{i \mathbf{t}(e)} \d \mathbf{t}.
    \end{align*}
    Note that the above integral exists since $|\CF{\beta}|$ is integrable due to our assumption that $\boldsymbol{\Sigma}_\beta$ is positive definite. 
    Taking $h \rightarrow \infty$ and setting $\mathbf{x}(e) = \tau$ for all $e \in \alpha$ now yields that 
    \begin{align*}
        \Pr{ \mathlarger\cap_{e \in \alpha} \{ \mathbf{z}_\beta(e)  \le \mathbf{x}(e) \} } = \lim_{h \rightarrow \infty} \frac{1}{(2\pi)^{|\alpha|}} \int_{\mathcal{T}}  \CF{\beta} \ e^{-\trp{i\mathbf{t}}{\mathbf{x}_{\tau}}} \prod_{e \in \alpha} \frac{e^{i\mathbf{t}(e)h} - 1}{i \mathbf{t}(e)} \d \mathbf{ t}.
    \end{align*} Multiplying by $-1$ and applying the alternating sum over $\beta \subseteq \alpha$, we get that 
    \begin{align*}
        \Expectedsub{\mathbf{x}}{ \SW{K_{1,\alpha}} \mid (\mathbf{x}_u)_{u \in \alpha} } = \frac{-1}{(2\pi)^{|\alpha|}} \int_{\mathcal{T}} e^{-\trp{i\mathbf{t}}{\mathbf{x}_\tau}}   \left( \sum_{\beta \subseteq \alpha} (-1)^{|\alpha \setminus \beta|} \CF{\beta} \right) \prod_{e \in \alpha} \frac{\onebutnotactuallyone}{i \mathbf{t}(e)} \d \mathbf{ t},
    \end{align*}
    as desired. Notice that the $1$ cancels out after applying the alternating sum over $\beta$.
\end{proof}
 
The rest of the proof is now concerned with squeezing something useful out of the expression from \Cref{lem:swfourier2}. The main idea is to apply a Taylor series to a part of $\CF{\beta}$ and to exploit cancellations occurring afterwards.

\subsection{Expanding the characteristic function and exploiting cancellations}
It is quite convenient for us that the characteristic functions $\CF{\beta}$ have a simple analytic expression given by 

\begin{align}\label{eq:expansioncf}
    \CF{\beta} = \exp\left(-\frac{1}{2}\mathbf{t}^\top \boldsymbol{\Sigma}_\beta \mathbf{t}\right) &= \exp\left(-\frac{\sighat^2}{2}\trp{\mathbf{t}}{\mathbf{t}} - \frac{1}{2}\mathbf{t}^\top \! (\boldsymbol{\Sigma}_\beta - \sighat^2\mathbf{I}) \ \mathbf{t}\right) \nonumber \\
    &= \exp \left(-\frac{\sighat^2}{2}\trp{\mathbf{t}}{\mathbf{t}}\right) \sum_{k=0}^\infty  \frac{(-1)^k}{2^k k!} (\mathbf{t}^\top \! (\boldsymbol{\Sigma}_\beta - \sighat^2\mathbf{I}) \ \mathbf{t})^k \nonumber \\
    &= \exp\left(-\frac{\sighat^2}{2}\trp{\mathbf{t}}{\mathbf{t}}\right) \sum_{k=0}^\infty  \frac{(-1)^k}{2^k k!} \sum_{r_1, \ldots r_k \in \alpha \times \alpha} \ \ \underbrace{ \prod_{j = 1}^k (\boldsymbol{\Sigma}_\beta - \sighat^2\mathbf{I})_{r_j(1), r_j(2)} \mathbf{t}(r_j(1))\mathbf{t}(r_j(2)) }_{\eqqcolon \phiterm{\beta}}
\end{align} where each $r_j$ is a 2-tuple of indices in $\alpha$. 
The crucial part in our analysis now lies in exploiting cancellations that occur once we apply the alternating sum from \Cref{lem:swfourier2} to $\CF{\beta}$. This will have the effect that all terms corresponding to $k < \lceil |\alpha| /2 \rceil$ are zero.
Concretely, given any $r_1, \ldots, r_k \in \alpha \times \alpha$, we recall that \begin{align*}
    \supp(r_1, \ldots, r_k) \coloneqq \{ e \in \alpha \mid r_j(1) = e \text{ or } r_j(2) = e \text{ for some } j \},
\end{align*} and we consider what happens when applying the alternating sum $\sum_{\beta \subseteq \alpha} (-1)^{|\alpha \setminus \beta|}$ to the innermost sum in \eqref{eq:expansioncf}. We define $\gamma \coloneqq \supp(r_1,\ldots,r_k)$ and  $\overline{\gamma} = \alpha \setminus \gamma$ to re-write 
\begin{align*}
     &\sum_{\beta \subseteq \alpha} (-1)^{|\alpha \setminus \beta|}  \phiterm{\beta} = \sum_{\beta_1 \subseteq \gamma} \ \ \sum_{\beta_2 \subseteq \overline{\gamma} } \ (-1)^{|\gamma \setminus \beta_1| + | \overline{\gamma} \setminus \beta_2|}  \phiterm{\beta_1 \cup \beta_2}.
\end{align*} 
Now, we can observe the following. 
\begin{lemma}
    Given any $r_1, \ldots, r_k \in \alpha \times \alpha$, set $\gamma \coloneqq \supp(r_1, \ldots, r_k)$. Then for any $\beta \subseteq \alpha$ and all $\mathbf{t}$, we have 
    \begin{align*}
         \phiterm{\beta} =  \phiterm{\beta \cap \gamma}.
    \end{align*}
\end{lemma}
\begin{proof}
    Follows from the fact that $\boldsymbol{\Sigma}_\beta[\gamma] = \boldsymbol{\Sigma}_{\beta \cap \gamma}[\gamma]$.
\end{proof}
\noindent With the above, our sum simplifies to 
\begin{align*}
     \sum_{\beta \subseteq \alpha} (-1)^{|\alpha \setminus \beta|}  \phiterm{\beta} = \sum_{\beta_1 \subseteq \gamma} (-1)^{|\gamma \setminus \beta_1|}  \phiterm{\beta_1} \left( \sum_{\beta_2 \subseteq \overline{\gamma} } \ (-1)^{ | \overline{\gamma} \setminus \beta_2)|} \right).
\end{align*} Now, we observe that the number of positive and negative terms in the inner most sum is exactly the same, except when $\overline{\gamma} = \emptyset$. This is because it is equal the the difference in the number of odd and even sized subsets of $\overline{\gamma}$. In other words, when evaluating $\sum_{\beta \subseteq \alpha} (-1)^{|\alpha \setminus \beta|} \CF{\beta}$ and expanding $\CF{\beta}$ as in \eqref{eq:expansioncf}, then all terms for which $\supp(r_1, \ldots, r_k) \neq \alpha$ are zero.
Moreover, regarding the terms that do not cancel out (i.e. those for which $\supp(r_1, \ldots, r_k) = \alpha$), we can find some further simplifications enabled by the following observation.
\begin{lemma}
    Consider any $r_1, \ldots, r_k \in \alpha \times \alpha$ such that $\supp(r_1, \ldots, r_k) = \alpha$. Then, for any $\beta \subsetneq \alpha$ and all $\mathbf{t}$, we have 
    \begin{align*}
         \phiterm{\beta} = 0.
    \end{align*}
\end{lemma}
\begin{proof}
    Recall the definition of $\Phi_\beta(r_1, \ldots, r_k, \mathbf{t})$. If $\supp(r_1, \ldots, r_k) = \alpha$, then there is some $r_j$ such that $r_j(1)$ or $r_j(2)$ is not in $\beta$. For such $r_j$, it is easy to see from the definition of $\boldsymbol{\Sigma}_\beta$ that $(\boldsymbol{\Sigma}_\beta - \sighat^2\mathbf{I})_{r_j(1), r_j(2)} = 0$. Hence, the entire product is zero and the lemma follows.
\end{proof}
\noindent Combining all the facts we gathered up to this point, we obtain the following simplification, using in particular that whenever $k < \lceil |\alpha| / 2 \rceil$, then satisfying $\supp(r_1, \ldots, r_k) = \alpha$ is impossible, so the corresponding terms are all zero.

\begin{lemma}\label{lem:aftercancellation}
    Define $\ell \coloneqq  \lceil |\alpha| / 2 \rceil$
    For every $\beta \subseteq \alpha$ and any $\mathbf{t}$, we have
    \begin{align*}
        \sum_{\beta \subseteq \alpha} (-1)^{|\alpha \setminus \beta|}  \hspace{.1cm}\CF{\beta} = e^{-\frac{\sighat^2}{2} \|\mathbf{t}\|^2 } \left( \leadinghat + \resthat \right).
    \end{align*}
    with \begin{align*}
        \leadinghat  \coloneqq  \frac{(-1)^\ell}{2^\ell \ell!}  \sum_{ \substack{ r_1, \ldots r_k \in \alpha \times \alpha \\ \supp(r_1, \ldots, r_k ) = \alpha }} \phiterm{\alpha} \text{ and } \resthat \coloneqq\sum_{\beta \subseteq \alpha} (-1)^{|\alpha \setminus \beta|}\sum_{k=\ell + 1}^\infty \frac{(-1)^k}{2^k k!} (\mathbf{t}^\top \! (\boldsymbol{\Sigma}_\beta - \sighat^2\mathbf{I}) \ \mathbf{t})^k
    \end{align*}
\end{lemma}

\subsection{Integrating out}
Now, it remains to use the facts established so far together with \Cref{lem:swfourier2} to get back to an expression for $ \Expectedsub{\mathbf{x}}{ \SW{K_{1,\alpha}} \mid (\mathbf{x}_u)_{u \in \alpha} }$. 
While \Cref{lem:spectralbound} yields that all the $\boldsymbol{\Sigma}_\beta$ are positive definite such that \Cref{lem:swfourier2} is applicable, we have to be a bit careful while integrating to avoid some convergence issues related to the infinite sums we wish to consider. To this end, define $\T = [-\infty, \infty]^{|\alpha|}$ and note that by \Cref{lem:swfourier2},
\begin{align*}
    \Expectedsub{\mathbf{x}}{ \SW{K_{1,\alpha}} \mid (\mathbf{x}_u)_{u \in \alpha} } & = \frac{-1}{(2\pi)^{|\alpha|}} \lim_{h\rightarrow \infty} \int_{\T} e^{-\trp{i\mathbf{t}}{\mathbf{x}}_{\tau}} \left( \sum_{\beta \subseteq \alpha} (-1)^{|\alpha \setminus \beta|}  \hspace{.1cm}\CF{\beta} \right) \prod_{e \in \mathbf{t}} \frac{\onebutnotactuallyone}{i\mathbf{t}(e)} \d \mathbf{t}. 
\end{align*}
To get what we want out of the above integral, we split 
\begin{align*}
    \Expectedsub{\mathbf{x}}{ \SW{K_{1,\alpha}} \mid (\mathbf{x}_u)_{u \in \alpha} } & = \underbrace{\frac{-1}{(2\pi)^{|\alpha|}} \lim_{h\rightarrow \infty} \int_{\T} e^{-\trp{i\mathbf{t}}{\mathbf{x}}_{\tau}} e^{-\frac{\sighat^2}{2} \|\mathbf{t}\|^2 } \leadinghat \prod_{e \in \mathbf{t}} \frac{\onebutnotactuallyone}{i\mathbf{t}(e)} \d \mathbf{t}}_{\eqqcolon \I_1} \\
    &\hspace{1.5cm} + \underbrace{ \frac{-1}{(2\pi)^{|\alpha|}} \lim_{h\rightarrow \infty} \int_{\T} e^{-\trp{i\mathbf{t}}{\mathbf{x}}_{\tau}} \left( \left( \sum_{\beta \subseteq \alpha} (-1)^{|\alpha \setminus \beta|}  \hspace{.1cm}\CF{\beta} \right) -  e^{-\frac{\sighat^2}{2}\|\mathbf{t}\|^2} \leadinghat \right) \prod_{e \in \mathbf{t}} \frac{\onebutnotactuallyone}{i\mathbf{t}(e)} \d \mathbf{t}}_{\eqqcolon \I_2}.  
\end{align*}

\noindent The first term evaluates explicitly and is equal to the leading term $\leading$ from \Cref{prop:fourierexpressionthing}.
\begin{lemma}[Leading term]\label{lem:smalltleading}
    \begin{align*}
        \I_1  =\lim_{h\rightarrow \infty} \frac{-1}{(2\pi)^{|\alpha|}}\int_{\T} e^{-\trp{i\mathbf{t}}{\mathbf{x}}_{\tau}} e^{-\frac{\sighat^2}{2} \| \mathbf{t} \|^2 } \  \leadinghat \prod_{e \in \mathbf{t}} \frac{\onebutnotactuallyone}{i\mathbf{t}(e)} \d \mathbf{t} = \leading.
    \end{align*}
\end{lemma}
\begin{proof}
    Since the sum is finite, we can exchange sum and integral to get 
    $$
        \I_1  =\frac{(-1)^\ell}{2^\ell \ell!}  \sum_{ \substack{ r_1, \ldots r_\ell \in \alpha \times \alpha \\ \supp(r_1, \ldots, r_\ell ) = \alpha }}  \lim_{h\rightarrow \infty} \frac{-1}{(2\pi)^{|\alpha|}}\int_{[-\infty, \infty]^{|\alpha|}} e^{-\frac{\sighat^2}{2} \| \mathbf{t} \|^2 } \phitermell{\alpha} \prod_{e \in \mathbf{t}} \frac{\onebutnotactuallyone}{i\mathbf{t}(e)} \d \mathbf{t}.
    $$
    Using the definition of $\phitermell{\alpha}$, the integral now evaluates explicitly as 
    \begin{align*}
        &\frac{-1}{(2\pi)^{|\alpha|}}\int_{[-\infty, \infty]^{|\alpha|}} e^{-\trp{i\mathbf{t}}{\mathbf{x}}_{\tau}}e^{-\frac{\sighat^2}{2} \| \mathbf{t} \|^2 } \phitermell{\alpha} \prod_{e \in \mathbf{t}} \frac{\onebutnotactuallyone}{i\mathbf{t}(e)} \d \mathbf{t} \\
        &\hspace{2cm}= \frac{(-1)^{\ell+1}}{(2\pi)^{|\alpha|}} \Bigg(\prod_{j = 1}^\ell (\boldsymbol{\Sigma}_\alpha - \sighat^2\mathbf{I})_{r_j(1), r_j(2)} \Bigg) \int_{[-\infty, \infty]^{|\alpha|}} e^{-\frac{\sighat^2}{2} \| \mathbf{t} \|^2 } \prod_{e \in \alpha} (i\mathbf{t}(e))^{s_e-1} \left(e^{-i\mathbf{t}(e)\tau} - e^{-i\mathbf{t}(e)(\tau - h)}\right) \d \mathbf{ t}
    \end{align*}
    where we ask the reader to recall the definition of $s_e$ from \Cref{prop:fourierexpressionthing}. Notice further how we used $\mathbf{t}(r_j(1))\mathbf{t}(r_j(2)) = - (i\mathbf{t}(r_j(1)))(i\mathbf{t}(r_j(2)))$. Because $\frac{1}{2\pi}\int_{-\infty}^\infty e^{-\sighat^2 \mathrm{t}^2/2 }  (i\mathrm{t})^{s} e^{-i\mathrm{t}\tau} \d \mathrm{t} = \phi_{\sighat}^{(s)}(\tau)$, the integral now evaluates explicitly and can be expressed in terms of a product over derivatives of the Gaussian density with variance $\sighat$
    \begin{align*}
        &\frac{1}{(2\pi)^{|\alpha|}}\int_{[-\infty, \infty]^{|\alpha|}} e^{-\frac{\sighat^2}{2} \| \mathbf{t} \|^2 } \prod_{e \in \alpha} (i\mathbf{t}(e))^{s_e-1} \left(e^{-i\mathbf{t}(e)\tau} - e^{-i\mathbf{t}(e)(\tau-h)}\right) \d \mathbf{ t} \\
        &\hspace{3cm} = \prod_{e \in \alpha} \frac{1}{2\pi}\int_{-\infty}^\infty e^{-\frac{\sighat^2}{2} \mathrm{t}^2 }  (i\mathrm{t})^{s_e-1} \left(e^{-i\mathrm{t}\tau} - e^{-i\mathrm{t}(\tau - h)}\right) \d \mathrm{t} =  \prod_{e \in \alpha} \left( \phi_{\sighat}^{(s_e - 1)}(\tau) - \phi_{\sighat}^{(s_e - 1)}(\tau - h) \right).
    \end{align*}
    Taking the limit over $h \rightarrow \infty$, the second term (i.e. $\phi_{\sighat}^{(s_e - 1)}(\tau - h)$) in each factor vanishes and we end up with $\leading$, as desired.
\end{proof}

Regarding $\I_2$, we use a different strategy and split the integral into two parts based on $\|\mathbf{t}\|\le T$ and $\| \mathbf{t}\| > T$ where $T$ is some large enough value (as a function of $d, |\alpha|$) to be chosen later. To split the integral, we set 
$$
    f(\mathbf{t}) \coloneqq \left( \sum_{\beta \subseteq \alpha} (-1)^{|\alpha \setminus \beta|}  \hspace{.1cm}\CF{\beta} \right) -  e^{-\frac{\sighat^2}{2}\|\mathbf{t}\|^2} \leadinghat 
$$ and bound
\begin{align*}
    | \I_2 | \le \int_{\| \mathbf{t} \| < T} | f(\mathbf{t}) | \d \mathbf{t} + \int_{\| \mathbf{t} \| \ge T} | f(\mathbf{t}) | \d \mathbf{t}.
\end{align*}
Notice that we can omit the limit over $h$, as $\left|\prod_{e \in \mathbf{t}} \frac{\onebutnotactuallyone}{i\mathbf{t}(e)} \d \mathbf{t}\right| \le 1$ for all $h$, so any bound on the above expression implies the same bound the limiting expression over $h \rightarrow \infty$. We bound the two parts above now separately.

\begin{lemma}[Bound for large $\|\mathbf{t}\|$]\label{lem:larget2}
    $$
        \lim_{T \rightarrow \infty }\int_{\| \mathbf{t} \| \ge T} | f(\mathbf{t}) | \d \mathbf{t} = 0. 
    $$
\end{lemma}
\begin{proof}
    Bound
    \begin{align*}
        \int_{\| \mathbf{t} \| \ge T} | f(\mathbf{t}) | \d \mathbf{t} \le \int_{\| \mathbf{t} \| \ge T} e^{-\frac{\sighat^2}{2}\|\mathbf{t}\|^2}  \left| \leadinghat \right| \d \mathbf{t} + \sum_{\beta \subseteq \alpha} \int_{\| \mathbf{t} \| \ge T} \left| \CF{\beta} \right| \d \mathbf{t}.
    \end{align*}
    To see that the limit over $T \rightarrow \infty$ for the first term is $0$, it suffices to note that $\leadinghat$ is a polynomial in $\mathbf{t}$, so the integral over all of $\T$ would be finite and thus  tends to zero once we restrict $\| \mathbf{t} \| \ge T$ and send $T \rightarrow \infty$. For the second term, we use \Cref{lem:spectralbound} to get $ \opnorm{\boldsymbol{\Sigma}_\beta - \mathbf{I} } \le \delta(\varepsilon)$ and thus $
        \mathbf{t}^\top \boldsymbol{\Sigma}_\beta \mathbf{t} = \|\mathbf{t}\|^2 + \mathbf{t}^\top \left( \boldsymbol{\Sigma}_\beta - \mathbf{I} \right) \mathbf{t} \ge (1 - \delta(\varepsilon))\|\mathbf{t}\|^2.
    $ Hence for $\varepsilon$ small enough, we get \begin{align*}
        |\CF{\beta}| &= e^{-\frac{1}{2} \mathbf{t}^\top \boldsymbol{\Sigma}_\beta \mathbf{t}} \le e^{-\frac{1}{4} \| \mathbf{t}\|^2} \text{ and thus } \lim_{T \rightarrow \infty} \sum_{\beta \subseteq \alpha} \int_{\| \mathbf{t} \| \ge T} \left| \CF{\beta} \right| \d \mathbf{t} = 0. \qedhere
    \end{align*}
\end{proof}

\noindent For $\|\mathbf{t}\| \le T$, we use the fact that $\resthat$ is equal to a Taylor expansion of $f(\mathbf{t})$, which is uniformly convergent for all $\mathbf{t} \le T$. Once we restrict $\|\mathbf{t}\| \le T$, it is therefore not hard to argue that we can exchange the sum and the integral. Each integral in the resulting sum can then be bounded explicitly such that the entire sum is convergent and small enough. We capture this in the following lemma.

\begin{lemma}[Bound for small $\|\mathbf{t}\|$]\label{lem:smallt2}
    There is an absolute constant $C > 0$ such that for every $T \in \mathbb{R}_{\ge 0}$,
    \begin{align*}
        \int_{\| \mathbf{t} \| < T} | f(\mathbf{t}) | \d \mathbf{t} = \int_{\| \mathbf{t} \| < T}  e^{-\frac{\sighat^2}{2}\|\mathbf{t}\|^2} \left| \resthat \right| \d \mathbf{t} \le \Bigg( \frac{C|\alpha| \rho}{\sqrt{d}} \Bigg)^{\ell + 1}.
    \end{align*}
\end{lemma}
\begin{proof}
    We apply Cauchy-Schwarz within the definition of $\resthat$ together with the spectral norm bound from \Cref{lem:spectralbound} to get  
    \begin{align*}
         \int_{\| \mathbf{t} \| < T}  e^{-\frac{\sighat^2}{2}\|\mathbf{t}\|^2} \left| \resthat \right| \d \mathbf{t} &\le  \sum_{\beta \subseteq \alpha} \sum_{k = \ell + 1}^\infty \frac{1}{2^k k!} \int_{[-\infty, \infty]^{|\alpha|}} e^{-\frac{\sighat^2}{2}\|\mathbf{t}\|^2} \|\mathbf{t}\|^{2k} \opnorm{\boldsymbol{\Sigma}_\beta - \sighat^2\mathbf{I}}^k \d \mathbf{t} \\
         &\le  2^{|\alpha|} \sum_{k = \ell + 1}^\infty \Bigg( \frac{C|\alpha| \rho }{k\sqrt{d}} \Bigg)^k  \int_{[-\infty, \infty]^{|\alpha|}} e^{-\frac{\sighat^2}{2}\|\mathbf{t}\|^{2}} \|\mathbf{t}\|^{2k} \d \mathbf{t}
    \end{align*}
    where in the last step, we used Stirling's approximation to take $k!$ into the exponential term. The integral above can now be bounded by exploiting the spherical symmetry, i.e., by rewriting  
    \begin{align*}
         &\int_{[-\infty, \infty]^{|\alpha|}} e^{-\frac{\sighat^2}{2}\|\mathbf{t}\|^{2}} \|\mathbf{t}\|^{2k} \d \mathbf{t} = \text{Vol}(\mathbb{S}^{|\alpha|-1}) \int_\infty^\infty e^{-\frac{\sighat^2}{2}t^2} t^{2k + |\alpha| - 1} \d t \\
         &\hspace{3.2cm}\le C \ \text{Vol}(\mathbb{S}^{|\alpha|-1}) \int_\infty^\infty e^{-s} s^{\frac{2k + |\alpha|}{2} - 1} \d s = C \ \text{Vol}(\mathbb{S}^{|\alpha|-1}) \ \Gamma\left( \frac{2k + |\alpha|}{2} \right).
    \end{align*} where $C> 0$ is a constant and $\text{Vol}(\mathbb{S}^{|\alpha|-1})$ is the volume of the $|\alpha|$-dimensional unit sphere, which can be bounded as 
    $
        \text{Vol}(\mathbb{S}^{|\alpha|-1}) \le ( C/|\alpha|)^{|\alpha|/2}
    $. Plugging all this back into our sum over $k$, we get that 
    \begin{align*}
        \int_{\| \mathbf{t} \| < T}  e^{-\frac{\sighat^2}{2}\|\mathbf{t}\|^2} \left| \resthat \right| \d \mathbf{t} &\le 2^{|\alpha|} \sum_{k=\ell + 1}^\infty \Bigg( \frac{C|\alpha| \rho }{k\sqrt{d}} \Bigg)^k \left( \frac{C}{|\alpha|}\right)^{\frac{|\alpha|}{2}} \left( Ck \right)^{\frac{2k +|\alpha|}{2}}\\ 
        &\le 2^{|\alpha|} \sum_{k=\ell + 1}^\infty \Bigg( \frac{C^2|\alpha| \rho }{\sqrt{d}} \Bigg)^{k}  \left( \frac{C^2k}{|\alpha|}\right)^{\frac{|\alpha|}{2}}.
    \end{align*}
    Now, by our assumption that $\dassump$ we get that the base of the first factor in the sum is at most $C^2\sqrt{\varepsilon}$. Hence, the ratio of two consecutive terms in the above sum is at most 
    $$
        C^2\sqrt{\varepsilon} \ \left( \frac{k+1}{k} \right)^{\frac{|\alpha|}{2}} = C^2\sqrt{\varepsilon} \ \left( 1 + \frac{1}{k} \right)^{\frac{|\alpha|}{2}} \le \frac{1}{2}
    $$ for sufficiently small $\varepsilon$, where we used that $k \ge \ell \ge  |\alpha|/2$, so the second factor is at most a constant. Therefore, the sum is geometric and dominated by its first term, which provides our explicit upper bound as stated in the lemma.
\end{proof}

Finally, we can stack everything toghether into a proof of \Cref{prop:fourierexpressionthing}.

\begin{proof}[Proof of \Cref{prop:fourierexpressionthing}]
    Using \Cref{lem:swfourier2}, we get $\Expectedsub{\mathbf{x}}{ \SW{K_{1,\alpha}} \mid (\mathbf{x}_u)_{u \in \alpha} } = \I_1 + \I_2$ with $\I_1, \I_2$ as defined at the beginning of this section. Positive definiteness of $\boldsymbol{\Sigma}_\beta$ follows by the spectral bound from \Cref{lem:spectralbound} for sufficiently small $\varepsilon$. By \Cref{lem:smalltleading}, we now have $\I_1 = \leading$, which can be bounded in absolute value by noting that the sum appearing in $\leading$ has at most $|\alpha|^{2\ell}$ terms, while each summand can be bounded by $C\rho /\sqrt{d}$ due to the conditioning on $\Seventshort$. Hence, using that $\ell! \ge (\frac{\ell}{2e})^\ell$ and $\ell \ge |\alpha|/2$, we get that $$
        |\leading| \le \left( \frac{C|\alpha|\rho}{\sqrt{d}} \right)^{\ell}.
    $$    
    Now, it only remains to argue that $$
        | \I_2 | \le \int_{\| \mathbf{t} \| < T} | f(\mathbf{t}) | \d \mathbf{t} + \int_{\| \mathbf{t} \| \ge T} | f(\mathbf{t}) | \d \mathbf{t} \le \Bigg( \frac{C|\alpha| \rho}{\sqrt{d}} \Bigg)^{\ell + 1}.
    $$ 
    To this end, we choose $T$ large enough such that the integral over $\|\mathbf{t}\| \ge T$ is negligibly small compared to the above bound we wish to obtain. This is possible by \Cref{lem:larget2}. Once we fix any such $T$ and consider the the integral over $\|\mathbf{t}\| \le T$, we get the desired bound from \Cref{lem:smallt2}. 
\end{proof}

\bibliography{literature}

\section{Deferred Proofs}\label{sec:deferredproofs}

\begin{proof}[Proof of \Cref{clm:stars}]
        Recall that $\sighat$ is defined such that $\Pr{Z \le \tau} = \Prnop{\frac{1}{\sqrt{d}} \langle \mathbf{x}_i, \mathbf{x}_j \rangle \le \tau }$ with $Z \sim \mathcal{N}(0, \sighat^2)$.
        Given a fixed $\mathbf{y} \in \mathbb{R}^d$ that represents the latent vector associated to the center of our star, we further get that for any fixed $\mathbf{y}$ and over the randomness of $\mathbf{x} \sim \mathcal{N}(0, \mathbf{I}_d)$, we get $\frac{1}{\sqrt{d}} \langle \mathbf{y}, \mathbf{x} \rangle \sim\mathcal{N}(0, s(\mathbf{y})^2 )$ over $\mathbf{x} \sim \mathcal{N}(0, \mathbf{I}_d)$, where $s(\mathbf{y}) \coloneqq \sqrt{\frac{1}{d} \langle \mathbf{y}, \mathbf{y} \rangle}$. Hence, by \Cref{lem:divergenceofgaussianswithdifferentvariances}, we get that whenever $|s(\mathbf{y})^2 - 1| \le \delta$ for some small constant $\delta$, then since $\Phi_{\sighat}(\tau)) = p$,
        $$
            \left| \Expected{ \SW{K_{1, \ell}} \mid \mathbf{y} }\right| = \big| \Phi_{s(\mathbf{y})}(\tau)  - \Phi_{\sighat}(\tau))\big|^\ell \le C^\ell| s(\mathbf{y})^2 - \sighat^2 |^\ell.
        $$
        From a Bernstein bound, we get that there is a constant $c > 0$ such that 
        $
            \Prnop{ |s(\mathbf{y})^2 - 1| \ge t } \le 2 \exp(-cd t^2 ).
        $ Because $|\sighat^2 - 1| \le C/\sqrt{d}$ by \Cref{lem:sigbound}, we hence get
        \begin{align*}
            \Pr{ |s(\mathbf{y})^2 - \sighat^2| \ge t } &\le  \Pr{ |s(\mathbf{y})^2 - 1| \ge t - \tfrac{C}{\sqrt{d}} } \le 2 \exp \left(-cd \big(t - C/\sqrt{d}\big)^2 \  \right) = 2\exp \left(-c \big(t\sqrt{d} - C\big)^2   \right).
        \end{align*}
        With this, integration by tails yields  
        \begin{align*}
            \left| \Expected{ \SW{K_{1, \ell}}} \right| &\le \left| \Expected{ \SW{K_{1, \ell}} \mid |s(\mathbf{y})^2 - 1| \le 
            \delta}\right| + \Pr{|s(\mathbf{y})^2 - 1| \ge \delta } \\&\le \frac{1}{\Pr{|s(\mathbf{y})^2 - 1| \le 
            \delta}} \int_{0}^{\delta} \Pr{ (C| s(\mathbf{y})^2 - \sighat^2 |)^\ell  > t } \d t  + \exp( - c \delta^2 d ) \\
            &\le 2 \int_{0}^{\infty} 2 \exp\left(-c \left(\frac{t^{1/\ell}\sqrt{d}}{C} - C\right)^2 \right) \d t + \exp( - c \delta^2 d ).
        \end{align*}
        Now, since for $t \ge (2C^2/\sqrt{d})^{\ell} \eqqcolon a$, we have $(t^{1/\ell}\sqrt{d}/C) - C \ge \frac{1}{2} (t^{1/\ell}\sqrt{d}/C)$, we get 
        \begin{align*}
            \int_{0}^{\infty} 2 \exp\left(-c \left(\frac{t^{1/\ell}\sqrt{d}}{C} - C\right)^2 \right) \d t &\le 2a + \int_{0}^{\infty} 2 \exp\left(- \frac{c}{4 C^2} t^{2/\ell}d \right) \d t\\
            &\le 2a + 2 \left( \frac{4C^2}{c}\frac{1}{d} \right)^{\ell/2} \int_{0}^{\infty} s^{\frac{\ell-2}{2}} e^{-s} \d s = 2a + 2 \left( \frac{4C^2}{c}\frac{1}{d} \right)^{\ell/2} \Gamma\left( \frac{\ell}{2} \right).
        \end{align*} Hence, in total there is a constant $C' \ge 0$ such that 
        \begin{align*}
             \left| \Expected{ \SW{K_{1, \ell}}} \right| \le \left( \frac{C' \ell}{d} \right)^{\ell/2} + e^{-d/C'}.
        \end{align*} We now claim that there is a constant $C > 0$ such that the above is at most $(C\ell / d)^{\ell/2}$ for all $\ell$. To this end, it suffices to show that for all $\ell \ge 1$,
        \begin{align*}
            e^{-d/C'} \le \left( \frac{C \ell}{d} \right)^{\ell/2} \Leftrightarrow \exp \left( \log\left( \frac{d}{C\ell} \right)\frac{\ell}{2} - \frac{d}{C'} \right) \le 1 \Leftrightarrow \log\left( \frac{d}{C\ell} \right)\frac{\ell}{2} \le \frac{d}{C'} \Leftrightarrow \log\left( \frac{d}{C\ell} \right)\frac{C\ell}{2d} \le 1,
        \end{align*}
        which is true for all $\ell$ once $C$ is large enough.
    \end{proof}

\begin{proof}[Proof of \Cref{clm:variance}]
    Assume for now that $k$ is even and consider $$
    f_{\alpha}(\xrnum{1}, \xrnum{2}) = m \left( \frac{ 2 q^{2} }{(p(1-p))}\right)^{k} \sum_{\alpha \subseteq R, |\alpha| = k} \leadingnum{1}\leadingnum{2}$$. Since $\Expected{ f_{\alpha}(\xrnum{1}) }= 0$ by definition of $\leading$, we get that 
    \begin{align*}
        \sigma^2_{f_{\alpha}}(\xrnum{2})) = \Expectedsubb{\xrnum{1}}{ f_{\alpha}(\xrnum{1}, \xrnum{2})^2 } =  m^2 \left( \frac{ 2 q^{2} }{(p(1-p))}\right)^{2k} \sum_{\substack{ \alpha_1, \alpha_2 \subseteq R\\ |\alpha_1| = |\alpha_2| = k}} \Expectedsubb{\xrnum{1}}{ \leadingnumalpha{1}{\alpha_1} \leadingnumalpha{1}{\alpha_2} } \leadingnumalpha{2}{\alpha_1} \leadingnumalpha{2}{\alpha_2}.
    \end{align*} 
    Recalling the definition of $\leading$, it is easy to see that the expectation in the above sum is $0$ whenever $\alpha_1 \neq \alpha_2$. This is because when expanding $\leadingnum{1}, \leadingnum{2}$, all the $r_j$ in every term must correspond to disjoint pairs of elements in $\alpha_1$ and $\alpha_2$, respectively, because otherwise (as $k$ is even and there are $\ell = k/2$ tuples $r_j$) we could not satisfy the coverage constraint $\supp(r_1, \ldots, r_\ell) = \alpha$. Hence any latent vector $\mathbf{x}_u$ for $u \in \alpha_1 \triangle \alpha_2$\footnote{$\triangle$ denotes the symmetric difference} appears only in one inner product corresponding to exactly one $r_j$, and sets the entire  expression  to zero after taking the expectation.
    Therefore,
    \begin{align*}
        \sigma^2_{f_{\alpha}}(\xrnum{2})) = m^2 \left( \frac{ 2 q^{2} }{(p(1-p))}\right)^{2k} \sum_{\substack{ \alpha \subseteq R\\ |\alpha| = k}} \Expectedsubb{\xrnum{1}}{ \leadingnumalpha{1}{\alpha}^2 } \leadingnumalpha{2}{\alpha}^2.
    \end{align*}
    Moreover,  
    \begin{align}\label{eq:variancelambda}
        &\Expectedsubb{\xrnum{1}}{ \leadingnumalpha{1}{\alpha}^2 } = \frac{ 1 }{2^{2\ell} (\ell!)^2 } \sum_{ \substack{ r_1, \ldots r_{\ell} \in \alpha\times \alpha \\ \supp(r_1, \ldots, r_\ell) = \alpha }} \sum_{ \substack{ r_1', \ldots r_{\ell}' \in \alpha\times \alpha \\ \supp(r_1', \ldots, r_\ell') = \alpha }} \ \Bigg( \prod_{e \in \alpha}\phi^{(s_e -1)}(\tau) \Bigg) \Bigg( \prod_{e \in \alpha}\phi^{(s_e' -1)}(\tau) \Bigg) \nonumber \\
        & \hspace{4cm} \times \Expectedsubb{\xr}{ \prod_{j = 1}^\ell \Big( \frac{1}{d} \big\langle \mathbf{x}_{r_j(1)} ,\mathbf{x}_{r_j(2)} \big\rangle - \sighat^2 \mathbf{I}_{r_j(1),r_j(2)}\Big) \Big( \frac{1}{d} \big\langle \mathbf{x}_{r_j'(1)} ,\mathbf{x}_{r_j'(2)} \big\rangle - \sighat^2 \mathbf{I}_{r_j'(1),r_j'(2)}\Big) }.
    \end{align}
    Regarding the expectation, given any fixed $r_1, \ldots, r_\ell$ and $r_1', \ldots, r_\ell'$, we wish to integrate by tails and we use the fact that the integrand is a product of $2\ell = k$ random variables $X_1, \ldots, X_k$ where each $X_i$ has the form 
    $$
        X_i = \tfrac{1}{d} \big\langle \mathbf{x}_{\zeta(1)} ,\mathbf{x}_{\zeta(2)} \big\rangle - \sighat^2 \mathbf{I}_{\zeta(1),\zeta(2)}
    $$ where $\zeta$ is one of the $r_1, \ldots, r_\ell, r_1', \ldots, r_\ell'$. Now, we bound 
    \begin{align*}
        \Expected{{\prod}_{j=1}^k X_j} \le \Expected{{\prod}_{j=1}^k |X_j| } \le \Expectednop{Y^k } \text{ for } Y \coloneqq \max_{j \in [k]} |X_j|.
    \end{align*} 
    Now, it is not hard to see that by standard Chernoff/Bernstein bounds, we have 
    $$
        \Pr{ |X_i| \ge t } \le 2 \exp\left( -cd \big( t - C/\sqrt{d} \big)^2 \right)
    $$ where $c > 0$ is a constant\footnote{Note that we use $t - C/\sqrt{d}$ instead of just $t$ in the exponent since the individual $X_i$ are only zero in expectation if $\zeta(1) \neq \zeta(2)$. Otherwise (if $\zeta(1) = \zeta(2)$), the expectation is at most $C/\sqrt{d}$ in absolute value, as follows from \Cref{lem:sigbound}. This allows us to use $t - C/\sqrt{d}$ in the exponent. }. Using a union bound, this implies 
    \begin{align*}
        \Pr{ Y \ge t } \le 2k \exp( -cd t^2 ).
    \end{align*} Now, integrating by tails and using that $t - C/\sqrt{d} \ge t/2 $ for $t \ge 2C/\sqrt{d}$, we get that 
    \begin{align*}
        \Expectednop{Y^k } = \int_{0}^\infty  \Pr{Y \ge t^{1/k}} \d t&\le \left(\frac{2C}{\sqrt{d}}\right)^{k} + 2k \int_{0}^\infty  \exp( -cd \big( t/2 \big)^{2/k} ) \d t \\ 
        & = \left(\frac{2C}{\sqrt{d}}\right)^{k} + 2k^2 \left( \frac{1}{cd} \right)^{\frac{k}{2}} \int_{0}^\infty  \exp( -s ) s^{\frac{k-2}{2}} \d t
        \le \left(\frac{2C}{\sqrt{d}}\right)^{k} + 2k^2 \left( \frac{1}{cd} \right)^{\frac{k}{2}} \Gamma\left( \frac{k}{2} \right) \le k^2 \left( \frac{C k}{\sqrt{d}} \right)^k,
    \end{align*} for some constant $C > 0$. 

    This yields a bound on the expectation in \eqref{eq:variancelambda}. 
    To further simplify the terms in front of the expectation, we note that $s_e = s_s' = 1$ for all $\alpha$ since every $e \in \alpha$ is covered exactly once (because $k $ is even). Moreover, each sum has at most $k!$ terms. Using Stirling's approximation for all the factorials, we then get that there is a constant $C> 0$ such that 
    \begin{align*}
        &\Expectedsubb{\xrnum{1}}{ \leadingnumalpha{1}{\alpha}^2 } \le \left( Ck \right)^k k^2 \left( \frac{C k}{\sqrt{d}} \right)^k.
    \end{align*}
    Combining this bound with the bounds on $\leadingnum{2}$ from \Cref{prop:fourierexpressionthing}, we get the total variance bound 
    \begin{align*}
        \sigma^2_{f_{\alpha}}(\xrnum{2})) \le m^2  q^{4k} \binom{n}{k} \ \left( Ck \right)^k k^2 \left( \frac{C k}{\sqrt{d}} \right)^k  \Bigg( \frac{C k \log(n)^{\frac{1}{2}}}{\sqrt{d}} \Bigg)^{k} \le m^2 n^k q^{4k} k^2 \Bigg( \frac{C' k \log(n)^{\frac{1}{2}}}{\sqrt{d}} \Bigg)^{2k},
    \end{align*} as desired.

    It only remains to handle the case of $k$ being odd. To this end, recall the expansion of the variance used at the beginning, i.e., 
    \begin{align*}
        \sigma^2_{f_{\alpha}}(\xrnum{2})) =  m^2 \left( \frac{ 2 q^{2} }{(p(1-p))}\right)^{2k} \sum_{\substack{ \alpha_1, \alpha_2 \subseteq R\\ |\alpha_1| = |\alpha_2| = k}} \Expectedsubb{\xrnum{1}}{ \leadingnumalpha{1}{\alpha_1} \leadingnumalpha{1}{\alpha_2} } \leadingnumalpha{2}{\alpha_1} \leadingnumalpha{2}{\alpha_2}.
    \end{align*}
    The main difference to the case where $k$ is even is that now, not all terms for which $\alpha_1 \neq \alpha_2$ are zero. This is because now, we have $\lceil k/2 \rceil > k/2$ tuples $r_j$ at disposition when expanding the above expectation. Hence, in case $|\alpha_1 \triangle \alpha_2| = 2$ with elements $u_1 \in \alpha_1\setminus\alpha_2$ and $u_2 \in \alpha_2\setminus\alpha_1$, we can set some $r_j = (u_1, u_1)$ and some $r_j' = (u_2, u_2)$ while using the remaining $r_j$ and $r_j'$ to cover the remaining (evenly many) elements in $\alpha_1 \cap \alpha_2$. In this case, the expectation would not be zero. 

    However, we can still guarantee that the expectation is zero whenever $|\alpha_1 \triangle \alpha_2| \ge 4$, since this property guarantees that for at least one of the $u \in \alpha_1 \triangle \alpha_2$, the latent vector $\mathbf{x}_u$ only appears in once in the inner products, otherwise the coverage constraint $\supp(r_1, \ldots, r_\ell)$ would be impossible to satisfy. Hence, in this case, we can bound 
    \begin{align*}
        \sigma^2_{f_{\alpha}}(\xrnum{2})) = m^2 \left( \frac{ 2 q^{2} }{(p(1-p))}\right)^{2k} \sum_{\substack{ \alpha_1, \alpha_2 \subseteq R\\ |\alpha_1| = |\alpha_2| = k \\ |\alpha_1 \triangle \alpha_2| \le 2}} \Expectedsubb{\xrnum{1}}{ \leadingnumalpha{1}{\alpha_1}\leadingnumalpha{1}{\alpha_2} } \leadingnumalpha{2}{\alpha_1}\leadingnumalpha{2}{\alpha_2}.
    \end{align*}
    Now, each of the expectations appearing above, and the terms $\leadingnumalpha{2}{\alpha_1}\leadingnumalpha{2}{\alpha_2}$ can be bounded as before. This yields 
    \begin{align*}
        \sigma^2_{f_{\alpha}}(\xrnum{2})) \le m^2  q^{4k} n \binom{n}{k} \ \left( Ck \right)^{2\ell} k^2 \left( \frac{C k}{\sqrt{d}} \right)^{2\ell} \Bigg( \frac{C k \log(n)^{\frac{1}{2}}}{\sqrt{d}} \Bigg)^{2\ell} \le m^2 n^{k+1} q^{4k}  k^3\Bigg( \frac{C' k \log(n)^{\frac{1}{2}}}{\sqrt{d}} \Bigg)^{4\ell},
    \end{align*} as desired (note that $k + 1 = 2\ell$ for odd $k$, and that the extra factor of $nk$ above accounts for the number of possible $\alpha_2$ given a fixed $\alpha_1$). \qedhere

\end{proof}

\section{Algorithmic upper bounds }\label{sec:upperbounds}

\newcommand{\Varsub}[2]{ \underset{#1}{ \mathbb{V}\text{ar} } \left[ #2 \right] }
It remains to argue about lower bounds on $\dtv{\cdot}{\cdot}$ by giving efficient tests.
\begin{theorem}[Signed four-cycles]
    Consider any fixed $p \in (0,1)$. Then, counting signed four-cycles distinguishes $\matnoiseedge$ from $\gaussmat$ whenever $\log(n)^3 \ll d \ll \boundlower$. The same holds for distinguishing $\matnoiseknown$ from $\gaussmat$ if $\ \log(n)^3 \ll d \ll \boundknownlower$ and if we constrain the signed four-cycles to be within the mask.
\end{theorem}
\begin{proof}
    We argue via a simple application of Chebyshev's inequality. To this end, denote by $$
        C_4(M) = \sum_{i,j \in [n], k, \ell \in [m]} (M_{i, k} - p)(M_{j, k} - p)(M_{i, \ell} - p)(M_{j, \ell} - p)
    $$ the signed-four cycle count of a matrix $M$. To bound the expectation of $C_4(\matnoiseedge)$ from below, we apply \Cref{prop:fourierexpressionthing} and condition on the event that the right sided latent vectors are in $\Seventshort$ for $\rho = C \log(n)^{1/2}$ such that the error arising from conditioning is at most $n^{-10}$. Then, applying \Cref{prop:fourierexpressionthing} and averaging over the leading term $\leading$, we get that there is a constant $C > 0$ such that for all $d \gg \log(n)^3$,
    \begin{align*}
        \Expected{C_4(\matnoiseedge)} = q^4 \Expected{\SW{C_4}} \ge \frac{C q^4 }{d}.
    \end{align*}
    
    Similarly, we can estimate the variance of $C_4(M)$ for $M \sim \matnoiseedge$. To this end, expand 
    $$
        \Varsub{M \sim \matnoiseedge}{ C_4(M) } = \sum_{\alpha_1, \alpha_2} \Expected{ \prod_{e \in \alpha_1} (M_e - p)  \prod_{e \in \alpha_2} (M_e - p) } - \Expectedsub{M \sim \matnoiseedge}{ C_4(M) }^2, 
    $$ the sum goes over all $\alpha_1, \alpha_2 \in K_{n, m}$ isomorphic to a four-cycle $C_4$. Conditioning on the appropriate vertices on the left or the right and applying \Cref{prop:fourierexpressionthing}, we can consider every possible overlap of $\alpha_1, \alpha_2$, and explicitly bound the contribution of the corresponding terms to the variance. To this end, we will condition either the right or left sided latent vectors to be in $\Seventshort$ for $\rho = C \log(n)^{1/2}$ such that the error arising from conditioning is at most $n^{-10}$, which is negligible in all terms. Then, we can apply \Cref{prop:fourierexpressionthing} and exploit the symmetries arising from the leading terms $\leading$ like we did in \Cref{sec:ponehalfcase} to zero-out some of the leading terms. 
    \begin{enumerate}
        \item If $|V(\alpha_1) \cap V(\alpha_2)| = 0$, then the contribution to the variance is zero.
        \item If $|V(\alpha_1) \cap V(\alpha_2)| = 1$, then the contribution is at most $
        n^3m^4 q^8 \left( C\log(n)/d \right)^2$.
        \item If $|V(\alpha_1) \cap V(\alpha_2)| = 2$, then the contribution is at most $
        n^3m^3 q^6 \left( C\log(n)/d \right)^{2}$.
        \item If $|V(\alpha_1) \cap V(\alpha_2)| = 3$, then the contribution is at most $
        n^2m^3 q^5 C\log(n)/d$.
        \item If $|V(\alpha_1) \cap V(\alpha_2)| = 4$, then the contribution is at most $
        n^2m^2$.
    \end{enumerate}
    Moreover, it is not hard to see that $\Varsub{M \sim \matnoiseedge}{ C_4(M) } = n^2m^2$.
    Combining all this, we get that, whenever $d \ll \boundlower$, then 
    $$
        n^2m^2\Expected{C_4(\matnoiseedge)} \gg  \sqrt{\max \left\{ \Varsub{M \sim \matnoiseedge}{ C_4(M) } , \Varsub{M \sim \gaussmat}{ C_4(M) }  \right\}},
    $$
    as desired. Regarding $\matnoiseknown$, the analysis is similar with the difference that the variance is lower in terms of its dependence on $q$. In particular, the factor of $q^5$ for $|V(\alpha_1) \cap V(\alpha_2)| = 3$ turns into $q^7$ and in the $|V(\alpha_1) \cap V(\alpha_2)| = 4$, the contribution is at most $n^2q^2$ for both $\matnoiseknown$ and $\gaussmat$. These changes then yield the lemma using otherwise the same arguments as above. 
\end{proof}

For the $p \neq \frac{1}{2}$ case, the following shows that signed wedges provide an efficient test for $d \ll \boundwedgeslower$ and $d \ll \boundwedgesknownlower$, respectively.
\begin{theorem}[Signed wedges]
    Consider any fixed $p \in (0,1), p \neq \frac{1}{2}$. Then, counting signed wedges distinguishes $\matnoiseedge$ from $\gaussmat$ whenever $\log(n)^3 \ll d \ll \boundwedgeslower$. The same holds for distinguishing $\matnoiseknown$ from $\gaussmat$ if $\ \log(n)^3 \ll d \ll \boundwedgesknownlower$ and if we constrain the signed four-cycles to be within the mask.
\end{theorem}
\begin{proof}
    Denote by $P_2(M) = \sum_{i, j \in [m], \ell \in [n]} (M_{\ell, i} - p)(M_{\ell, j} - p)$ the signed-wedge count of a matrix $M$. To bound the expectation of $P_2(\matnoiseedge)$ from below, we apply \Cref{prop:fourierexpressionthing} while conditioning latent vector of center vertex $\ell$ to be in $\Seventshort$ for $\rho = C \log(n)^{1/2}$ such that the error arising from conditioning is at most $n^{-10}$. Then, applying \Cref{prop:fourierexpressionthing} and averaging over the leading term $\leading$, we get that there is a constant $C > 0$ such that for all $d \gg \log(n)^3$,
    \begin{align*}
        \Expected{C_4(\matnoiseedge)} = q^2 \Expected{\SW{P_4}} \ge \frac{C q^2 }{d}.
    \end{align*}
    Computing the variance as before, we get the following contributions for each overlap $|V(\alpha_1) \cap V(\alpha_2)|$.
    \begin{enumerate}
        \item If $|V(\alpha_1) \cap V(\alpha_2)| = 0$, then the contribution to the variance is zero.
        \item If $|V(\alpha_1) \cap V(\alpha_2)| = 1$, then the contribution is at most $
        nm^4 q^4 \left( C\log(n)/d \right)^2$.
        \item If $|V(\alpha_1) \cap V(\alpha_2)| = 2$, then the contribution is at most $
        nm^3 q^2 \log(n) /d$.
        \item If $|V(\alpha_1) \cap V(\alpha_2)| = 3$, then the contribution is at most $
        nm^2$.
    \end{enumerate}
    Moreover, it is not hard to see that $\Varsub{M \sim \matnoiseedge}{ C_4(M) } \le nm^2$.
    Combining all this, we get that, whenever $\log(n)^3 \ll d \ll \boundwedgeslower$, then 
    $$
        nm^2\Expected{P_2(\matnoiseedge)} \gg \sqrt{\max \left\{ \Varsub{M \sim \matnoiseedge}{ P_2(M) } , \Varsub{M \sim \gaussmat}{ P_2(M) }  \right\}},
    $$
    as desired. Regarding $\matnoiseknown$, the analysis is similar with the difference that the variance is lower in terms of its dependence on $q$. In particular, the factor of $q^2$ for $|V(\alpha_1) \cap V(\alpha_2)| = 2$ turns into $q^3$, and in case $|V(\alpha_1) \cap V(\alpha_2)| = 3$, the contribution is at most $nm^2q^2$ for both $\matnoiseknown$ and $\gaussmat$. These changes then yield the lemma using otherwise the same arguments as above. 
\end{proof}

\end{document}